\documentclass{amsart}
\usepackage{mathrsfs}
 \usepackage{amsmath}
 \usepackage{graphicx}
 \usepackage{epstopdf}
 \usepackage{epsfig} %% add this line and the next only if you have pictures
 \usepackage{graphics} %% pictures should be in esp format
\usepackage{subfig}

\usepackage{amssymb}
\usepackage{amscd}
\usepackage{amsbsy}
\usepackage{enumerate}
\usepackage{graphicx}
\usepackage{calc}
\usepackage{xcolor}
\usepackage{setspace}
\usepackage[dvips]{psfrag}
\usepackage{color}
\usepackage{url}

\usepackage[pagewise]{lineno}%\linenumbers

\def\currentvolume{X}
 \def\currentissue{X}
  \def\currentyear{2023}
   \def\currentmonth{XX}

\newtheorem{theorem}{Theorem}

\newtheorem{lemma}{Lemma}
\newtheorem{proposition}{Proposition}

\theoremstyle{definition}

\title[Stability of Travelling Pulses]
 { Spectrum and stability of travelling pulses in a coupled FitzHugh--Nagumo equation }

\author[Q. Qiao, X. Zhang]{Qi Qiao$^\dag$ and Xiang Zhang$^\ddag$}

\address{
$^\dag$ School of Mathematical Sciences, Shanghai Jiao Tong University, Shanghai 200240, People's Republic of China\newline
\quad $^\ddag$ School of Mathematical Sciences, MOE--LSC,  and CMA-Shanghai, Shanghai Jiao Tong University, Shanghai 200240, People's Republic of China}
\email{qiaoqijs@163.com(Qiao); xzhang@sjtu.edu.cn(Zhang)
}

\begin{document}

\begin{abstract}
For a coupled slow--fast FitzHugh--Nagumo(FHN) equation derived from a reaction-diffusion-mechanics (RDM) model,  Holzer, Doelman and Kaper in 2013 studied existence and stability of the travelling pulse, which consists of two fast orbit arcs and two slow ones, where one fast segment passes the unique fold point with algebraic decreasing and two slow ones follow normally hyperbolic critical curve segments. Shen and Zhang in 2019 obtained existence of the travelling pulse, whose two fast orbit arcs both exponentially decrease, and one of the slow orbit arcs could be normally hyperbolic or not at the origin. Here we characterize both nonlinear and spectral stabilities of this travelling pulse.
  %{\color{red} to be expanded}
\end{abstract}

%\subjclass[2010]{35Q56; 34D25; 74J35}
%\subjclass{35C07; 35B25; 35B35; 34E15; 34F10; 35B32}
\keywords{coupled FitzHugh--Nagumo equation; singular perturbation; travelling pulse; spectrum; nonlinear stability; spectral stability.\\
    $^\ddag$Corresponding author: Xiang Zhang}

\maketitle

\hspace{0.13in}
\section{Introduction}\label{sx1}
\hspace{0.05in}

The FitzHugh--Nagumo (FHN) equation is a typical reaction--diffusion equation, which was originally proposed by FitzHugh \cite{FitzHugh} in 1961 and Nagumo et al. \cite{Nagumoetal} in 1962 as a simplification of the Hodgkin--Huxley model of a nerve axon
\begin{equation}\label{FHN}
%\left\{
\begin{array}{ll}
u_{t}=u_{xx}+u(u-a)(1-u)-w,\\
w_{t}=\epsilon(u-\gamma w).
\end{array}
%\right.
\end{equation}
with $a\in [0,1],~\gamma>0$ and $0<\epsilon\ll 1$.
Here $u(x,t)$  represents the membrane potential and $w(x,t)$ denotes the recovery variable.
{Usually, in system \eqref{FHN}, it is assumed that $\gamma$ is small enough such that it only allows a trivial steady state.}

There are many studies on the travelling pulses of FHN system, which represent the propagation of the action potential along the nerve cells.
It is well known that, if $a\in (0,1/2)$, system \eqref{FHN} has a travelling pulse solution by using a number of different techniques \cite{Carpenter,CS,CG,Hastings1976,Hastings1982,Langer}.
Hastings \cite{Hastings1976} researched that system \eqref{FHN} has two pulse solutions with different propagation speeds, which are
called the slow pulse solution and the fast pulse solution, respectively.
And Hastings \cite{Hastings1982}, by topological method, investigated the travelling pulses with oscillatory tails in the FHN system for $0<a,\epsilon\ll 1$.
Recently, by using geometric singular perturbation theory, exchange lemma and geometric blow--up techniques, Carter and Sandstede \cite{CS} also discussed the existence of travelling pulses with oscillatory tails for $0<a,\epsilon\ll 1$.

The stability of the pulse solution has been studied also by many authors \cite{CRS,Evans1972III,Evans1975,Flores,Jones1984,Yanagida}.
Such as Jones \cite{Jones1984} and Yanagida \cite{Yanagida} proved that the fast pulses are stable for $0<a<1/2$ by defining the Evans function \cite{AGJ,Evans1975}.
Carter et al. \cite{CRS} verified that the fast pulses obtained in \cite{CS} is nonlinearly stable for $a\in [0,1/2)$, where $a$ can be taken as a small parameter by applying exponential dichotomies \cite{1965book,1978book} and Lin's method \cite{BCD,Lin,Sandstede1998}.

Nash and Panfilov \cite{NP} and Panfilov et al. \cite{PKN} introduced a reaction--diffusion--mechanics (RDM) model that couples the elasticity equation, which simulated the mechanical deformation of a two--dimensional patch of myocardial fibers, to the modified FHN equations.
Recently, Holzer et al. \cite{HDK} derived the next RDM model on the real line
\begin{equation}\label{eqFHN 2013}
\begin{array}{ll}\vspace{0.12in}
U_{t}=\dfrac{1}{F}\dfrac{\partial}{\partial x}\left(\dfrac{1}{F}\dfrac{\partial U}{\partial x}\right)+k U(U-a)(1-U)-UW,\\
W_{t}= \epsilon (kU- W),
\end{array}
\end{equation}
where
\begin{equation}\label{eqFHN2013-1}
F(x,t)=\frac{1}{2}+\frac{M}{4c_{1}}+\frac{1}{2}\sqrt{\left(1+\frac{M}{2c_{1}}\right)^2-\frac{2}{c_{1}}W(x,t)}.
\end{equation}
Here $U$ represents voltage and $W$ a recovery variable. Moreover, the parameter $a$ measures the degree of excitability in the medium and $k$ is a rate constant. And the function $F(x,t)$ is used to simulate the deformation effect under some simplified conditions, where $M$ is the stress constant and $c_{1}$ is the parameter, which measures the internal energy of the deformable medium.
Then the authors established the existence of the travelling pulses for system \eqref{eqFHN 2013} by using the
geometric singular perturbation theory and  geometric blow--up techniques with $a\in (0,1/2)$. More precisely, the travelling pulse is
located in the region near a non--hyperbolic point (the maximum point). And then, they discussed the spectral stability of the travelling pulse solution by exponential dichotomies and Evans function.

Recently, Shen and Zhang \cite{SZ} also discussed the coupled FHN system
\begin{equation}\label{eqFHN}
%\left\{
\begin{array}{ll}\vspace{0.12in}
U_{t}=\dfrac{1}{F}\dfrac{\partial}{\partial x}\left(\dfrac{1}{F}\dfrac{\partial U}{\partial x}\right)+k U(U-a)(1-U)-UW,\\
W_{t}= \epsilon (U-\gamma W),
\end{array}
%\right.
\end{equation}
where $F(x,t)$ has the expression \eqref{eqFHN2013-1}, and  %=\frac{1}{2}+\frac{M}{4c_{1}}+\frac{1}{2}\sqrt{(1+\frac{M}{2c_{1}})^2-\frac{2}{c_{1}}W(x,t)}.$$
$\gamma$ is small enough such that it only allows a trivial steady state. The authors obtained existence of the travelling pulses for $a\in [0,1/2)$ by using the theories of geometric singular perturbation, for instance Fenichel's invariant manifold theorems and exchange lemma together with qualitative analysis and the center manifold theorem, as well as provided an explanation on nonexitence of travelling pulse for $a\in[1/2,1]$. It should be noted here that the travelling pulses of systems \eqref{FHN} and \eqref{eqFHN} were obtained as homoclinic orbits to the origin of the associated systems of ordinary differential equations from  \eqref{FHN} and \eqref{eqFHN} via travelling wave transformations.
Not like in \cite{HDK}, the homoclinic orbit in \cite{SZ} jumps before reaching the non--hyperbolic fold point (the maximum point) of the critical curves for $a\in[0,1/2)$. And the associated ordinary differential systems to systems \eqref{FHN} and \eqref{eqFHN} have a key difference when $a= 0$:  the trivial equilibrium of the model \eqref{FHN} overlaps with a folded point (i.e., the minimum point) of the critical curves, while
the trivial equilibrium of the model \eqref{eqFHN} overlaps with a transcritical point of the critical curves.

In this paper, our investigation will be based on the results of Shen and Zhang \cite{SZ}, and discuss the stability of the travelling pulses obtained there, by the exponential dichotomy \cite{BCD,CRS,1965book,1978book,HDK,Sandstede} and Lin's method \cite{BCD,CRS,Lin,Sandstede1998}.

We remark that one of the main techniques in this paper is the exponential dichotomy, which is stated and deeply utilized here for studying the point spectrum of the associated second order linear differential operator, which is in fact determined by the variational equation of system \eqref{eqFHN} along the travelling pulse. We must say that besides the exponential dichotomy, there are also Evans function and geometric and topological method \cite{AGJ,GJ90,GJ91,Jones1984,WZ}, Evans function and NLEP(Nonlocal Eigenvalue Problem) method \cite{DGK01,DGK02,DHV},  SLEP(Singular Limit Eigenvalue Problem) method \cite{NMIF,NS} et al.
Compared with the geometric and topological method \cite{AGJ,GJ90,GJ91,Jones1984,WZ}, the method of exponential dichotomy is analytical method and may be simpler.
The NLEP method \cite{DGK01,DGK02,DHV} will cause the nonlocal eigenvalue problem which is a hypergeometric differential equation. And similar, the SLEP method \cite{NMIF,NS}
will cause the singular limit eigenvalue problem which is not easy to solve. Thus compared with NLEP method and SLEP method, the method of exponential dichotomy may be simpler about calculation.
We remark that the Evans function and exponential dichotomy method in \cite{HDK} is not directly applicable to this article. Since, in \cite{HDK}, the back solution in some layer has an exponential decay rate as $\xi\rightarrow +\infty$ and an algebraic decay rate as $\xi\rightarrow -\infty$, which implies that the back solution does not contribute eigenvalue to the associated second order linear differential operator $\mathcal{L}_{a,\epsilon}$ in region $R_1$ seeing figure \ref{R1R2R3}.
Thus the eigenvalue of  $\mathcal{L}_{a,\epsilon}$ in region $R_1$ is unique and is trivial which is relevant with the front solution decaying with exponential rate as $\xi\rightarrow \pm\infty$.
However, in this paper, the back solution has exponential decay rates as $\xi\rightarrow \pm\infty$ which implies that the back solution contributes an eigenvalue to operator $\mathcal{L}_{a,\epsilon}$ in region $R_1$. Meanwhile, the front solution also have exponential decay rates as $\xi\rightarrow \pm\infty$. Thus it follows that there exist two eigenvalue $\lambda_0=0$ and $\lambda_1$ is region $R_1$. In order to discuss the stability of travelling pulse, we must consider the sign of $\textrm{Re}\lambda_1$ by exponential dichotomy and Lin's method.

The remaining part of this paper is organized as follows. Section \ref{sx2} introduces the existence of travelling pulses established by Shen and Zhang \cite{SZ} for $a\in[0,1/2)$.
Section \ref{sx3} states the main results on  stability of the travelling pulse, that is, the pulse is nonlinear stable for $a\in(0,1/2)$ and is spectrally stable for $a=0$.
Section \ref{sx4} is to calculate the essential spectrum of the linearization $\mathcal{L}_{a,\epsilon}$ of system \eqref{eqFHN} along the travelling pulse.
Section \ref{sx5} focuses on the point spectrum of the linearization operator, where we divide the right region of the essential spectrum  into three regions $R_{1},~R_{2}$ and $R_{3}$: The regions $R_{2}$ and $R_{3}$ do not intersect point spectrum, while the region $R_{1}$ contains at most two eigenvalues, if they exist, one is a translational eigenvalue $\lambda=0$ and the other is a negative eigenvalue.
Sections \ref{sx6} is the proofs of Theorems \ref{the.stability} and \ref{the.stability for a=0}.
The last section is an appendix, which for readers' convenience, recall some related results on exponential dichotomy and on Sturm--Liouville theorem on infinite interval.

\section{Existence of travelling pulses}\label{sx2}
\hspace{0.1in}

A \textit{travelling wave solution} of system \eqref{eqFHN} is a particular non--constant solution of the form
\begin{equation}\label{Ansatz}
U(x,t)=u(\xi),~W(x,t)=w(\xi),~\xi=x+ct,
\end{equation}
where $c$ is called wave speed. The next process is one of normal procedures for studying existence of travelling waves of slow--fast models, see e.g. \cite{HDK,SZ}.
Substituting the ansatz \eqref{Ansatz} into \eqref{eqFHN} leads to
\begin{equation}\label{eqTraWave}
\begin{split}%{ll}\vspace{0.12in}
&\frac{1}{F}\left(\frac{1}{F}u'\right)'-cu'-f(u,w)=0,\\
&w'=\frac{\epsilon}{c}(u-\gamma w),
\end{split}
\end{equation}
where $'=\frac{d}{d\xi}$ and $f(u,w)=f(u,w;a)=ku(u-a)(u-1)+uw.$ Let $v=\frac{u'}{F},$  then system \eqref{eqTraWave} can be rewritten as a system of the
first--order ODEs, namely,
\begin{equation}\label{eqfast}
\begin{split}%{ll}
u'&=F(w)v,\\
v'&=cF^2(w)v+F(w)f(u,w),\\
w'&=\frac{\epsilon}{c}(u-\gamma w).
\end{split}
\end{equation}
 System \eqref{eqfast} is called a \textit{fast system}, and in the slow scale $\widetilde{\xi}=\varepsilon \xi$, its associated \textit{slow one} reads
\begin{equation}\label{eqslow}
\begin{split}%{ll}
\epsilon \dot{u}&=F(w)v,\\
\epsilon \dot{v}&=cF^2(w)v+F(w)f(u,w),\\
\dot{w}&=\frac{1}{c}(u-\gamma w),
\end{split}
\end{equation}
where the dot means the derivative with respect to $\widetilde{\xi}$.  The \textit{layer system} is
\begin{equation}\label{eqlayer}
\begin{split}%{ll}
u'&=F(w)v,\\
v'&=cF^2(w)v+F(w)f(u,w),\\
w'&=0.
\end{split}
\end{equation}
Setting $\varepsilon=0$ in system \eqref{eqslow} results in the \textit{reduced system}
\begin{equation}\label{eqreduce}
\begin{split}%{ll}
0&=F(w)v,\\
0&=cF^2(w)v+F(w)f(u,w),\\
\dot{w}&=\frac{1}{c}(u-\gamma w),
\end{split}
\end{equation}
The critical set of system (\ref{eqreduce}) is composed of the straight line
$$L_{0}=\{(u,v,w)\in \mathbb{R_{+}}\times \mathbb{R}\times \mathbb{R_{+}}~|~u=0,~v=0 \}$$
and the parabola
$$
S_{0}=\{(u,v,w)\in \mathbb{R_{+}}\times \mathbb{R}\times \mathbb{R_{+}}~|~w=-k(u-a)(u-1),~v=0 \}.
$$
The maximum point of $S_{0}$ is $(\frac{1+a}{2},0,\frac{k}{4}(1-a)^2)$, which is a non--hyperbolic point.
Moreover, we define the right branch of $S_{0}$ as $S_{0}^{r}$ and the left branch of $S_{0}$ as $S_{0}^{l}$.
Recall that $S_{0}^{r}$ and $S_{0}^{l}$ do not contain the non--hyperbolic point $(\frac{1+a}{2},0,\frac{k}{4}(1-a)^2)$, the local maximum point.
It is clear that all points on the critical set $L_{0}\cup S_{0}$ are equilibria of the layer system (\ref{eqlayer}). When $w_0 \in \left[0,\frac{k}{4}(1-a)^2\right)$, the layer system \eqref{eqlayer} on $w=w_0$ has three equilibria $(0,0,w_0),~(U_{1}(w_0),0,w_0)$ and $(U_{2}(w_0),0,w_0)$,
where
$$U_{1,2}(w_0)=\frac{1+a}{2}\pm \frac{1}{2}\sqrt{(1-a)^2-\frac{4w_0}{k}}.$$

According to \cite{SZ}, in the $w=w_{f}=0$ plane the layer system has a heteroclinic orbit $\phi_{f}=(u_{f},v_{f})^{T}$ {in the first quadrant}, with the wave speed $c=c_{0}(a)=\frac{\sqrt{2k}}{F(0)}(\frac{1}{2}-a)$, connecting $(0,0,0)$ and $(1,0,0)$, where
$$
u_{f}(\xi)=\frac{C_fe^{F(0)\sqrt{\frac{k}{2}}\xi}}{1+C_fe^{F(0)\sqrt{\frac{k}{2}}\xi}}, \quad {~~v_{f}(\xi)=\frac{u_{f}'(\xi)}{F(0)},}
$$
with $C_f$ an integral constant. In the $w=w_{b}$ plane, where $w_{b}$ satisfies
\begin{equation}\label{eq.wb and w0 relation}
\begin{array}{ll}
(1-2a)F(w_{b})=(2U_{1}(w_{b})-U_{2}(w_{b}))F(0)
\end{array}
\end{equation}
and $0<w_{b}<\frac{k(2a^2-5a+2)}{9}<\frac{k(1-a)^2}{4}$,
system \eqref{eqlayer} on $w=w_b$ has a heteroclinic orbit $\phi_{b}=(u_{b},v_{b})^{T}$  {in the fourth quadrant}, with the wave speed $c=c_{0}(a)$, connecting $(U_{2}(w_{b}),0,w_b)$ and $(0,0,w_b)$.
Here
$$u_{b}(\xi)=\frac{U_{2}(w_{b})}{1+C_be^{F(w_{b})\sqrt{\frac{k}{2}}U_{2}(w_{b})\xi}}, \quad { ~~v_{b}(\xi)=\frac{u_{b}'(\xi)}{F(w_b)},}$$
with $C_b$ an integral constant.

From the above discussions, one can construct the singular homoclinic orbit, consisting of the above mentioned two heteroclinic fast orbits and two slow orbit arcs on $L_0$ and $S_0^r$ in between $0\le w\le w_b$, see Fig. 1. Shen and Zhang \cite{SZ} proved the next result.

\begin{lemma}\label{lem.SZ}
For $a\in [0,~1/2)$ and any fixed positive values of $(M,~c_{1},~k)$, if $\gamma>0$ then for $\epsilon>0$ sufficiently small, there is a function $c(a,\epsilon)=c_{0}(a)+O(\epsilon)$ such that system \eqref{eqTraWave} when $c=c(a,\epsilon)$ admits a homoclinic orbit to the origin, % which is in an $O(\epsilon)$ neighborhood of the singular homoclinic orbit,
as that in {red} shown in Fig.~\ref{fig.pulse}.
\end{lemma}

\begin{figure}
  \centering
  \includegraphics[width=6cm]{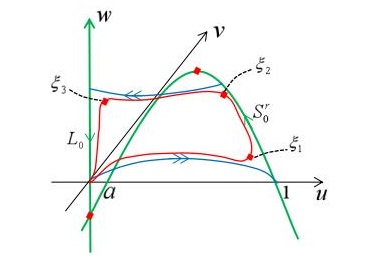}\\
  \caption{For $a\in (0,1/2)$, the green curves are given by $L_{0}$ and $S_{0}$, and the two red points on green curves are  non--hyperbolic point. Moreover, when $a=0$, the red spot on the $w$--axis coincides with the origin. The blue curves with the double arrow represent the front and back solutions of the layer system \eqref{eqlayer}, respectively, which are fast orbits. Here, the red curve is the pulse $\phi_{a,\epsilon}$, and $\xi_{1}=\Xi_{\tau}(\epsilon)$, $\xi_{2}=Z_{a,\epsilon}-\Xi_{\tau}(\epsilon)$ and ${ \xi_3}=Z_{a,\epsilon}+\Xi_{\tau}(\epsilon)$.
  }\label{fig.pulse}
\end{figure}

The homoclinic orbit in Lemma \ref{lem.SZ} of system \eqref{eqlayer} provides a travelling wave solution of system \eqref{eqFHN}, as stated in the following theorem, where there presents an approximation of the wave with the singular homoclinic orbit.

\begin{theorem}\label{the.expression}
Let $\widetilde{\phi}_{a,\epsilon}(\xi)=(u_{a,\epsilon},w_{a,\epsilon})^{T}(\xi)$ be the travelling pulse
solution derived in Lemma \ref{lem.SZ} for $a\in [0, 1/2)$ and $\epsilon>0$ sufficiently small.
For each sufficiently small $\sigma_{0},~\tau>0$, set $\Xi_{\tau}(\epsilon)=-\tau \log\epsilon$, there exist $\epsilon_{0}>0$, $C>1$, and $\xi_{0},~Z_{a,\epsilon}>0$ with $\xi_{0}$ independent of $a$ and $\epsilon$ and $1/C \leq \epsilon Z_{a.\epsilon}\leq C$, such that the next estimation holds.
\begin{itemize}

\item[$(i)$]
On $J_{f}:=(-\infty,\Xi_{\tau}(\epsilon)]$ and $J_{b}:=[Z_{a,\epsilon}-\Xi_{\tau}(\epsilon),Z_{a,\epsilon}+\Xi_{\tau}(\epsilon)]$, the pulse solution $\widetilde{\phi}_{a,\epsilon}(\xi)$ satisfies respectively
\[
\begin{split}
(u_{a,\epsilon}(\xi),w_{a,\epsilon}(\xi))^{T}&=(u_{f}(\xi)+O(\epsilon \log\epsilon),O(\epsilon \log\epsilon))^{T},\\
(u_{a,\epsilon}(\xi),w_{a,\epsilon}(\xi))^{T}&=(u_{b}(\xi)+O(\epsilon \log\epsilon),w_{b}+O(\epsilon \log\epsilon))^{T}.
\end{split}
\]

\item[$(ii)$]
On $J_{r}:=[\xi_{0},Z_{a,\epsilon}-\xi_{0}],$ $\phi_{a,\epsilon}(\xi)=(u_{a,\epsilon},v_{a,\epsilon},w_{a,\epsilon})^{T}(\xi)$ is approximated by the right slow manifold $S_{0}^{r}$
with
$$d(\phi_{a,\epsilon},S_{0}^{r})\leq \sigma_{0}.$$
\item[$(iii)$]
On $J_{l}:=[Z_{a,\epsilon}+\xi_{0},\infty),$ $\phi_{a,\epsilon}(\xi)$ is approximated by the left slow manifold $L_{0}$
with
$$d(\phi_{a,\epsilon},L_{0})\leq \sigma_{0}.$$

\end{itemize}

\end{theorem}

\begin{proof}
The proof can be found in Appendix A or obtained using the techniques in \cite[Theorem 4.5]{CRS}.
\end{proof}

\section{The main stability results}\label{sx3}

In this section, we state our main results of this paper, which are on stability of the travelling pulse solution $\widetilde{\phi}_{a,\epsilon}(\xi)$
obtained in Lemma \ref{lem.SZ} and Theorem \ref{the.expression}. By linearizing the system along the pulse solutions we study their stability via essential spectrum and point spectrum.

Without abuse using of notation, we write $(u,w)^{T}(\xi)$ as the travelling wave solution of system \eqref{eqTraWave} obtained in the last section.
Let $(U,W)^{T}(x,t)$ be a solution of system \eqref{eqFHN}, which is a perturbation of the travelling wave solution.  Utilizing the moving coordinate $\xi=x+ct$, we write this perturbed solution in
$$(U,W)^{T}(x,t)=(u(\xi)+p(\xi,t),w(\xi)+r(\xi,t))^{T},$$
with $p,~r \in C_{ub}(\mathbb{R},\mathbb{R})$, where
$$C_{ub}(\mathbb{R},\mathbb{R}):=\{u:\mathbb{R} \rightarrow\mathbb{R}|\  u~\mathrm{is~bounded~and~uniformly~continuous}\}$$

Plugging this new expression of $(U,W)^T$ in system \eqref{eqFHN} and replacing $(x,t)$ by $(\xi,t),$ one gets the system of linear partial differential equations that $(p,r)$ satisfy
\begin{equation}\label{equation.of.pr}
\begin{split}%{ll}\vspace{0.12in}
p_{t}&=\frac{1}{F}\frac{\partial}{\partial \xi}\left(\frac{1}{F}p'\right)-cp'-f_{u}(u,w)r-\frac{1}{F}\frac{\partial}{\partial \xi}\left(\frac{F_{w}u'r}{F^2}\right)-\left(\frac{F_{w}}{F^2}\frac{\partial}{\partial \xi}(\frac{u'}{F})\right)r,\\
r_{t}&=-cr'+\epsilon p-\epsilon \gamma r.
\end{split}
\end{equation}
The right--hand side of \eqref{equation.of.pr} defines a linear operator, denoted by $\mathcal{L}_{a,\epsilon}$. That is, $\mathcal{L}_{a,\epsilon}(p,r)^{T}$ is equal to the right--hand side of \eqref{equation.of.pr}.
Note that the linear operator $\mathcal{L}_{a,\epsilon}$ is partially determined by the pulse solution $(u,w)^T$. To study stability of the pulse solution, we will investigate the spectrum of this linear operator.
For this aim we will seek values of $\lambda$ for which the linearized eigenvalue problem $(\mathcal{L}_{a,\epsilon}-\lambda I)\vec{P}=0$ has a nontrivial bounded solution  $\vec{P}=(p,r)^{T}$.

The next is our first main result, which is on the spectrum of the linear operator $\mathcal{L}_{a,\epsilon}$.

\begin{theorem}\label{the.stability}
Let $\widetilde{\phi}_{a,\epsilon}(\xi)$ denote the travelling pulse solution obtained from Theorem \ref{the.expression} with the associated linear operator $\mathcal{L}_{a,\epsilon}$ for $\epsilon>0$ sufficiently small. For $a\in (0,1/2)$, there exists $b_{0}>0$ such that
the spectrum of $\mathcal{L}_{a,\epsilon}$ is contained in
$$
\{0\}\cup \{\lambda \in \mathbb{C}|~\mbox{\rm Re}\lambda\leq -b_{0}\epsilon \}.
$$
\end{theorem}

Applying Theorem \ref{the.stability} together with the results in \cite{AGJ,Evans 1972 III,Evans 1975}, we obtain easily the next conclusion, which is on nonlinear stability of the travelling pulse $\widetilde{\phi}_{a,\epsilon}(\xi)$.

\begin{theorem}\label{the.nonlinearity stable}
{For $a\in (0, 1/2)$, the} travelling pulse solution from Theorem \ref{the.expression} is nonlinearly
stable for system \eqref{eqFHN}.
\end{theorem}

By definition the travelling pulse solution $\widetilde{\phi}_{a,\epsilon}(\xi)$ is \textit{nonlinearly stable} for system \eqref{eqFHN}, if there exists a $d>0$ such that for any solution $\phi(\xi,t)$ of system \eqref{eqFHN} satisfying $\|\phi(\xi,0)-\widetilde{\phi}_{a,\epsilon}(\xi)\|\leq d$, there exists a $\xi_{0}\in \mathbb{R}$ such that
$\|\phi(\xi+\xi_{0},t)-\widetilde{\phi}_{a,\epsilon}(\xi)\|\rightarrow 0$ as $t\rightarrow \infty$, {where the norm is taken to be the $L^\infty$ norm, i.e. the supremum one}.

For the critical value $a=0$, we have the following result.
\begin{theorem}\label{the.stability for a=0}
Let $\widetilde{\phi}_{a,\epsilon}(\xi)$ denote the travelling pulse solution obtained from Theorem \ref{the.expression}. For $a=0$, the pulse solution $\widetilde{\phi}_{a,\epsilon}(\xi)$ is spectrally stable for system \eqref{eqFHN}.
\end{theorem}

By definition a travelling wave of a system is \textit{spectrally stable} if the linearization operator $\mathcal{L}$ of the system along this wave has its spectrum $\sigma(\mathcal{L})$ satisfying  $\sigma(\mathcal{L})\cap \{\lambda \in \mathbb C| ~\mbox{\rm Re}\lambda>0\}=\emptyset$, i.e., there is no a spectrum point in the open right--half part of the complex plane. Otherwise, the wave is \textit{spectrally unstable}, see e.g. Kapitula and Promislow \cite[Definition 4.1.7]{2013book}.

On Theorems \ref{the.stability} and \ref{the.stability for a=0}, we have the next remarks.
\begin{itemize}
\item Our next proofs show that the point spectra in both cases $a\in (0,1/2)$ and $a=0$ are the same. That is, each of them contains two elements, one is zero and another is negative.

\item Whereas the essential spectra of the travelling wave in the cases $a\in (0,1/2)$ and  $a=0$ are different. The former has the essential spectrum in the interior of the left half of the complex plane, which causes the nonlinear stability of the travelling pulse by using the results from \cite{AGJ,Evans1972III,Evans1975}. The latter has the essential spectrum in the left half of the complex plane with a unique point on the imaginary axis, which is the origin. So, there does not happen the Andronov-Hopf bifurcation.
    We strongly believe that in this last case the travelling pulse is also nonlinearly stable. But at the moment we cannot prove it.

\item In the proof of our main theorems, we get help of the shifted eigenvalue problem. This proof also shows that in the case $a=0$, the travelling pulse is stable in the exponential weighted space $X_{\eta}:=\{(u,v)\in C_{ub}(\mathbb{R},\mathbb{R})\times C_{ub}(\mathbb{R},\mathbb{R})~|~(e^{-\eta \xi}u,e^{-\eta \xi}v) \in C_{ub}(\mathbb{R},\mathbb{R})\times C_{ub}(\mathbb{R},\mathbb{R})\}$ with $\eta>0$ defined in Lemma \ref{lem.1}.
\end{itemize}

To prove Theorem \ref{the.stability}, we need to investigate the spectra of these waves by the exponential dichotomy of the linearized operator $\mathcal{L}_{a,\epsilon}$ of system \eqref{eqFHN} along the travelling pulse.
%The linearized eigenvalue problem $(\mathcal{L}_{a,\epsilon}-\lambda I)\vec{P}=0$ is a second order ordinary differential system {with solutions in $C_{ub}(\mathbb R,\mathbb R)\times C_{ub}(\mathbb R,\mathbb R)$}.
We further write the linearized eigenvalue problem $(\mathcal{L}_{a,\epsilon}-\lambda I)\vec{P}=0$ in a first order linear differential system as follows
\begin{eqnarray}\label{full equation}
%\left\{
\begin{split}
p'=&F(w)q,\\
q'=&F(w)(f_{u}(u,w)+\lambda)p+cF^{2}(w)q+F(w)f_{w}(u,w)r
\\
&\quad +\left(\frac{F_{w}}{F}\frac{\partial}{\partial \xi}\left(\frac{u'}{F}\right)\right)r+\frac{\partial}{\partial\xi}\left(\frac{F_{w}}{F^2}u'r\right),\\
r'=&\frac{\epsilon}{c}p-\frac{\epsilon\gamma+\lambda}{c}r,
\end{split}
%\right.
\end{eqnarray}
where
\[
\begin{split}
 &\left(\frac{F_{w}}{F}\frac{\partial}{\partial \xi}\left(\frac{u'}{F}\right)\right)r+\frac{\partial}{\partial\xi}\left(\frac{F_{w}}{F^2}u'r\right)  \\
 &   = \frac{1}{F^2}\left( 2F_{w}u''-3\frac{F_{w}^2}{F}u'w'+F_{ww}u'w'-\frac{\epsilon \gamma+\lambda}{c}F{w}u'\right)r+\frac{\epsilon F_{w}u'}{cF^{2}}p .
\end{split}
\]
The coefficient matrix of \eqref{full equation} is denoted by $A_{0}(\xi,\lambda)=A_{0}(\xi,\lambda; a,\epsilon)$. Then system \eqref{full equation} can be written in the form
\begin{equation}\label{full matrix}
\varphi'=A_{0}(\xi,\lambda; a,\epsilon) \varphi,
\end{equation}
where $\varphi=(p,q,r)^{T}.$

To study spectral stability of the travelling wave solutions of the partial differential system \eqref{eqFHN}, {one needs to prove existence of non--trivial solutions of the ordinary differential system \eqref{full matrix} satisfying $p,q,r\in C_{ub}(\mathbb{R},\mathbb{R})$.
By \cite{P1984,P1988,Sandstede} the existence of such kind of solutions for the linear nonautonomous system \eqref{full matrix} can be characterized in terms of exponential dichotomies (see Appendix).
So, spectral properties of $\mathcal{L}_{a,\epsilon}$, namely invertibility of $\mathcal{L}_{a,\epsilon}-\lambda I$ in a Banach Space $C_{ub}(\mathbb{R},\mathbb{R})\times C_{ub}(\mathbb{R},\mathbb{R})$,} can be
restated in terms of properties of exponential dichotomies of system (\ref{full matrix}). Recall from \cite{Sandstede} the following properties on spectrum of linear operators.
\begin{itemize}
  \item[$(O_1)$] $\lambda\in\mathbb C$ is in the \textit{resolvent set} of $\mathcal{L}_{a,\epsilon}$ if and only if the asymptotic matrix $A_{\infty}(\lambda):= \lim_{\xi\rightarrow \pm \infty} A(\xi,\lambda)$ is hyperbolic
      and the projections $P_{\pm}(\xi,\lambda)$ of the exponential
dichotomies of (\ref{full matrix}) on $J=\mathbb{R}_{\pm}$ satisfy $\ker(P_{-}(0,\lambda))\oplus R(P_{+}(0,\lambda))=\mathbb{C}^{n}.$ {  Here $\ker$ and $R$ denote respectively the kernel and range of a linear operator.}

  \item[$(O_2)$] $\lambda\in\mathbb C$  is in the \textit{point spectrum} if {and only if}  the projections
$P_{\pm}(\xi,\lambda)$ of the exponential dichotomies of system (\ref{full matrix}) on $J=\mathbb{R}_{\pm}$ satisfy $\ker(P_{-}(0,\lambda))\cap R(P_{+}(0,\lambda))\neq \{0\}.$

  \item[$(O_3)$] $\lambda\in\mathbb C$ is in the \textit{essential spectrum}  if   the asymptotic matrix $ A_{\infty}(\lambda)$ is not hyperbolic.

\end{itemize}
Hereafter $\mathbb R_+=[0,\infty)$, $\mathbb R_-=(-\infty,0]$. Note that $\lambda=0$ is always contained in the spectrum of $\mathcal{L}_{a,\epsilon}$ with the associated eigenfunction $\widetilde{\phi}_{a,\epsilon}'(\xi)=(u_{a,\epsilon}'(\xi),w_{a,\epsilon}'(\xi))^{T}$. { This fact follows from the properties of solutions of the variational equations of a differential system along a given solution.} %{\color{red}[Does $\lambda=0$ belong to point spectrum or essential spectrum?]}

Since our travelling pulse has the slow-fast structure, and the essential spectrum is partly determined by the asymptotic matrices, we make a remark here on the contribution of the slow and fast parts of the pulse to its point spectrum. By our next proofs, one gets that the two slow motions do not contribute eigenvalues to the associated second order linear differential operator $\mathcal{L}_{a,\epsilon}$. The contribution to the eigenvalues of $\mathcal{L}_{a,\epsilon}$ comes from the front and back solutions, which are the fast motions of the pulse near two different layers.

{To prove our main results, we analyze in the next two sections the essential and point spectra of the linear operator $\mathcal{L}_{a,\epsilon}$. We remark that the main techniques of the proof on the analysis of the spectra are from those of \cite{BCD,CRS}.}

\section{Essential Spectrum}\label{sx4}
\hspace{0.1in}

{In this section, we first prove that the essential spectrum of $\mathcal{L}_{a,\epsilon}$ is contained in the left half plane and that it has
a distance from the imaginary axis for $a\in (0,1/2)$.} When $a=0,$  the essential spectrum of $\mathcal{L}_{a,\epsilon}$ {includes the origin}.

\begin{proposition}\label{pro.ess}
The essential spectrum of $\mathcal{L}_{a,\epsilon}$ is contained in the half plane $\left\{\lambda\in\mathbb C|\ \mbox{\rm Re}(\lambda)\leq \max\{-\epsilon \gamma,-ka\}\right\}$ of the complex plane. Moreover, for all $\lambda \in \mathbb{C}$ located { on the right hand side } of the essential spectrum,
 the asymptotic matrix
$A_{\infty}(\lambda)=A_{\infty}(\lambda;a,\epsilon)$ of system \eqref{full matrix} has
precisely one $($spatial$)$ eigenvalue with positive real part.
%Finally, the essential spectrum intersects with the real axis at points
\end{proposition}
\begin{proof}
The asymptotic matrix of $A_{0}(\xi,\lambda)$ is
$$A_{\infty}(\lambda)=\left(
                        \begin{array}{ccc}
                          0 & 1+\dfrac{M}{2c_{1}} & 0 \\
                          \left(1+\dfrac{M}{2c_{1}}\right)(ka+\lambda) & c\left(1+\dfrac{M}{2c_{1}}\right)^2 & 0 \\
                          \dfrac{\epsilon}{c} & 0 & -\dfrac{\epsilon \gamma+\lambda}{c} \\
                        \end{array}
                      \right).
$$
The essential spectrum of $\mathcal{L}_{a,\epsilon}$ is given by the solutions of the algebraic equation
\begin{equation}\label{eq.lambda}
\begin{split}
0&=\det(A_{\infty}(\lambda)-il)\\
&=\left(il+\frac{\epsilon \gamma +\lambda}{c}\right)\left(l^2+c\left(1+\frac{M}{2c_{1}}\right)^2 il+(ka+\lambda)\left(1+\frac{M}{2c_{1}}\right)^2 \right)
\end{split}
\end{equation}
with $l \in \mathbb{R}$, where $i=\sqrt{-1}$. %Here for simplicity, we take $1$ as the identity matrix.
Obviously, the solutions of the equation \eqref{eq.lambda} are the straight line
$$\lambda=-icl-\epsilon \gamma, \quad ~l\in \mathbb{R}$$
and the parabola
$$\lambda=-icl-ka- {l^2}{\left(1+\frac{M}{2c_{1}}\right)^{-2}}, \quad ~l\in \mathbb{R}.$$
Thus the essential spectrum is confined to $\mbox{\rm Re}(\lambda)\leq \max\{-\epsilon \gamma,-ka\}$.

A straightforward computation shows that the eigenvalues of $A_{\infty}(\lambda)$ are
\[
\begin{split}
\mu_{1}&= -\frac{\epsilon \gamma+\lambda}{c}, \\
\mu_{2,3}&=\frac{1}{2}\left( c\left(1+\frac{M}{2c_{1}}\right)^2 \pm \sqrt{c^2\left(1+\frac{M}{2c_{1}}\right)^4+4(ka+\lambda)\left(1+\frac{M}{2c_{1}}\right)}\right).
\end{split}
\]
{By the assumption that the $\lambda$ belongs to the right hand side of the essential spectrum, one has $\mbox{\rm Re}(\epsilon \gamma+\lambda)>0$, and $\mbox{\rm Re}(\lambda+ka)>0$. This verifies that $\mbox{\rm Re}\mu_1<0$, and one of $\mu_{2,3}$ has negative real part and another one has positive real part.
This proves that for all $\lambda \in \mathbb{C}$ belonging to the right hand side of the essential spectrum, the asymptotic matrix $A_\infty(\lambda)$ has a unique eigenvalue with positive real part.}
%Obviously, for sufficiently large {\color{red}$\mbox{\rm Re}\lambda>0$??}, the asymptotic matrix $A_{\infty}(\lambda)$ has precisely one eigenvalue with positive real part. By continuity, this holds for all $\lambda \in \mathbb{C}$ to the right of the essential spectrum.
\end{proof}

According to Proposition \ref{pro.ess}, for detecting the spectral stability of the travelling wave solution $\widetilde{\phi}_{a,\epsilon}(\xi)$,
we need to study the point spectrum of the linear operator $\mathcal{L}_{a,\epsilon}$.

\section{Point spectrum}\label{sx5}
\hspace{0.1in}

In this section, we calculate the point spectrum in the right hand side of the essential spectrum for $a \in [0,1/2)$.
Firstly, we will show that the point spectrum of $\mathcal{L}_{a,\epsilon}$ for $a \in (0,1/2)$ to
the right hand side of the essential spectrum consists of at most two eigenvalues. If both exist, one is the simple eigenvalue $\lambda=0$, and the other is strictly negative.
Next, we will show that there is no an element in the point spectrum of $\mathcal{L}_{a,\epsilon}$ for $a=0$ to
the right hand side of the essential spectrum.

In order to determine location of the point spectrum, it is useful to split the complex plane in several regions.
For $\widetilde{M}\gg 1$ and $\delta\ll 1$ fixed and independent of $a\in[0,1/2)$ and $\epsilon$, we define the following three regions (see Fig. \ref{R1R2R3})
\[
\begin{array}{ll}
{R_{1}=R_{1}(\delta)}:=B(0,\delta),\\
{R_{2}=R_{2}(\delta,\widetilde{M})}:=\{\lambda \in \mathbb{C} |\  \mbox{\rm Re}(\lambda)\geq -\delta, ~\delta \leq|\lambda|\leq\widetilde{M}   \},\\
{R_{3}=R_{3}(\widetilde{M})}:=\{\lambda \in \mathbb{C} |\  |\arg(\lambda)|\leq 2\pi/3, ~|\lambda|>\widetilde{M}  \}.
\end{array}
\]
{Recall that the point spectrum of $\mathcal{L}_{a,\epsilon}$  is given by
the {values of $\lambda$ } such that the linear differential system \eqref{full equation} has an exponentially {localized} solution.}
\begin{figure}
  \centering
  \includegraphics[width=8cm]{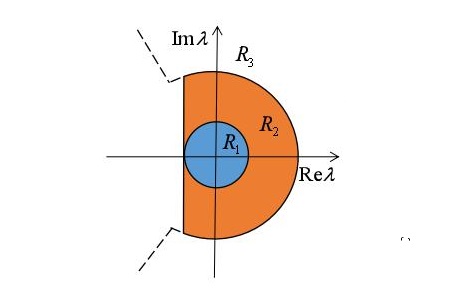}\\
  \caption{The regions $R_{1},~R_{2}$ and $R_{3}$}\label{R1R2R3}
\end{figure}

\subsection{The region $R_{3}(\widetilde{M})$}\label{subsec.R-3}
\hspace{0.1in}

We start by showing that the region $R_{3}(\widetilde{M})$ {does not intersect} the point spectrum by rescaling the eigenvalue problem \eqref{full equation}.

Since $|\lambda|>\widetilde M$ in $R_{3}(\widetilde{M})$, take the rescaling $\check{\xi}=\sqrt{|\lambda|}\xi$, $P=p,~Q=q/\sqrt{|\lambda|},~R=r,$ system \eqref{full equation} is transformed to the next one
\begin{equation}\label{eq.rescale}
\begin{split}
\frac{dP}{d\check{\xi}}&=F(w)Q,\\
\frac{dQ}{d\check{\xi}}&=\frac{\lambda}{|\lambda|}F(w)P+O(\frac{1}{\sqrt{|\lambda|}}),\\
\frac{dR}{d\check{\xi}}&=-\frac{\lambda}{c\sqrt{|\lambda|}}R+O(\frac{\epsilon}{\sqrt{|\lambda|}}).
\end{split}
\end{equation}
Set
$$B(\xi;\lambda):=\left(
                   \begin{array}{ccc}
                     0 & F(w) & 0 \\
                     \frac{\lambda}{|\lambda|}F(w) & 0 & 0 \\
                     0 & 0 & -\frac{\lambda}{c\sqrt{|\lambda|}} \\
                   \end{array}
                 \right).
$$
Obviously, its eigenvalues are
\[
\check{\mu}_{1}=-\frac{\lambda}{c\sqrt{|\lambda|}},\qquad\qquad ~{\check{\mu}_{2,3}=\pm F(w)\sqrt{\frac{\lambda}{|\lambda|}}.   }
\]
According to $|\arg(\lambda)|<\frac{2\pi}{3}$ for all $\lambda \in R_{3}$, we obtain $\mbox{\rm Re}(\sqrt{\frac{\lambda}{|\lambda|}})>\frac{1}{2}.$ Combining with the fact
$F(w)\geq \frac{1}{2}+\frac{M}{4c_{1}}$, it holds
$$
|\check{\mu}_{2,3}|>\frac{1}{4}+\frac{M}{8c_{1}}.
$$
Next,  we distinguish the two cases $8|\mbox{\rm Re}\lambda|>c\sqrt{|\lambda|}$ and $8|\mbox{\rm Re}\lambda|\leq c\sqrt{|\lambda|}$.

\noindent {\it Case }1. $8|\mbox{\rm Re}\lambda|>c\sqrt{|\lambda|}$. Then $B(\xi;\lambda)$ is hyperbolic with their spectral gap larger than $1/8$. Thus,
by Theorems \ref{th.dichotomy} and \ref{th.dichotomy perturbation} in Appendix B, and the roughness (\cite[p. 34]{1978 book}) system \eqref{eq.rescale} has an exponential dichotomy on $\mathbb{R}$ for $\lambda\in R_{3}(\widetilde{M})$. %(with a lower bound independent of $a,~\epsilon$ and $\lambda$).
Hence, system \eqref{eq.rescale} admits no nontrivial exponentially localized solutions. Consequently, $\lambda$ is not in {the intersection set of $R_{3}(\widetilde{M})$ with} the point spectrum of $\mathcal{L}_{a,\epsilon}$.

\noindent {\it Case }2.  $8|\mbox{\rm Re}\lambda|\leq c\sqrt{|\lambda|}$. By the roughness system \eqref{eq.rescale} has an exponential trichotomy on $\mathbb{R}$ with one--dimensional {center subspace}, and any bounded solution must lie entirely in the center subspace. By {continuity}, the eigenvalues of the asymptotic matrix
$\check{A}_{\infty}(\lambda;a,\epsilon):=\lim_{\xi\rightarrow\pm\infty} \check{A}(\xi,\lambda;a,\epsilon) $ are separated {in the following way: one,  saying $\check{\mu}_{\epsilon}$, has the absolute value of its real part less than $1/8+\kappa$ for some small $\kappa>0$  and the other two have the absolute values of their real parts larger than $\frac{1}{4}+\frac{M}{8c_{1}}-\kappa$. Let $\beta$ be the eigenvector of $\check{A}_{\infty}(\lambda;a,\epsilon)$ associated with $\check{\mu}_{\epsilon}$,} then any solution $\check{\varphi}$ in the center subspace satisfies $\lim_{\xi\rightarrow\pm\infty}\check{\varphi}e^{-\check{\mu}_{\epsilon}\check{\xi}}=b_{\pm}\beta$ for some $b_{\pm}\in\mathbb{C}\setminus \{0\}$ (see \cite{Levinson}, Theorem 1).  Hence, system \eqref{eq.rescale} admits no exponentially localized solutions, and consequently $\lambda \in R_{3}(\widetilde{M})$ is not in the point spectrum of $\mathcal{L}_{a,\epsilon}$.

\subsection{The region $R_{1}(\delta)\cup R_{2}(\delta,\widetilde{M})$}
\hspace{0.1in}

In this subsection, we introduce a weight $\eta>0$ and consider the shifted system
\begin{equation}\label{eq.shift}
\varphi'=A(\xi,\lambda)\varphi,
\end{equation}
where
\begin{eqnarray*}\label{eq.shift matrx}
\aligned
A(\xi,\lambda)&= A(\xi,\lambda;a,\epsilon)=A_{0}(\xi,\lambda;a,\epsilon)-\eta I\\
 &=\left(
                                            \begin{array}{ccc}
                                              -\eta & F(w) & 0 \\
                                              F(w)(f_{u}+\lambda)+\dfrac{\epsilon F_{w}u'}{cF^2} & cF^{2}(w)-\eta & \Delta_{a,\epsilon} \\
                                              \dfrac{\epsilon}{c} & 0 & -\dfrac{\lambda+\epsilon\gamma}{c}-\eta \\
                                            \end{array}
                                          \right),
\endaligned
\end{eqnarray*}
and
$$\Delta_{a,\epsilon} =F(w)f_{w}(u,w)+\frac{1}{F^{2}(w)}\left(2F_{w}u''-\frac{3F_{w}^2u'w'}{F}+F_{ww}u'w'-\frac{\lambda+\epsilon\gamma}{c}F_{w}u'\right),$$
and the functions $u=u_{a,\epsilon},~w=w_{a,\epsilon}$ for simplification to notations.
We remark that the introduction of the weight $\eta$ is for shifting the eigenvalues of the matrix $A_{0}(\xi,\lambda;a,\epsilon)$ to the left. {Recall that $A_0$ is defined in \eqref{full matrix}.}

\subsubsection{ The shifted eigenvalue problem}
\hspace{0.1in}

{To study the eigenvalues of $A(\xi,\lambda)$, we have the next result, which is for a modification of  $A(\xi,\lambda)$ with general $u,\ w$. }

\begin{lemma}\label{lem.1}
Let $k_{1},~\sigma_{0}>0$ be small and define
\[
\begin{split}
{\mathcal{U}(\sigma_{0},k_{1})}:= & \left\{ (a,u,w)\in \mathbb{R}^3|\ ~a\in \left[0,\frac{1}{2}-k_{1}\right],~u\in [0,\sigma_{0}],~w\in [0,w_{b}+\sigma_{0}]\right\}\\
&\bigcup \left\{ (a,u,w)\in \mathbb{R}^3|\ ~a\in \left[0,\frac{1}{2}-k_{1}\right],~u\in [U_{2}(w_{b})-\sigma_{0},1+\sigma_{0}],\right.\\
   & \qquad\qquad \left. w\in [-k(u-a)(u-1)-\sigma_{0},-k(u-a)(u-1)+\sigma_{0}] \right\}.
\end{split}
\]
Set $F_{m}:= \frac{1}{2}+\frac{M}{4c_{1}}$ and  $\eta:=\frac{\sqrt{2k}}{4}F_{m}k_{1}$. For $\sigma_{0},~\delta>0$ sufficiently small, there exists $\epsilon_{0}>0$ and
$\mu\in(0, \eta]$ such that the matrix
$$\hat{A}(u,w,\lambda,a,\epsilon)=\left(
                                            \begin{array}{ccc}
                                              -\eta & F(w) & 0 \\
                                              F(w)(f_{u}(u,w)+\lambda)+\dfrac{\epsilon F_{w}u'}{cF^2} & cF^{2}(w)-\eta & \Delta \\
                                              \dfrac{\epsilon}{c} & 0 & -\dfrac{\lambda+\epsilon\gamma}{c}-\eta \\
                                            \end{array}
                                          \right),
$$
where
$$\Delta=F(w)f_{w}(u,w)+\frac{1}{F^{2}(w)}\left(2F_{w}u''-\frac{3F_{w}^2u'w'}{F}+F_{ww}u'w'-\frac{\lambda+\epsilon\gamma}{c}F_{w}u'\right),$$
admits a uniform spectral gap larger than $\mu>0$  for $(a,u,w)\in \mathcal{U}(\sigma_{0},k_{1}),~\lambda\in R_{1}(\delta)\cup R_{2}(\delta,\widetilde{M})$, and $\epsilon \in [0,\epsilon_{0}]$.
Moreover, the matrix $\hat{A}$ has precisely one eigenvalue with positive real part.
\end{lemma}
\begin{proof}
 The matrix $\hat{A}(u,w,\lambda,a,\epsilon)$ is nonhyperbolic if and only if
\begin{eqnarray*}
\aligned
0&= \det(\hat{A}(u,w,\lambda,a,\epsilon)-il{ I})\\
 &=\left(\eta+il+\frac{\lambda+\epsilon\gamma}{c}\right) \left((\eta+il)(cF^2-\eta-il)+F^{2}(f_{u}+\lambda)+\frac{\epsilon F_{w}u'}{cF} \right)+\frac{\epsilon F}{c}\Delta
\endaligned
\end{eqnarray*}
is satisfied for some $l\in\mathbb{R}.$ Here $I$ is the identity operator. In what follows we also use $1$ to represent the identity operator. %{\color{red}[Definition?]}
For $\epsilon=0$, $\hat{A}(u,w,\lambda,a,0)$ is nonhyperbolic if and only if $\lambda$ is located on the line $\ell_s$
$$\lambda=-c_{0}\eta-ic_{0}l$$
{ or } on the parabola $\mathcal P_s$
\begin{equation}\label{eq.parabola}
\lambda=-f_{u}+\frac{1}{F^2}(\eta+2\eta il-c_{0}\eta F^2-l^2-c_{0}F^2 il).
\end{equation}
For $u\in [0,\sigma_{0}],~w\in [0,w_{b}+\sigma_{0}]$, it holds
$$
f_{u}(u,w)\geq f_{u}(u,0)\geq f_{u}(\sigma_{0},0)=k(3\sigma_{0}^{2}-2(a+1)\sigma_{0}+a),$$
then
$$-f_{u}(u,w)\leq 3k\sigma_{0}.$$
For $u\in [U_{2}(w_{b})-\sigma_{0},1+\sigma_{0}],~ w\in [-k(u-a)(u-1)-\sigma_{0},-k(u-a)(u-1)+\sigma_{0}]$, it holds
\begin{eqnarray*}
\begin{split}
f_{u}(u,w)&\geq k\left(3u^2-2(a+1)u+a \right)-k(u-a)(u-1)-\sigma_{0}\\
          &=k\left(2u^2-(a+1)u \right)-\sigma_{0}\\
          &\geq k \left(2\left(\frac{a+1}{2}\right)^2-(a+1)\frac{a+1}{2} \right)-\sigma_{0}=-\sigma_{0},
\end{split}
\end{eqnarray*}
then
$$-f_{u}(u,w)\leq \sigma_{0}.$$
Thus, we obtain $-f_{u}(u,w)\leq (3k+1)\sigma_{0}$ for any $(a,u,w)\in {\mathcal{U}(\sigma_{0},k_{1})}$.
Take $\eta=\frac{\sqrt{2k}}{4}F_{m}k_{1}$ and $(3k+1)\sigma_{0}<\frac{1}{8}kk_{1}^2$, where $F_{m}\triangleq \frac{1}{2}+\frac{M}{4c_{1}}=\frac{1}{2}F(0)$, it holds
$$c_{0}\eta=\frac{kk_{1}}{4}(\frac{1}{2}-a)\geq\frac{kk^2_{1}}{4}.$$
By the expression \eqref{eq.parabola} it holds
\begin{eqnarray*}
\aligned
\mbox{\rm Re}(\lambda)&\leq-f_{u}(u,w)+\eta^2\frac{1}{F^2(w)}-c_{0}\eta \leq-f_{u}(u,w)+\eta^2\frac{1}{F^2_{m}}-c_{0}\eta\\
          &\leq-(3k+1)\sigma_{0}- \frac{1}{4}kk_{1}^2\leq - \frac{1}{8}kk_{1}^2,
\endaligned
\end{eqnarray*}
Thus, the union of the line $\ell_s$ and the parabola $\mathcal P_s$ lies in the half plane
$$\mbox{\rm Re}(\lambda)\leq -\frac{1}{8}kk_{1}^2$$
for any $(a,u,w)\in \mathcal{U}(\sigma_{0},k_{1})$.
Hence, provided $\delta>0$ is sufficiently small, the union of $\ell_s$ and $\mathcal P_s$ does not intersect the compact set $R_{1}\cup R_{2}$ for any $(a,u,w)$ {belonging to}
the compact set $\mathcal{U}(\sigma_{0},k_{1})$, namely, when $\lambda \in R_{1}\cup R_{2}$, $\hat{A}(u,w,\lambda,a,0)$ is hyperbolic for any
$(a,u,w)\in \mathcal{U}(\sigma_{0},k_{1})$. By continuity we conclude that there exists $\epsilon_{0}>0$ such
that the matrix $\hat{A}(u,w,\lambda,a,\epsilon)$  has, for $(a,u,w)\in \mathcal{U}(\sigma_{0},k_{1}),~\lambda \in R_{1}\cup R_{2}$ and
$\epsilon \in [0,\epsilon_{0}]$, a uniform spectral gap larger than some $\mu>0$. Since $-\eta$ is in the
spectrum of $\hat{A}(0,0,0,a,0)$, it forces $\mu\leq \eta$.

Moreover, $\hat{A}(u,w,\lambda,a,0)$ has precisely one eigenvalue with positive real part for sufficiently large $\lambda>0.$  Therefore, by continuity, we obtain that
$\hat{A}(u,w,\lambda,a,0)$ has also precisely one eigenvalue with positive real part for
$\lambda \in \mathbb{C}$ lying in the right hand side of $\mbox{\rm Re}(\lambda)\leq -\frac{1}{8}kk_{1}^2$.
So, $\hat{A}(u,w,\lambda,a,\epsilon)$ has precisely one eigenvalue with positive real part for $(a,u,w)\in \mathcal{U}(\sigma_{0},k_{1}),~\lambda\in R_{1}\cup R_{2},~\epsilon \in [0,\epsilon_{0}]$.  This proves the proposition.
\end{proof}

To characterize the point spectrum, we fix
$$\eta=\frac{\sqrt{2k}}{4}F_{m}k_{1},$$
and take
\begin{equation}\label{eq.nu bound}
\nu\geq \max \left\{\sqrt{\frac{2}{k}}\frac{2}{F(0)},~\sqrt{\frac{2}{k}}\frac{2}{F(w_{b})U_{2}(w_{b})},\frac{2}{\mu}   \right\}.
\end{equation}

\begin{proposition}\label{pro.1}
For $a\in (0,1/2)$, $\lambda\in R_{1}\cup R_{2}$ is a point spectrum of $\mathcal{L}_{a,\epsilon}$  if and only if it is an eigenvalue of the shifted eigenvalue problem \eqref{eq.shift}. {Correspondingly, associated to the $\lambda$ the shifted eigenvalue problem \eqref{eq.shift} has a bounded solution.}
\end{proposition}
\begin{proof}
The relation of asymptotic matrices $A_{\infty}(\lambda,a,\epsilon)$ and $\hat{A}(0,0,\lambda,a,\epsilon)$ is
$$\sigma(\hat{A}(0,0,\lambda,a,\epsilon))=\sigma(A_{\infty}(\lambda,a,\epsilon))-\eta.$$
Moreover, when $a\in(0,1/2)$, for $\lambda\in R_{1}\cup R_{2}$, $A_{\infty}(\lambda,a,\epsilon)$ and $\hat{A}(0,0,\lambda,a,\epsilon)$ have precisely one eigenvalue with positive real part by Proposition \ref{pro.ess} and Lemma \ref{lem.1}.
Thus system \eqref{full matrix} admits a nontrivial exponentially localized
solution $\varphi(\xi)$ if and only if system \eqref{eq.shift} admits the one given by $e^{-\eta\xi}\varphi(\xi)$ for $\lambda\in R_{1}\cup R_{2}$ with $a\in(0,1/2)$.
\end{proof}

Let $\Omega$ be the right region of the essential spectrum of $\mathcal{L}_{a,\epsilon}$ for $a=0$. We have the next equivalent result.

\begin{proposition}\label{pro.1.a=0}
For $a=0$, $\lambda\in \left(R_{1}\cup R_{2}\right)\cap { \Omega} $ is a point spectrum of $\mathcal{L}_{a,\epsilon}$  if and only if it is an eigenvalue of the shifted eigenvalue problem \eqref{eq.shift}.
\end{proposition}
\begin{proof}
When $a=0$, for $\lambda\in (R_{1}\cup R_{2})\cap { \Omega}$, the matrices $A_{\infty}(\lambda,0,\epsilon)$ and $\hat{A}(0,0,\lambda,0,\epsilon)$ both have precisely one eigenvalue with positive real part by Proposition \ref{pro.ess} and Lemma \ref{lem.1}.
Thus system \eqref{full matrix} admits a nontrivial exponentially localized
solution $\varphi(\xi)$ if and only if system \eqref{eq.shift} admits the one given by $e^{-\eta\xi}\varphi(\xi)$ for $\lambda\in R_{1}\cup R_{2}\cap { \Omega}$ with $a=0$.
\end{proof}

In order to prove Theorem \ref{the.stability for a=0}, we just need to prove that the point spectrum of the operator $\mathcal{L}_{0,\epsilon}$ does not intersect the region $\Omega_{+}:=\Omega\cap \{\lambda:~\mbox{\rm Re}(\lambda)>0\}$. {According to these last results, it is only necessary to} show that the shifted eigenvalue problem \eqref{eq.shift} does not have intersection of the point spectrum with the region $\Omega_{+}$ with $a=0$.

\subsubsection{Exponential dichotomies along the right and left slow manifolds }\label{subsec.5.2.2}
\hspace{0.1in}

In this subsection, we will prove that system \eqref{eq.shift} has exponential dichotomies on the
intervals $I_{r}=[L_{\epsilon},Z_{a,\epsilon}-L_{\epsilon}]$ and $I_{l}=[Z_{a,\epsilon}+L_{\epsilon},\infty)$, where $L_{\epsilon}=-\nu \log \epsilon$, and
$Z_{a,\epsilon}$ is as in Theorem \ref{the.expression}, for $\lambda \in R_{1}\cup R_{2}$.
By Lemma \ref{lem.1}, the matrix $A(\xi,\lambda)$ is pointwise hyperbolic for $\xi \in I_{r}\cup I_{l}$ and  has slowly varying
coefficients. According to Theorem \ref{th.dichotomy}, the shifted eigenvalue system admits exponential dichotomies.
The main results are the following.

\begin{proposition}\label{pro.2}
For the shifted eigenvalue system \eqref{eq.shift}, the following statements hold.
\begin{itemize}
\item There exists an $\epsilon_{0}>0$ such that for $0<\epsilon<\epsilon_{0}$, system \eqref{eq.shift} admits exponential dichotomies on the intervals $I_{r}=[L_{\epsilon},Z_{a,\epsilon}-L_{\epsilon}]$ and $I_{l}=[Z_{a,\epsilon}+L_{\epsilon},\infty)$ with an exponential decay rate $\mu>0$, as that in Lemma \ref{lem.1}.
\item The projections $\mathcal{Q}_{r,l}^{u,s}(\xi,\lambda)=\mathcal{Q}_{r,l}^{u,s}(\xi,\lambda;a,\epsilon)$  associated to the exponential dichotomies are analytic
in $\lambda \in R_{1}\cup R_{2}$ and are approximated by $\mathcal P$ at the endpoints $L_{\epsilon},~Z_{a,\epsilon}$, in the following way
$$\|(\mathcal{Q}_{r}^{s}-\mathcal{P})(L_{\epsilon},\lambda)\|,
\|(\mathcal{Q}_{r}^{s}-\mathcal{P})(Z_{a,\epsilon}-L_{\epsilon},\lambda)\|,
\|(\mathcal{Q}_{l}^{s}-\mathcal{P})(Z_{a,\epsilon}+L_{\epsilon},\lambda)\|
\leq C\epsilon |\log\epsilon|$$
where $\mathcal{P}(\xi,\lambda)=\mathcal{P}(\xi,\lambda;a,\epsilon)$ is
the spectral projections onto the stable eigenspace of the coefficient matrix $A(\xi,\lambda)$
of system \eqref{eq.shift}, and $C>0$ is a constant independent of $\lambda,a$ and $\epsilon$. % {\color{red}[What are the $\mathcal{Q}_{r,l}^{u,s}(\xi,\lambda;a,\epsilon)$? where they project to?]}
\end{itemize}

\end{proposition}
\begin{proof}
The proof is similar to Proposition 6.5 in \cite{CRS} or Proposition 5.2 in \cite{BCD}.
\end{proof}

\subsection{The Region $R_{1}(\delta)$}
\hspace{0.13in}

\subsubsection{The reduced eigenvalue problem}\label{section.Reduced Eigenvalue Problem 1}
\hspace{0.13in}

The spectra of the reduced problems along the front and the back will be of critical importance  for the full system.
In this section, we construct a reduced eigenvalue problem by setting $\epsilon$ to $0$ in
system \eqref{eq.shift} for $\xi$ in $I_{f}$ or $I_{b}$. Note that since $\lambda\in R_{1}$ is sufficiently small, the reduced eigenvalue problem, i.e. \eqref{eq.shift}$|_{\lambda=\epsilon=0}$, does not depend on $\lambda$. The reduced eigenvalue problem reads
\begin{equation}\label{eq.refuced eigenvalue}
\varphi'=A_{j}(\xi)\varphi, \qquad ~j=f,~b,
\end{equation}
where
\begin{equation}\label{eq.j=f,b matrx}
\aligned
A_{j}(\xi)= A_{j}(\xi;a)
 =\left(
                                            \begin{array}{ccc}
                                              -\eta & F(w_{j}) & 0 \\
                                              F(w_{j})f_{u}(u_{j}(\xi),w_{j})& c_{0}F^{2}(w_{j})-\eta & \Delta_{1,j} \\
                                              0 & 0 & -\eta \\
                                            \end{array}
                                          \right)
\endaligned
\end{equation}
and
$$\Delta_{1,j}=F(w_{j})f_{w}(u_{j},w_{j})+2\frac{F_{w}(w_{j})}{F^{2}(w_{j})}u_{j}''.$$
Here $u_{j}(\xi)$ denotes the $u$--component of $\phi_{j}$ and $a\in[0,\frac{1}{2}-k_{1}]$, { with $k_1>0$ small given in Lemma \ref{lem.1}}.

For $\xi \in I_{f}=(-\infty,L_{a,\epsilon}]$, the shifted eigenvalue system \eqref{eq.shift} can be written as the perturbed one
\begin{equation*}\label{eq.f's perturbation }
\varphi'=(A_{f}(\xi)+B_{f}(\xi,\lambda))\varphi,
\end{equation*}
where
\begin{eqnarray*}\label{eq.matrix Bf}
\aligned
B_{f}(\xi,\lambda)=B_{f}(\xi,\lambda;a,\epsilon)
 =\left(
                                            \begin{array}{ccc}
                                              0 & F(w_{a,\epsilon})-F(0)  & 0 \\
                                             \widetilde{\Delta}_{f} & c_{0}F^{2}(w_{j})-\eta &  \Delta_{a,\epsilon}-\Delta_{1,f} \\
                                              \dfrac{\epsilon}{c} & 0 & -\dfrac{\lambda+\epsilon\gamma}{c}\\
                                            \end{array}
                                          \right),
\endaligned
\end{eqnarray*}
and
$$\widetilde{\Delta}_{f}=\left.\left(F(w)(f_{u}+\lambda)+\frac{\epsilon F_{w}u'}{cF^2}\right)\right|_{(u,w)=(u_{a,\epsilon},w_{a,\epsilon})}-F(0)f_{u}(u_{f},0).$$

For $\xi \in [-L_{a,\epsilon},L_{a,\epsilon}]$ , system \eqref{eq.shift} can be written as the perturbed one
\begin{equation}\label{eq.b's perturbation }
\varphi'=A(\xi+Z_{a,\epsilon},\lambda)\varphi=(A_{b}(\xi)+B_{b}(\xi,\lambda))\varphi,
\end{equation}
where
$$
B_{b}(\xi,\lambda)=B_{b}(\xi,\lambda;a,\epsilon):=A(\xi+Z_{a,\epsilon},\lambda)-A_{b}(\xi).
$$
{ The upper triangular block structure of the coefficient matrix \eqref{eq.j=f,b matrx} for the reduced eigenvalue problem \eqref{eq.refuced eigenvalue} induces existence of the two--dimensional invariant subspace $\mathbb{C}^2\times \{0\}\subset \mathbb{C}^3$ for system \eqref{eq.refuced eigenvalue}, i.e. the $w=0$ plane}, and the dynamics
of system \eqref{eq.refuced eigenvalue} on this invariant space is given by
\begin{eqnarray}\label{eq.fast subsystem's eigenvalue}
\aligned
\psi'=C_{j}(\xi)\psi,~~j=f,~b,
\endaligned
\end{eqnarray}
with
\begin{eqnarray*}
\aligned
C_{j}(\xi)=\left(
                                              \begin{array}{ccc}
                                                -\eta & F(w_{j}) \\
                                                F(w_{j})f_{u}(u_{j},w_{j}) & c_{0}F^{2}(w_{j})-\eta \\
                                              \end{array}
                                              \right).
\endaligned
\end{eqnarray*}
Obviously, system \eqref{eq.fast subsystem's eigenvalue} admits a one--dimensional {  invariant subspace formed by its bounded solutions, which are} spanned by
\begin{equation*}\label{eq.psi j's solution}
\psi_{j}(\xi)=\psi_{j}(\xi;a):=e^{-\eta\xi}\phi'_{j}(\xi),~j=f,b.
\end{equation*}
Then the adjoint system of system \eqref{eq.fast subsystem's eigenvalue}
\begin{equation*}\label{eq.adjoint system }
\psi'=-C_{j}^{*}(\xi)\psi,~j=f,b,
\end{equation*}
with $C_j^*$ being the transpose of conjugate of $C_j$, also has a one--dimensional {  invariant space formed by its  bounded solutions, which are} spanned by
\begin{equation*}\label{eq.adjoint solution}
\psi_{j,ad}(\xi)=\psi_{j,ad}(\xi;a):=e^{(\eta-c_{0}F^2(w_{j}))\xi}\left(
                                                                   \begin{array}{c}
                                                                    v_{j}'(\xi)\\
                                                                    -u_{j}'(\xi)
                                                                   \end{array}
                                                                  \right)
,~j=f,b.
\end{equation*}
{Recall that $v_j$'s are defined in the formulae on the lines above and under \eqref{eq.wb and w0 relation}.}
Note that the inner product of $\psi_{j}$ and $\psi_{j,ad}$ vanishes, i.e. $\langle \psi_{j},\psi_{j,ad}\rangle=0.$  Next, we  construct exponential dichotomies for subsystem \eqref{eq.fast subsystem's eigenvalue}
on  the both half--lines $\mathbb R_+$ and $\mathbb R_-$.

\begin{proposition}\label{pro.3}
Taking $k_{1}>0$ small. For each $a\in [0,\frac{1}{2}-k_{1}]$,
\begin{itemize}
\item system \eqref{eq.fast subsystem's eigenvalue} admits exponential dichotomies on the both half--lines $\mathbb{R}_{\pm}$, %    with the same decay rate $\mu>0$
     \end{itemize}
with the $a$--independent decay rate $ \mu>0$ and the coefficient $C>0$, % {\color{red}[without appearing $C,\mu$??]}
and the
projections $\Pi_{j,\pm}^{u,s}(\xi)=\Pi_{j,\pm}^{u,s}(\xi;a),~j=f,b,$ which satisfy
\begin{equation}\label{eq.subprojections }
\begin{split}
R(\Pi_{j,+}^{s}(0))&=\mbox{\rm Span}(\psi_{j}(0))=R(\Pi_{j,-}^{u}(0)),\\
R(\Pi_{j,+}^{u}(0))&=\mbox{\rm Span}(\psi_{j,ad}(0))=R(\Pi_{j,-}^{s}(0)),
\end{split}\quad ~j=f,~b.
\end{equation}
Here $\mu,C>0$ are the quantities given in the definition of exponential dichotomy as those in Appendix B.
\end{proposition}
\begin{proof}
Let $C_{j,\pm\infty}:=\lim_{\xi\rightarrow\pm\infty}C_{j}(\xi)$ be the asymptotic matrices, with $j=f,~b$.
{According to Lemma \ref{lem.1}, the spectra of $C_{f,-\infty}$ and $C_{f,+\infty}$ are contained in the spectra of
$\hat{A}(0,0,0,a,0)$ and $\hat{A}(1,0,0,a,0)$, respectively, namely,

\centerline{$\sigma(C_{f,-\infty})\subset \sigma(\hat{A}(0,0,0,a,0))$ and
$\sigma(C_{f,+\infty})\subset \sigma(\hat{A}(1,0,0,a,0))$.}
\noindent In the same way, it follows

\centerline{$\sigma(C_{b,-\infty})\subset \sigma(\hat{A}(U_{2}(w_{b}),w_{b},0,a,0))$ and
$\sigma(C_{b,+\infty})\subset \sigma(\hat{A}(0,w_{b},0,a,0))$.}
\noindent Moreover, the fact that $\hat{A}(u,w,0,a,0)$ has a uniform spectral gap larger than $\mu>0$ for
$(a,u,w) \in \mathcal{U}(\sigma_{0},k_{1})$ again by Lemma \ref{lem.1} is inherited by the asymptotic matrices $C_{j,\pm\infty},~j=f,~b$, namely the asymptotic matrices $C_{j,\pm\infty},~j=f,~b$  have a uniform spectral gap larger than $\mu>0$.
By Theorems \ref{th.dichotomy} and \ref{th.dichotomy intergral}, it holds that system \eqref{eq.fast subsystem's eigenvalue} admits exponential dichotomies on the both half--lines with the constants $C,~\mu>0$ {(see the definition of exponential dichotomy in Appendix B)} and the projections as in \eqref{eq.subprojections }.}
Since the interval $[0,\frac{1}{2}-k_{1}]$ is compact, we can choose the constant $C>0$ independent of $a$.
\end{proof}

Focus on system \eqref{eq.refuced eigenvalue} again, and observe that
\[ %begin{equation}\label{eq.omega expression }
\begin{array}{ll}
\omega_{j}(\xi):=\left(\begin{array}{cc}
                        \psi_{j}(\xi) \\
                        0 \\
                      \end{array}
                 \right)=\left(\begin{array}{cc}
                        e^{-\eta\xi}\phi'_{j}(\xi) \\
                        0 \\
                      \end{array}
                 \right),~j=f,~b,
\end{array}
\] %end{equation}
is a bounded solution to \eqref{eq.refuced eigenvalue}. In addition, using the variation of constants formulas, the
exponential dichotomies of the subsystem \eqref{eq.fast subsystem's eigenvalue} can be extended to the system \eqref{eq.refuced eigenvalue}.

\begin{proposition}\label{pro.4}
Taking $k_{1}>0$ small. For each $a\in [0,\frac{1}{2}-k_{1}]$,
\begin{itemize}
\item system \eqref{eq.refuced eigenvalue} admits exponential dichotomies on the both half--lines $\mathbb{R}_{\pm}$,
 \end{itemize}
with $a$--independent constants $C,~\mu>0$ and the
projections $Q_{j,\pm}^{u,s}(\xi)=Q_{j,\pm}^{u,s}(\xi;a),~j=f,b,$ which satisfy
\begin{equation}\label{eq.Q projections}
\begin{split}
Q_{j,+}^{s}(\xi)&= \left(
                          \begin{array}{cc}
                            \Pi_{j,+}^{s}(\xi) & \int^{\xi}_{\infty}e^{\eta(\xi-\hat{\xi})}\Phi_{j,+}^{u}(\xi,\hat{\xi})F_{j}d\hat{\xi} \\
                            0 & 1 \\
                          \end{array}
                   \right)   =1-Q_{j,+}^{u}(\xi),~\xi\geq 0,\\
Q_{j,-}^{s}(\xi)&= \left(
                          \begin{array}{cc}
                            \Pi_{j,-}^{s}(\xi) & \int^{\xi}_{0}e^{\eta(\xi-\hat{\xi})}\Phi_{j,-}^{u}(\xi,\hat{\xi})F_{j}d\hat{\xi} \\
                            0 & 1 \\
                          \end{array}
                   \right)   =1-Q_{j,-}^{u}(\xi), ~\xi\leq 0.
\end{split}
\end{equation}
Here
$$F_{j}=\left(\begin{array}{c}
          0 \\
          \Delta_{1,j}
        \end{array}
        \right),~j=f,~b,
$$
and $\Phi_{j,+}^{u,s}(\xi,\hat{\xi})=\Phi_{j,+}^{u,s}(\xi,\hat{\xi};a)$  denotes the $($un$)$stable evolution of subsystem
\eqref{eq.fast subsystem's eigenvalue} under the exponential dichotomies established in Proposition \ref{pro.3}.
Moreover,  the projections satisfy
\begin{equation}\label{eq.projections space}
\begin{split}
{\rm R}(Q_{j,+}^{u}(0))&=\mbox{\rm Span}(\varphi_{1,j}),\ \ ~ ~{\rm R}(Q_{j,+}^{s}(0))=\mbox{\rm Span}(\omega_{j}(0),\varphi_{2}),\\
{\rm R}(Q_{j,-}^{u}(0))&=\mbox{\rm Span}(\omega_{j}(0)),~{\rm R}(Q_{j,-}^{s}(0))=\mbox{\rm Span}(\varphi_{1,j},\varphi_{2}),
\end{split} \quad \ \  \ ~j=f,~b.
\end{equation}
Here
$$\varphi_{1,j}=\varphi_{i,j}(a):=\left(\begin{array}{c}
                                    \psi_{j,ad}(0) \\
                                    0
                                  \end{array}
                                  \right),\ \ \ \ ~\varphi_{2}:=\left(\begin{array}{c}
                                    0 \\
                                    0  \\
                                    1
                                  \end{array}
                                  \right).
$$
\end{proposition}
\begin{proof}
Denote by $\Phi_{j}(\xi,\hat{\xi})$ the evolution of subsystem \eqref{eq.fast subsystem's eigenvalue} with its second column belonging to $R(\Pi_{j,+}^{s}(\xi))$.
By the variation of constants formula, the evolution $T_{j}(\xi,\hat{\xi})=T_{j}(\xi,\hat{\xi};a)$ of system \eqref{eq.refuced eigenvalue} is given by
\begin{equation*}\label{eq.T evolution}
T_{j}(\xi,\hat{\xi})= \left(
                          \begin{array}{cc}
                            \Phi_{j}(\xi,\hat{\xi}) & \int^{\xi}_{\hat{\xi}}\Phi_{j}(\xi,z) F_{j}e^{-\eta(z-\hat{\xi})}dz \\
                            0 & e^{-\eta(\xi-\hat{\xi})} \\
                          \end{array}
                   \right), ~j=f,~b.\\
\end{equation*}
{Some calculations yield
$$\Pi_{j,+}^{s}(\xi)\Phi_{j}(\xi,z) F_{j}=\Phi_{j}(\xi,z) F_{j},~\xi\geq z,~j=f,b$$
and
$$\Pi_{j,-}^{s}(\xi)\Phi_{j}(\xi,z) F_{j}=0,~\xi\leq z,~ j=f,b.$$
And then the projections defined in \eqref{eq.Q projections} yield exponential dichotomies on the both half--lines for \eqref{eq.refuced eigenvalue} with the constants $C,~\mu>0$, where $C$ is independent of $a$.}
\end{proof}

\subsubsection{Along the front}
\hspace{0.13in}

{From Proposition \ref{pro.4},  system \eqref{eq.refuced eigenvalue} admits an exponential dichotomy on  $(-\infty,0]$, then, by  the variation of constants formula, the solutions of system \eqref{eq.shift} can be expressed on interval $(-\infty,0]$. It holds that the exponentially decaying solution to \eqref{eq.shift} in the backward time admits an exit condition at $\xi=0$.}
{Here {\it exit condition} is the one under which the next constructed exponentially decaying solution \eqref{eq.varphi f-} to (17) leaves the neighborhood of the critical curve in the negative time. The time when the orbit negatively leaves the critical curve  is called {\it exit time}.}

Next, we will establish the entry and exit conditions for existence of  the solutions to system \eqref{eq.shift}
on $[0,Z_{a,\epsilon}]$ and for existence of  exponentially decaying solutions to \eqref{eq.shift} in the forward time on
$[Z_{a,\epsilon},\infty)$.
{Then, equating these exit and entry conditions at
$\xi=0$ and $\xi=Z_{a,\epsilon}$, we obtain a matching equations whose solutions indicate that system \eqref{eq.shift} admits an exponentially localized solution.}

\begin{proposition}\label{pro.5}
Let $T_{f,-}^{u,s}(\xi,\hat{\xi})=T_{f,-}^{u,s}(\xi,\hat{\xi};a)$  be the $($un$)$stable evolution of system \eqref{eq.refuced eigenvalue} under the exponential dichotomy on $I_{f}=(-\infty,0]$ established in Proposition \ref{pro.4} and the
associated projections are given by  $Q_{f,-}^{u,s}(\xi)=Q_{f,-}^{u,s}(\xi;a)$. The following statements hold.
\begin{itemize}
\item[$(i)$]
 There exist $\delta,~\epsilon_{0}>0$,  such that any solution $\varphi_{f,-}(\xi,\lambda)$ to \eqref{eq.shift} decaying exponentially in the backward time  for $\lambda \in R_{1}(\delta)$ and $\epsilon\in(0,\epsilon_{0})$ satisfies
\begin{equation}\label{eq.varphi f-}
\begin{split}
&\varphi_{f,-}(0,\lambda)=\beta_{f,-}\omega_{f}(0)+\beta_{f,-}\int_{-\infty}^{0}T_{f,-}^{s}(0,\hat{\xi})B_{f}(\hat{\xi},\lambda)\omega_{f}(\hat{\xi})d\hat{\xi}
                          +\mathcal{H}_{f,-}(\beta_{f,-}),\\
&Q_{f,-}^{u}(0)\varphi_{f,-}(0,\lambda)=\beta_{f,-}\omega_{f}(0),
\end{split}
\end{equation}
for some $\beta_{f,-}\in\mathbb{C}$, where $\mathcal{H}_{f,-}$ is a linear map satisfying the bound condition
$$\|\mathcal{H}_{f,-}(\beta_{f,-})\|\leq C(\epsilon |\log\epsilon|+|\lambda|)^{2}|\beta_{f,-}|,$$
with $C>0$ independent of $\lambda,~a$ and $\epsilon$. Moreover, $\varphi_{f,-}(\xi,\lambda)$ is analytic in $\lambda$.

\item[$(ii)$]
The derivative $\phi'_{a,\epsilon}:=  (u'_{a,\epsilon},v'_{a,\epsilon},w'_{a,\epsilon})^{T}$ of the pulse solution satisfies
\begin{equation}\label{eq.pulse solution}
\begin{array}{ll}
Q_{f,-}^{s}(0)\phi'_{a,\epsilon}(0)=\int_{-\infty}^{0}T_{f,-}^{s}(0,\hat{\xi})B_{f}(\hat{\xi},0)e^{-\eta\hat{\xi}}\phi'_{a,\epsilon}(\hat{\xi})d\hat{\xi}.
\end{array}
\end{equation}
\end{itemize}
\end{proposition}

\begin{proof}
For $(i)$, take $0<\hat{\mu}<\mu$. {Denote by $C_{\hat{\mu}}(I_{f,-},\mathbb{C}^3)$ the space of $\hat{\mu}$--exponentially decaying, and continuous functions defined on $I_{f,-}$ with range in $\mathbb{C}^3$, and its associated norm is}
$$\|\varphi\|_{\hat{\mu}}=\sup_{\xi\leq 0}\|\varphi(\xi)\|e^{\hat{\mu}|\xi|}.$$
By Theorem \ref{the.expression} $(i)$, the perturbed matrix $B_{f}(\xi,\lambda;a,\epsilon)$ has the bound estimation
\begin{equation}\label{eq.29}
\begin{array}{ll}
\|B_{f}(\xi,\lambda;a,\epsilon)\|\leq C(\epsilon|\log\epsilon|+|\lambda|),\ \ \ ~\xi\in I_{f,-}.
\end{array}
\end{equation}
Taking $\beta\in \mathbb{C},~\lambda\in R_{1}(\delta)$, one can check that the function
$$\mathcal{G}_{\beta,\lambda}:~C_{\hat{\mu}}(I_{f,-},\mathbb{C}^3) \rightarrow C_{\hat{\mu}}(I_{f,-},\mathbb{C}^3)$$ with
\begin{equation*}\label{eq.30}
\begin{split}%{ll}
\mathcal{G}_{\beta,\lambda}(\varphi)(\xi)=\beta \omega_{f}(\xi) &+\int_{0}^{\xi}T_{f,-}^{u}(\xi,\hat{\xi})B_{f}(\hat{\xi},\lambda)\varphi(\hat{\xi})d\hat{\xi}\\
&\quad +\int_{-\infty}^{\xi}T_{f,-}^{s}(\xi,\hat{\xi})B_{f}(\hat{\xi},\lambda)\varphi(\hat{\xi})d\hat{\xi},
\end{split}
\end{equation*}
is well--defined, and is a contraction mapping for each $\delta,~\epsilon>0$ sufficiently small (with the contraction constant independent of $\beta$ and $a$). By the Banach Contraction Theorem, the mapping $\mathcal{G}_{\beta,\lambda}$ has a unique fixed point $\varphi_{f,-}$ in $ C_{\hat{\mu}}(I_{f,-},\mathbb{C}^3)$, i.e.
\begin{equation}\label{eq.31}
\begin{array}{ll}
\varphi_{f,-}=\mathcal{G}_{\beta,\lambda}(\varphi_{f,-}),\ \ \ {~\xi\in I_{f,-}} .
\end{array}
\end{equation}
Since the perturbed matrix $B_{f}(\xi,\lambda;a,\epsilon)$ is analytic in $\lambda$, then $\varphi_{f,-}(\xi,\lambda)$ is analytic in $\lambda$. Moreover, $\varphi_{f,-}$ is linear in $\beta$ by construction, and we derive
$$\|\varphi_{f,-}\|\leq \|\beta\omega_{f}\|+\frac{2}{\mu}\|B_{f}\|\|\varphi_{f,-}\|.$$
Combining this with \eqref{eq.29} yields
\begin{equation}\label{eq.32}
\begin{array}{ll}
\|\varphi_{f,-}(\xi,\lambda)-\beta\omega_{f}(\xi)\|\leq C|\beta|(\epsilon|\log\epsilon|+|\lambda|),\ \ \ ~\xi\in I_{f,-}.
\end{array}
\end{equation}

The family of fixed points to Eqs.\eqref{eq.31} parameterized by $\beta\in \mathbb{C}$ form
a one--dimensional space, which consists of exponentially decaying solutions as $\xi\rightarrow -\infty$ to \eqref{eq.shift}.
From Lemma \ref{lem.1}, the asymptotic matrix $\hat{A}(0,0,\lambda,a,\epsilon)$ of system \eqref{eq.shift} has exactly one eigenvalue with positive real part. {Thus, the dimension of the space in which the solutions to \eqref{eq.shift} decay exponentially in the backward time is one.
This proves that there exists some $\beta\in \mathbb{C}$ such that any solution $\varphi_{f,-}(\xi,\lambda)$ to \eqref{eq.shift} that converges to $0$ as $\xi\rightarrow -\infty$ satisfies \eqref{eq.31}.}
Using \eqref{eq.29} and \eqref{eq.32} arrives
\begin{equation*}
\begin{split}
\varphi_{f,-}(0)&=\beta_{f,-}\omega_{f}(0)+\int_{-\infty}^{0}T_{f,-}^{s}(0,\hat{\xi})B_{f}(\hat{\xi},\lambda)\varphi_{f,-}(\hat{\xi})d\hat{\xi}\\
 &=\beta_{f,-}\omega_{f}(0)+\beta_{f,-}\int_{-\infty}^{0}T_{f,-}^{s}(0,\hat{\xi})B_{f}(\hat{\xi},\lambda)\omega_{f}(\hat{\xi})d\hat{\xi}+\mathcal{H}_{f,-}(\beta_{f,-}),\\
Q_{f,-}^{u}(0)&\varphi_{f,-}(0,\lambda)=\beta_{f,-}\omega_{f}(0),
\end{split}
\end{equation*}
where
$$\|\mathcal{H}_{f,-}(\beta_{f,-})\|\leq C(\epsilon |\log\epsilon|+|\lambda|)^{2}|\beta_{f,-}|.$$

For $(ii)$, since $e^{-\eta\xi}\phi'_{a,\epsilon}$ is an eigenfunction of system \eqref{eq.shift} at $\lambda=0$, it follows that $e^{-\eta\xi}\phi'_{a,\epsilon}$ satisfies equation \eqref{eq.31} at $\lambda=0$ for some $\beta \in \mathbb{C}$. Set $\xi=0$,  we obtain
$$Q_{f,-}^{s}(0)\phi'_{a,\epsilon}(0)=\int_{-\infty}^{0}T_{f,-}^{s}(0,\hat{\xi})B_{f}(\hat{\xi},0)e^{-\eta\hat{\xi}}\phi'_{a,\epsilon}(\hat{\xi})d\hat{\xi}.$$
This proves the proposition.
\end{proof}

\subsubsection{Passage near the right slow manifold}
\hspace{0.13in}

This part focuses on the expressions of the solution of system \eqref{eq.shift} at the end of the orbit along the right critical manifold, and of the derivative of the pulse solution.

\begin{proposition}\label{pro.6}
Let $T_{j,\pm}^{u,s}(\xi,\hat{\xi})=T_{j,\pm}^{u,s}(\xi,\hat{\xi};a)$  be the $($un$)$stable evolution of system \eqref{eq.refuced eigenvalue} under the exponential dichotomy  established in Proposition \ref{pro.4}, and let the
associated projections be $Q_{j,\pm}^{u,s}(\xi)=Q_{j,\pm}^{u,s}(\xi;a)$, $j=f,~b$. The following statements hold.
\begin{itemize}
\item[$(i)$]
 There exist $\delta,~\epsilon_{0}>0$,  such that any solution $\varphi^{sl}(\xi,\lambda)$ to \eqref{eq.shift}  for $\lambda \in R_{1}(\delta)$ and $\epsilon\in(0,\epsilon_{0})$ satisfies
\begin{equation}\label{eq.varphi sl f}
\begin{split}
\varphi^{sl}(0,\lambda)&=\beta_{f}\omega_{f}(0)+\zeta_{f}Q_{f,+}^{s}(0)\varphi_{2}\\
&\qquad +\beta_{f}\int_{L_{\epsilon}}^{0}T_{f,+}^{u}(0,\hat{\xi})B_{f}(\hat{\xi},\lambda)\omega_{f}(\hat{\xi})d\hat{\xi}
                          +\mathcal{H}_{f}(\beta_{f},\zeta_{f},\beta_{b}),\\
Q_{f,-}^{u}(0)&\varphi^{sl}(0,\lambda)=\beta_{f}\omega_{f}(0),
\end{split}
\end{equation}
and
\begin{equation}\label{eq.varphi sl b}
\begin{split}
\varphi^{sl}(Z_{a,\epsilon},\lambda)&=\beta_{b}\omega_{b}(0)+\beta_{b}\int_{-L_{\epsilon}}^{0}T_{b,-}^{s}(0,\hat{\xi})B_{b}(\hat{\xi},\lambda)\omega_{b}(\hat{\xi})d\hat{\xi}
                          +\mathcal{H}_{b}(\beta_{f},\zeta_{f},\beta_{b}),\\
Q_{b,-}^{u}(0)\varphi^{sl}&(Z_{a,\epsilon},\lambda)=\beta_{b}\omega_{b}(0),
\end{split}
\end{equation}
for some $\beta_{f},\zeta_{f},\beta_{b}\in\mathbb{C}$, where $\mathcal{H}_{f}$ and $\mathcal{H}_{b}$ are linear maps satisfying the estimations
\[
\begin{split}
\|\mathcal{H}_{f}(\beta_{f},\zeta_{f},\beta_{b})\|&\leq C\left((\epsilon |\log\epsilon|+|\lambda|)|\zeta_{f}|
                                                     +(\epsilon |\log\epsilon|+|\lambda|)^{2}|\beta_{f}|+e^{-q/\epsilon}|\beta_{b}|\right),\\
\|\mathcal{H}_{b}(\beta_{f},\zeta_{f},\beta_{b})\|&\leq C\left((\epsilon |\log\epsilon|+|\lambda|)^{2}|\beta_{b}|+e^{-q/\epsilon}(|\beta_{f}|+|\zeta_{f}|)\right),
\end{split}
\]
with $q,~C>0$ independent of $\lambda,~a$ and $\epsilon$. Moreover, $\varphi^{sl}(\xi,\lambda)$ is analytic in $\lambda$.

\item[$(ii)$]
The derivative $\phi'_{a,\epsilon}$ of the pulse solution satisfies
\begin{equation}\label{eq.pulse solution}
\begin{split}
Q_{f,+}^{u}(0)\phi'_{a,\epsilon}(0)=&T_{f,+}^{u}(0,L_{\epsilon})e^{-\eta L_{\epsilon}}\phi'_{a,\epsilon}(L_{\epsilon})     \\
       &\qquad +\int_{L_{\epsilon}}^{0}T_{f,+}^{u}(0,\hat{\xi})B_{f}(\hat{\xi},0)e^{-\eta\hat{\xi}}\phi'_{a,\epsilon}(\hat{\xi})d\hat{\xi}.\\
Q_{b,-}^{s}(0)\phi'_{a,\epsilon}(Z_{a,\epsilon})=&T_{b,-}^{s}(0,-L_{\epsilon})e^{\eta L_{\epsilon}}\phi'_{a,\epsilon}(Z_{a,\epsilon}-L_{\epsilon})\\
       &\qquad +\int_{-L_{\epsilon}}^{0}T_{b,-}^{s}(0,\hat{\xi})B_{b}(\hat{\xi},0)e^{-\eta\hat{\xi}}\phi'_{a,\epsilon}(Z_{a,\epsilon}+\hat{\xi})d\hat{\xi}.\\
\end{split}
\end{equation}
\end{itemize}
\end{proposition}
\begin{proof}
$(i)$.
Since the front $\phi_{f}(\xi)$ is a heteroclinic orbit connecting the equilibria  $(0,0)$ and $(1,0)$ of the layer system on $w=0$, and converges { to them at the exponential rate $\sqrt{\frac{k}{2}}F(0)$} as $\xi\rightarrow\pm\infty$,  the coefficient matrix $A_{f}(\xi)$ of \eqref{eq.refuced eigenvalue} converges at the exponential rate $\sqrt{\frac{k}{2}}F(0)$ to some asymptotic matrix $A_{f,\infty}$ as $\xi\rightarrow\infty$.
{Hence, by Lemma 3.4 of \cite{P1984} and its proof, the associated  projections $Q_{f,+}^{u,s}$ of the exponential dichotomy to system \eqref{eq.refuced eigenvalue} admit}
\begin{equation}\label{eq.35}
\begin{array}{ll}
\|Q_{f,+}^{u,s}(\xi)-P_{f}^{u,s}\|\leq C\left(e^{-\sqrt{\frac{k}{2}}F(0)\xi}+e^{-\mu\xi}\right),\quad ~\xi\geq 0.
\end{array}
\end{equation}
Here $P_{f}^{u,s}=P_{f}^{u,s}(a)$ is the spectral projection on the (un)stable eigenspace of the asymptotic matrix $A_{f,\infty}$.
At the endpoint $L_{\epsilon}=-\nu\log(\epsilon)$, the coefficient matrix $A(\xi,\lambda)$ of \eqref{eq.shift} satisfies
$$\|A(L_{\epsilon},\lambda)-A_{f,\infty}\|\leq C(\epsilon|\log\epsilon|+|\lambda|),$$
where we have used the fact that the coefficient matrix $A_{f}(\xi)$ of \eqref{eq.refuced eigenvalue} converges at the exponential rate $\sqrt{\frac{k}{2}}F(0)$ to the asymptotic matrix $A_{f,\infty}$ as $\xi\rightarrow\infty$ and $\nu\geq\sqrt{\frac{2}{k}}\frac{2}{F(0)}$.
{Therefore, the spectral projections associated with the matrices $A(L_{\epsilon},\lambda)$
and $A_{f,\infty}$ admit the same bound by continuity, namely,}
$$\|\mathcal{Q}_{r}^{u,s}(L_{\epsilon},\lambda)-P_{f}^{u,s}\|\leq C(\epsilon|\log\epsilon|+|\lambda|).$$
Combining this with \eqref{eq.35}, we obtain
\begin{equation}\label{eq.36}
\begin{array}{ll}
\|\mathcal{Q}_{r}^{u,s}(L_{\epsilon},\lambda)-Q_{f,+}^{u,s}(L_{\epsilon})\|\leq C(\epsilon|\log\epsilon|+|\lambda|).
\end{array}
\end{equation}
In a similar way, at $\xi=Z_{a,\epsilon}-L_{\epsilon}$, we obtain
\begin{equation*}\label{eq.37}
\begin{array}{ll}
\|\mathcal{Q}_{r}^{u,s}(Z_{a,\epsilon}-L_{\epsilon},\lambda)-Q_{b,-}^{u,s}(-L_{\epsilon})\|\leq C(\epsilon|\log\epsilon|+|\lambda|).
\end{array}
\end{equation*}

{By the variation of constants formula, any solution $\varphi_{f}^{sl}(\xi,\lambda)$, on $I_{f,+}$, to the shifted eigenvalue system \eqref{eq.shift} must satisfy}
\begin{equation}\label{eq.38}
\begin{split}
\varphi_{f}^{sl}(\xi,\lambda)=& T_{f,+}^{u}(\xi,L_{\epsilon})\alpha_{f}+\beta_{f}\omega_{f}(\xi)+\zeta_{f}T_{f,+}^{s}(\xi,0)\varphi_{2}\\
&+\int _{0}^{\xi}T_{f,+}^{s}(\xi,\hat{\xi})B_{f}(\hat{\xi},\lambda)\varphi_{f}^{sl}(\hat{\xi},\lambda)d\hat{\xi}+\int _{L_{\epsilon}}^{\xi}T_{f,+}^{u}(\xi,\hat{\xi})B_{f}(\hat{\xi},\lambda)\varphi_{f}^{sl}(\hat{\xi},\lambda)d\hat{\xi}
\end{split}
\end{equation}
for some  $\alpha_{f},~\zeta_{f}\in \mathbb{C}$ and $\alpha_{f} \in {\rm R}(Q_{f,+}^{u}(L_{\epsilon}))$.
By Theorem \ref{the.expression}$(i)$, the perturbed matrix $B_{f}(\xi,\lambda;a,\epsilon)$ has
\begin{equation}\label{eq.38.1}
\begin{array}{ll}
\|B_{f}(\xi,\lambda;a,\epsilon)\|\leq C(\epsilon|\log\epsilon|+|\lambda|),~\xi\in I_{f,+}.
\end{array}
\end{equation}
Then, by the contraction mapping principle, equation \eqref{eq.38} has a unique solution $\varphi_{f}^{sl}$ for all sufficiently small $|\lambda|,~\epsilon>0$. {Note that $\varphi_{f}^{sl}$ is linear in $(\alpha_{f},~\beta_{f},~\zeta_{f})$. By estimate \eqref{eq.38.1} yields}
\begin{equation}\label{eq.39}
\begin{array}{ll}
\sup\limits_{\xi\in[0,L_{\epsilon}]}\|\varphi_{f}^{sl}(\xi,\lambda)\|\leq C(\|\alpha_{f}\|+|\beta_{f}|+|\zeta_{f}|),
\end{array}
\end{equation}
taking $\delta,~\epsilon_{0}>0$ smaller if necessary.

Denote by $\mathcal{T}_{r}^{u,s}(\xi,\hat{\xi},\lambda)=\mathcal{T}_{r}^{u,s}(\xi,\hat{\xi},\lambda;a,\epsilon)$ the (un)stable evolution of system \eqref{eq.shift} under the exponential dichotomy on $I_{r}$ established in Proposition \ref{pro.2}. Then any solution $\varphi_{r}(\xi,\lambda)$ to \eqref{eq.shift} on $I_{r}$ is
\begin{equation}\label{eq.40}
\begin{array}{ll}
\varphi_{r}(\xi,\lambda)=\mathcal{T}_{r}^{u}(\xi,Z_{a,\epsilon}-L_{\epsilon},\lambda)\alpha_{r}+\mathcal{T}_{r}^{s}(\xi,L_{\epsilon},\lambda)\beta_{r},
\end{array}
\end{equation}
where $\alpha_{r}\in {\rm R}(\mathcal{Q}_{r}^{u}(Z_{a,\epsilon}-L_{\epsilon},\lambda))$ and $\beta_{r}\in {\rm R}(\mathcal{Q}_{r}^{s}(L_{\epsilon},\lambda))$.
Applying the projection $\mathcal{Q}_{r}^{u}(L_{\epsilon},\lambda)$ to $\varphi_{r}(L_{\epsilon},\lambda)-\varphi_{f}^{sl}(L_{\epsilon},\lambda)$, we obtain the matching condition
\begin{equation}\label{eq.41ab}
\begin{split}
\alpha_{f}&=\mathcal{H}_{1}(\alpha_{f},\beta_{f},\alpha_{r}),\\
\|\mathcal{H}_{1}(\alpha_{f},\beta_{f},\alpha_{r})\|& \leq C[(\epsilon|\log\epsilon|+|\lambda|)(\|\alpha_{f}\|+|\beta_{f}|+|\zeta_{f}|)+e^{-\frac{q}{\epsilon}}||\alpha_{r}|| ],
\end{split}
\end{equation}
by using \eqref{eq.36}, \eqref{eq.38.1}, \eqref{eq.39}, \eqref{eq.nu bound} and $Z_{a,\epsilon}=O(\epsilon^{-1})$.
Similarly, applying the projection $\mathcal{Q}_{r}^{s}(L_{\epsilon},\lambda)$ to $\varphi_{r}(L_{\epsilon},\lambda)-\varphi_{f}^{sl}(L_{\epsilon},\lambda)$ yields
\begin{equation}\label{eq.42}
\begin{split}
\beta_{f}& =\mathcal{H}_{2}(\alpha_{f},\beta_{f},\zeta_{f}),\\
\|\mathcal{H}_{2}(\alpha_{f},\beta_{f},\zeta_{f})\|& \leq C(\epsilon|\log\epsilon|+|\lambda|)(\|\alpha_{f}\|+|\beta_{f}|+|\zeta_{f}|).
\end{split}
\end{equation}

Consider the translated version \eqref{eq.b's perturbation } of system \eqref{eq.shift}. Then any solution $\varphi_{b}^{sl}(\xi,\lambda)$ to \eqref{eq.b's perturbation } on $[-L_{\epsilon},0]$ must satisfy
\begin{equation}\label{eq.43}
\begin{split}
\varphi_{b}^{sl}(\xi,\lambda)=& T_{b,-}^{s}(\xi,-L_{\epsilon})\alpha_{b}+\beta_{b}\omega_{b}(\xi)+\int _{0}^{\xi}T_{b,-}^{u}(\xi,\hat{\xi})B_{b}(\hat{\xi},\lambda)\varphi_{b}^{sl}(\hat{\xi},\lambda)d\hat{\xi}\\
&+\int _{-L_{\epsilon}}^{\xi}T_{b,-}^{s}(\xi,\hat{\xi})B_{b}(\hat{\xi},\lambda)\varphi_{b}^{sl}(\hat{\xi},\lambda)d\hat{\xi}
\end{split}
\end{equation}
for some  $\beta_{b}\in \mathbb{C}$ and $\alpha_{b} \in R(Q_{b,-}^{s}(-L_{\epsilon}))$ by the variation of constants formula.  According to Theorem \ref{the.expression}$(i)$ we obtain
\begin{equation}\label{eq.44}
\begin{array}{ll}
\|B_{b}(\xi,\lambda;a,\epsilon)\|\leq C(\epsilon|\log\epsilon|+|\lambda|),~~\xi
 \in I_{b,-}=[-L_{\epsilon},0].
\end{array}
\end{equation}
Then the contraction mapping principle verifies that there exists a unique solution
$\varphi_{b}^{sl}$ to equation \eqref{eq.43} for all sufficiently small $|\lambda|,~\epsilon>0$.
Note that $\varphi_{b}^{sl}$ is linear in $(\alpha_{b},~\beta_{b})$ and satisfies the estimation
\begin{equation}\label{eq.45}
\begin{array}{ll}
\sup\limits_{\xi\in[-L_{\epsilon},0]}\|\varphi_{b}^{sl}(\xi,\lambda)\|\leq C(\|\alpha_{b}\|+|\beta_{b}|)
\end{array}
\end{equation}
via \eqref{eq.44}, taking $\delta,~\epsilon_{0}>0$ smaller if necessary.
Similarly, applying the projections $\mathcal{Q}_{r}^{u}(Z_{a,\epsilon}-L_{\epsilon},\lambda)$ and $\mathcal{Q}_{r}^{s}(Z_{a,\epsilon}-L_{\epsilon},\lambda)$ to $\varphi_{r}(Z_{a,\epsilon}-L_{\epsilon},\lambda)-\varphi_{b}^{sl}(Z_{a,\epsilon}-L_{\epsilon},\lambda)$  respectively produce the matching conditions
\begin{equation}\label{eq.46}
\begin{split}
\alpha_{r}&=\mathcal{H}_{3}(\alpha_{b},\beta_{b}),\\
\|\mathcal{H}_{3}(\alpha_{b},\beta_{b})\| &\leq C(\epsilon|\log\epsilon|+|\lambda|)(\|\alpha_{b}\|+|\beta_{b}|),
\end{split}
\end{equation}
\begin{equation}\label{eq.47}
\begin{split}%{ll}
\alpha_{b}&=\mathcal{H}_{4}(\alpha_{b},\beta_{b},\beta_{r}),\\
\|\mathcal{H}_{4}(\alpha_{b},\beta_{b},\beta_{r})\| &\leq C[(\epsilon|\log\epsilon|+|\lambda|)(\|\alpha_{b}\|+|\beta_{b}|)+e^{-\frac{q}{\epsilon}}\|\beta_{r}\|],
\end{split}
\end{equation}
where $\mathcal{H}_{3}$ and $\mathcal{H}_{4}$ are linear maps in their variables.

Next, we will combine these last results about the solution on $[0,Z_{a,\epsilon}]$ to obtain the
relevant conditions satisfied at $\xi=0$ and $\xi=Z_{a,\epsilon}$. Combining \eqref{eq.42} and \eqref{eq.47} yields
\begin{equation*}\label{eq.48}
\begin{split}
\alpha_{b}&=\mathcal{H}_{5}(\alpha_{b},\beta_{b},\alpha_{f},\beta_{f},\zeta_{f}),\\
\|\mathcal{H}_{5}(\alpha_{b},\beta_{b},\alpha_{f},\beta_{f},\zeta_{f})\|
&\leq C\left((\epsilon|\log\epsilon|+|\lambda|)(\|\alpha_{b}\|+|\beta_{b}|)+e^{-\frac{q}{\epsilon}}(\|\alpha_{f}\|+|\beta_{f}|+|\zeta_{f}|)\right),
\end{split}
\end{equation*}
which induce that for all sufficiently small $|\lambda|,~\epsilon>0$
\begin{equation}\label{eq.49}
\begin{split}
\alpha_{b}&=\alpha_{b}(\alpha_{f},\beta_{b},\beta_{f},\zeta_{f}),\\
\|\alpha_{b}(\alpha_{f},\beta_{b},\beta_{f},\zeta_{f})\|
&\leq C\left((\epsilon|\log\epsilon|+|\lambda|)|\beta_{b}|+e^{-\frac{q}{\epsilon}}(\|\alpha_{f}\|+|\beta_{f}|+|\zeta_{f}|)\right).
\end{split}
\end{equation}
By \eqref{eq.41ab}, \eqref{eq.46} and \eqref{eq.49}, we obtain a linear map $\mathcal{H}_{6}$  satisfying
\begin{equation*}\label{eq.50}
\begin{split}
\alpha_{f}&=\mathcal{H}_{6}(\alpha_{f},\beta_{b},\beta_{f},\zeta_{f}),\\
\|\mathcal{H}_{6}(\alpha_{f},\beta_{b},\beta_{f},\zeta_{f})\|
&\leq C\left((\epsilon|\log\epsilon|+|\lambda|)(\|\alpha_{f}\|+|\beta_{f}|+|\zeta_{f}|)+e^{-\frac{q}{\epsilon}}|\beta_{b}|\right),
\end{split}
\end{equation*}
which further imply that for all sufficiently small $|\lambda|,~\epsilon>0$
\begin{equation}\label{eq.51}
\begin{split}
\alpha_{f}&=\alpha_{f}(\beta_{f},\zeta_{f},\beta_{b}),\\
\|\alpha_{f}(\beta_{f},\zeta_{f},\beta_{b})\|
&\leq C\left(\epsilon|\log\epsilon|+|\lambda|)(|\beta_{f}|+|\zeta_{f}|)+e^{-\frac{q}{\epsilon}}|\beta_{b}|\right).
\end{split}
\end{equation}
Substituting \eqref{eq.51} into \eqref{eq.38} at $\xi=0$ gives
\begin{equation*}
\begin{split}
\varphi^{sl}_{f}(0,\lambda)=&\beta_{f}\omega_{f}(0)+\zeta_{f}Q_{f,+}^{s}(0)\varphi_{2}\\
                          &+
                            \beta_{f}\int_{L_{\epsilon}}^{0}T_{f,+}^{u}(0,\hat{\xi})B_{f}(\hat{\xi},\lambda)\omega_{f}(\hat{\xi})d\hat{\xi}+\mathcal{H}_{f}(\beta_{f},\zeta_{f},\beta_{b}),
\end{split}
\end{equation*}
where we have used \eqref{eq.projections space}, \eqref{eq.38.1} and \eqref{eq.39}. Applying the projection $Q_{f,-}^{u}(0)$ and \eqref{eq.projections space} to the above equation, one gets
\begin{equation*}
\begin{array}{ll}
Q_{f,-}^{u}(0)\varphi^{sl}_{f}(0,\lambda)=\beta_{f}\omega_{f}(0).
\end{array}
\end{equation*}
Thus, any solution $\varphi^{sl}$ to system \eqref{eq.shift} satisfies the entry condition \eqref{eq.varphi sl f}. Similarity, substituting \eqref{eq.51} into \eqref{eq.49} at $\xi=0$ gives
\begin{equation*}
\begin{split}
\alpha_{b}&=\alpha_{b}(\alpha_{f}(\beta_{f},\zeta_{f},\beta_{b}),\beta_{f},\zeta_{f},\beta_{b}),\\
\|\alpha_{b}(\alpha_{f}(\beta_{f},\zeta_{f},\beta_{b}),\beta_{f},\zeta_{f},\beta_{b})\|
&\leq C\left((\epsilon|\log\epsilon|+|\lambda|)|\beta_{b}|+e^{-\frac{q}{\epsilon}}(|\beta_{f}|+|\zeta_{f}|)\right).
\end{split}
\end{equation*}
Substituting the above expression of $\alpha_{b}$ and its associated estimation into \eqref{eq.43} at $\xi=0$, together with { \eqref{eq.projections space}}, \eqref{eq.44} and \eqref{eq.45}, verifies
\begin{equation*}
%\begin{array}{ll}
\varphi^{sl}_{b}(0,\lambda)=\beta_{b}\omega_{b}(0)+\beta_{b}\int_{-L_{\epsilon}}^{0}T_{b,-}^{s}(0,\hat{\xi})B_{b}(\hat{\xi},\lambda)\omega_{b}(\hat{\xi})d\hat{\xi}
                          +\mathcal{H}_{b}(\beta_{f},\zeta_{f},\beta_{b})
%\end{array}
\end{equation*}
 Applying the projection $Q_{b,-}^{u}(0)$ and \eqref{eq.projections space} to the above equation shows
\begin{equation*}
%\begin{split}
Q_{b,-}^{u}(0)\varphi^{sl}_{b}(0,\lambda)=\beta_{b}\omega_{b}(0).
%\end{split}
\end{equation*}
Thus, any solution $\varphi^{sl}$ to system \eqref{eq.shift} satisfies the condition \eqref{eq.varphi sl b}. {Since all quantities, the perturbed matrices $B_{j}(\xi,\lambda),~j=f,~b$, the evolution $\mathcal{T}(\xi,\hat{\xi},\lambda)$  of system \eqref{eq.shift} and the projections $\mathcal{Q}^{u,s}_{r}(\xi,\lambda)$ associated with the exponential dichotomy of \eqref{eq.shift}, involved in the above proofs depend analytically on $\lambda$, $\varphi^{sl}(\xi,\lambda)$ is analytic in $\lambda$.} The statement $(i)$ follows.

\noindent $(ii)$. Note that $e^{-\eta\xi}\phi'_{a,\epsilon}(\xi)$ is an eigenfunction of system \eqref{eq.shift} at $\lambda=0$. Thus there exist $\beta_{f,0},~\zeta_{f,0}\in \mathbb{C},~\alpha_{f,0}\in R(Q_{f,+}^{u}(L_{\epsilon}))$ such that \eqref{eq.38} holds at
$\lambda=0$ with $\varphi^{sl}_{f}(\xi,0)=e^{-\eta\xi}\phi'_{a,\epsilon}(\xi)$, where $(\alpha_{f},\beta_{f},\zeta_{f})=(\alpha_{f,0},\beta_{f,0},\zeta_{f,0})$. Applying the projection $Q_{f,+}^{u}(L_{\epsilon})$ and \eqref{eq.projections space} to \eqref{eq.38} at $\xi=L_{\epsilon}$ derives
$$\alpha_{f,0}=Q_{f,+}^{u}(L_{\epsilon})e^{-\eta L_{\epsilon}}\phi'_{a,\epsilon}(L_{\epsilon}),
$$
and acting the projection $Q_{f,+}^{u}(0)$ on \eqref{eq.38} at $\xi=0$  yields
\begin{equation*}
%\begin{array}{ll}
Q_{f,+}^{u}(0)\phi'_{a,\epsilon}(0)=T_{f,+}^{u}(0,L_{\epsilon})e^{-\eta L_{\epsilon}}\phi'_{a,\epsilon}(L_{\epsilon})
       +\int_{L_{\epsilon}}^{0}T_{f,+}^{u}(0,\hat{\xi})B_{f}(\hat{\xi},0)e^{-\eta\hat{\xi}}\phi'_{a,\epsilon}(\hat{\xi})d\hat{\xi}.
%\end{array}
\end{equation*}
In a similar fashion,  there exist $\beta_{b,0}\in \mathbb{C},~\alpha_{b,0} \in {\rm R}(Q_{b,-}^{s}(-L_{\epsilon}))$ such that \eqref{eq.43} holds at $\lambda=0$  with $\varphi^{sl}_{b}(\xi,0)=e^{-\eta(\xi+Z_{a,\epsilon})}\phi'_{a,\epsilon}(\xi+Z_{a,\epsilon})$. Acting the projection $Q_{b,-}^{s}(-L_{\epsilon})$  on \eqref{eq.43} at $\xi=-L_{\epsilon}$,  together with \eqref{eq.projections space}, gives
$$
\alpha_{b,0}=Q_{b,-}^{s}(-L_{\epsilon})e^{-\eta(Z_{a,\epsilon}- L_{\epsilon})}\phi'_{a,\epsilon}(Z_{a,\epsilon}- L_{\epsilon}),
$$
and the projection $Q_{b,-}^{s}(0)$ on \eqref{eq.38} at $\xi=0$ forces
\begin{equation*}
\begin{split}%{ll}
Q_{b,-}^{s}(0)\phi'_{a,\epsilon}(Z_{a,\epsilon})=&T_{b,-}^{s}(0,-L_{\epsilon})e^{\eta L_{\epsilon}}\phi'_{a,\epsilon}(Z_{a,\epsilon}-L_{\epsilon})\\
       &+\int_{-L_{\epsilon}}^{0}T_{b,-}^{s}(0,\hat{\xi})B_{b}(\hat{\xi},0)e^{-\eta\hat{\xi}}\phi'_{a,\epsilon}(Z_{a,\epsilon}+\hat{\xi})d\hat{\xi}.\\
\end{split}
\end{equation*}
Statement $(ii)$ follows. It completes the proof of the proposition.
\end{proof}

\subsubsection{Along the back}
\hspace{0.13in}

{Similar to the last subsection, here we study the properties of the solution along the heteroclinic orbit of the layer system on the layer $w=w_b$.}

\begin{proposition}\label{pro.7}
Let $T_{b,\pm}^{u,s}(\xi,\hat{\xi})=T_{b,\pm}^{u,s}(\xi,\hat{\xi};a)$  be the $($un$)$stable evolution of system \eqref{eq.refuced eigenvalue} under the exponential dichotomy  established in Proposition \ref{pro.4} and the
associated projections are $Q_{b,\pm}^{u,s}(\xi)=Q_{b,\pm}^{u,s}(\xi;a)$. The following statements hold.
\begin{itemize}
\item[$(i)$]
 There exist $\delta,~\epsilon_{0}>0$ such that any solution $\varphi_{b,+}(\xi,\lambda)$ to \eqref{eq.shift} decaying exponentially in forward time  for $\lambda \in R_{1}(\delta)$ and $\epsilon\in(0,\epsilon_{0})$ satisfies
\begin{equation}\label{eq.varphi b+}
\begin{split}
\varphi_{b,+}(Z_{a,\epsilon},\lambda)&=\beta_{b,+}\omega_{b}(0)+\zeta_{b,+}Q_{b,+}^{s}\varphi_{2}\\
                    &+\beta_{b,+}\int_{L_{\epsilon}}^{0}T_{b,+}^{u}(0,\hat{\xi})B_{b}(\hat{\xi},\lambda)\omega_{b}(\hat{\xi})d\hat{\xi}
                          +\mathcal{H}_{b,+}(\beta_{b,+},\zeta_{b,+}),\\
Q_{b,-}^{u}(0)\varphi_{b,+}&(Z_{a,\epsilon},\lambda)=\beta_{b,+}\omega_{b}(0),
\end{split}
\end{equation}
for some $\beta_{b,+},~\zeta_{b,+}\in\mathbb{C}$, where $\mathcal{H}_{b,+}$ is a linear map satisfying the estimate
$$
\|\mathcal{H}_{b,+}(\beta_{b,+},\zeta_{b,+})\|\leq C \left((\epsilon |\log\epsilon|+|\lambda|)|\zeta_{b,+}|
                                                        +(\epsilon |\log\epsilon|+|\lambda|)^{2}|\beta_{b}|\right),
                                                        $$
with $C>0$ {a constant} independent of $\lambda,~a$ and $\epsilon$. Moreover, $\varphi_{b,+}(\xi,\lambda)$ is analytic in $\lambda$.

\item[$(ii)$]
The derivative $\phi'_{a,\epsilon}$ of the pulse solution satisfies
\begin{equation}\label{eq.pulse solution}
\begin{split}
Q_{b,+}^{u}(0)\phi'_{a,\epsilon}(Z_{a,\epsilon})=&T_{b,+}^{u}(0,L_{\epsilon})e^{-\eta L_{\epsilon}}\phi'_{a,\epsilon}(Z_{a,\epsilon+L_{\epsilon}})\\
                &+\int_{L_{\epsilon}}^{0}T_{b,+}^{u}(0,\hat{\xi})B_{b}(\hat{\xi},0)e^{-\eta\hat{\xi}}\phi'_{a,\epsilon}(Z_{a,\epsilon}+\hat{\xi})d\hat{\xi}.
\end{split}
\end{equation}
\end{itemize}
\end{proposition}

\begin{proof}
$(i)$. Consider the translated version \eqref{eq.b's perturbation } of system \eqref{eq.shift}, then any solution $\hat{\varphi}_{b,+}(\xi,\lambda)$ to \eqref{eq.b's perturbation } on $[0,~L_{\epsilon}]$ must
satisfy
\begin{equation}\label{eq.54}
\begin{split}
\hat{\varphi}_{b,+}(\xi,\lambda)=&T_{b,+}^{u}(\xi,L_{\epsilon})\alpha_{b,+}+\beta_{b,+}\omega_{b}(\xi)+\zeta_{b,+}T_{b,+}^{s}(\xi,0)\varphi_{2}\\
                &+\int_{0}^{\xi}T_{b,+}^{s}(\xi,\hat{\xi})B_{b}(\hat{\xi},\lambda)\hat{\varphi}_{b,+}(\hat{\xi},\lambda)d\hat{\xi}\\
                &\quad +\int_{L_{\epsilon}}^{\xi}T_{b,+}^{u}(\xi,\hat{\xi})B_{b}(\hat{\xi},\lambda)\hat{\varphi}_{b,+}(\hat{\xi},\lambda)d\hat{\xi},
\end{split}
\end{equation}
for some $\beta_{b,+},~\zeta_{b,+} \in \mathbb{C}$ and $\alpha_{b,+}\in {\rm R}(Q^{u}_{b,+}(L_{\epsilon}))$ by
the variation of constants formula. From Theorem \ref{the.expression}$(i)$, we derive
\begin{equation}\label{eq.55}
\begin{array}{ll}
\|B_{b}(\xi,\lambda;a,\epsilon)\|\leq C(\epsilon|\log \epsilon|+|\lambda|), ~\xi\in[0,L_{\epsilon}].
\end{array}
\end{equation}
Applying the contraction mapping principle to the functional equation \eqref{eq.54} yields a unique solution
$\hat{\varphi}_{b,+}$ for all $|\lambda|,~\epsilon>0$ sufficiently small.
Note that $\hat{\varphi}_{b,+}$ is linear in $(\alpha_{b,+},~\beta_{b,+},~\zeta_{b,+} )$ and satisfies the estimate
\begin{equation}\label{eq.56}
\begin{array}{ll}
\sup\limits_{\xi\in[0,L_{\epsilon}]} \|\hat{\varphi}_{b,+}(\xi,\lambda)\|\leq C(\|\alpha_{b,+}\|+|\beta_{b,+}|+|\zeta_{b,+}|),
\end{array}
\end{equation}
following \eqref{eq.55}, and taking $\delta,~\epsilon_{0}>0$ smaller if necessary.

By Proposition \ref{pro.2}, system \eqref{eq.shift} has the exponential dichotomy on $I_{l}=[Z_{a,\epsilon}+L_{\epsilon},\infty)$ with the associated projections $\mathcal{Q}_{l}^{u,s}(\xi,\lambda)$. Similar to the derivation of \eqref{eq.36} in the proof of Proposition \ref{pro.6}, we arrive at
\begin{equation}\label{eq.57}
\begin{array}{ll}
\|\mathcal{Q}_{l}^{u,s}(Z_{a,\epsilon}+L_{\epsilon},\lambda)-Q_{b,+}^{u,s}(L_{\epsilon})\|\leq C(\epsilon|\log\epsilon|+|\lambda|).
\end{array}
\end{equation}
Since any exponentially decaying solution of system \eqref{eq.shift} at $\xi=Z_{a,\epsilon}+L_{\epsilon}$ under the action of $\mathcal{Q}_{l}^{u}(Z_{a,\epsilon}+L_{\epsilon},\lambda)$ must be $0$, it follows that any solution $\varphi_{l}(\xi,\lambda)$ of system \eqref{eq.shift} decaying { exponentially} in forward time can be written as
\begin{equation}\label{eq.58}
\begin{array}{ll}
\varphi_{l}(\xi,\lambda)=\mathcal{T}_{l}^{s}(\xi,Z_{a,\epsilon}+L_{\epsilon},\lambda)\beta_{l},
\end{array}
\end{equation}
{with some $\beta_{l}\in {\rm R}(\mathcal{Q}_{l}^{s}(Z_{a,\epsilon}+L_{\epsilon},\lambda))$. Here $\mathcal{T}_{l}^{s}(\xi,\hat{\xi},\lambda)$ represents the stable evolution of system \eqref{eq.shift}.} Applying $\mathcal{Q}^{u}_{l}(Z_{a,\epsilon}+L_{\epsilon},\lambda)$ to $\hat{\varphi}_{b,+}(L_{\epsilon},\lambda)$ shows
\begin{equation}\label{eq.59}
\begin{split}
\alpha_{b,+}&=\mathcal{H}_{1}(\alpha_{b,+},\beta_{b,+},\zeta_{b,+}),\\
\|\mathcal{H}_{1}(\alpha_{b,+},\beta_{b,+},\zeta_{b,+})\|
&\leq C(\epsilon|\log\epsilon|+|\lambda|)(\|\alpha_{b,+}\|+|\beta_{b,+}|+|\zeta_{b,+}|),
\end{split}
\end{equation}
by using \eqref{eq.projections space}, \eqref{eq.55}, \eqref{eq.56} and \eqref{eq.57}. Therefore, solving \eqref{eq.59} for $\alpha_{b,+}$ yields
\begin{equation}\label{eq.60}
\begin{split}
\alpha_{b,+}&=\alpha_{b,+}(\beta_{b,+},\zeta_{b,+}),\\
\|\alpha_{b,+}(\beta_{b,+},\zeta_{b,+})\| &\leq C(\epsilon|\log\epsilon|+|\lambda|)(|\beta_{b,+}|+|\zeta_{b,+}|),
\end{split}
\end{equation}
for sufficiently small $|\lambda|,~\epsilon>0$. Substituting \eqref{eq.60} into \eqref{eq.54}, it holds
\begin{equation*}
\begin{split}
\hat{\varphi}_{b,+}(\xi,\lambda)=&T_{b,+}^{u}(\xi,L_{\epsilon})\alpha_{b,+}(\beta_{b,+},\zeta_{b,+})+\beta_{b,+}\omega_{b}(\xi)+\zeta_{b,+}T_{b,+}^{s}(\xi,0)\varphi_{2}\\
                &+\int_{0}^{\xi}T_{b,+}^{s}(\xi,\hat{\xi})B_{b}(\hat{\xi},\lambda)\hat{\varphi}_{b,+}(\hat{\xi},\lambda)d\hat{\xi}\\
                &+\int_{L_{\epsilon}}^{\xi}T_{b,+}^{u}(\xi,\hat{\xi})B_{b}(\hat{\xi},\lambda)\hat{\varphi}_{b,+}(\hat{\xi},\lambda)d\hat{\xi}.
\end{split}
\end{equation*}
Recall that
$$\hat{\varphi}_{b,+}(\xi-Z_{a,\epsilon},\lambda)=\varphi_{b,+}(\xi,\lambda),~\xi\in[Z_{a,\epsilon},Z_{a,\epsilon}+L_{\epsilon}].$$
Then, by using \eqref{eq.55}, \eqref{eq.56} and \eqref{eq.projections space}, we derive
\begin{equation*}
\begin{split}
\varphi_{b,+}(Z_{a,\epsilon},\lambda)=&\beta_{b,+}\omega_{b}(0)+\zeta_{b,+}Q_{b,+}^{s}\varphi_{2}\\
                    &+\beta_{b,+}\int_{L_{\epsilon}}^{0}T_{b,+}^{u}(0,\hat{\xi})B_{b}(\hat{\xi},\lambda)\omega_{b}(\hat{\xi})d\hat{\xi}
                     +\mathcal{H}_{b,+}(\beta_{b,+},\zeta_{b,+}).
\end{split}
\end{equation*}
{Since all quantities, the perturbed matrices $B_{b}(\xi,\lambda)$, the evolution $\mathcal{T}(\xi,\hat{\xi},\lambda)$ of system \eqref{eq.shift} and the projections $\mathcal{Q}^{u,s}_{l}(\xi,\lambda)$ associated with the exponential dichotomy of \eqref{eq.shift}, occurring in the above proofs analytically depend on $\lambda$, it induces that $\varphi_{b,+}(\xi,\lambda)$ is analytic in $\lambda$.}

\noindent $(ii)$. Similar to the proof of Proposition \ref{pro.6},  there exist $\beta_{b,+},~\zeta_{b,+}\in \mathbb{C}$ and $\alpha_{b,+}\in {\rm R}(Q_{b,+}^{u}(L_{\epsilon}))$ such that equation \eqref{eq.54} holds at
$\lambda=0$ with
\begin{equation}\label{eq.60.1}
\hat{\varphi}_{b,+}(\xi,0)=e^{-\eta(\xi+Z_{a,\epsilon})}\phi'_{a,\epsilon}(\xi+Z_{a,\epsilon}).
\end{equation}
Applying the projection $Q_{b,+}^{u}(L_{\epsilon})$ and \eqref{eq.projections space} to the above equation at $\xi=L_{\epsilon}$   derives
$$\alpha_{b,+}=Q_{b,+}^{u}(L_{\epsilon})e^{-\eta(L_{\epsilon}+Z_{a,\epsilon})}\phi'_{a,\epsilon}(L_{\epsilon}+Z_{a,\epsilon}).$$
Then, applying the projection $Q_{b,+}^{u}(0)$ to \eqref{eq.60.1} at $\xi=0$ gives \eqref{eq.pulse solution}.
This proves the proposition.
\end{proof}

\subsubsection{The matching procedure}
\hspace{0.13in}

In the previous section, we divided the real line $\mathbb{R}$ into three intervals and constructed a piecewise continuous, exponentially localized solution of system \eqref{eq.shift} for any $\lambda\in R_{1}(\delta)$. In the two discontinuous jumps at $\xi=0$ and $\xi=Z_{a,\epsilon}$, we get the expressions of the left and right limits of the solution, which are the entry and exit conditions {  along the right branch of the critical curve}. Determining existence of eigenvalues is now reduced to find $\lambda \in R_{1}(\delta)$ such that the exit and entry conditions match. After equalizing the exit condition and the entry condition, a single analytical matching equation in $\lambda$ can be obtained in the next result.

\begin{theorem}\label{the.lambda approximate}
There exist $\delta,~\epsilon_{0}>0$ such that for $\epsilon\in(0,\epsilon_{0})$ the shifted eigenvalue system \eqref{eq.shift} has precisely two different eigenvalues $\lambda_{0},~\lambda_{1}\in R_{1}(\delta)$.
\begin{itemize}
\item The eigenvalue $\lambda_{0}$ equals $0$ and the corresponding eigenspace is spanned by the solution $e^{-\eta\xi}\phi_{a,\epsilon}'(\xi)$ of system \eqref{eq.shift}.
\item The eigenvalue $\lambda_{1}$ is $a$--uniformly approximated by
\[ %begin{equation}\label{eq.lambda1 approximate}
%\begin{array}{ll}
\lambda_{1}=-\frac{M_{b,2}}{M_{b,1}}+O(|\epsilon\log\epsilon|^2),
%\end{array}
\] %end{equation}
where
\begin{equation}\label{eq.Mb1 and Mb2}
\begin{split}
M_{b,1}&=\int_{-\infty}^{+\infty}F(w_{b})(u'_{b}(\xi))^{2}e^{-c_{0}F^{2}(w_{b})\xi}d\xi,\\
M_{b,2}&=\langle\Psi_{*},\ \phi'_{a,\epsilon}(Z_{a,\epsilon}-L_{\epsilon})\rangle,
\end{split}
\end{equation}
with
$$\Psi_{*}=\left(\begin{array}{c}
                 e^{c_{0}F^{2}(w_{b})L_{\epsilon}}v_{b}'(-L_{\epsilon}) \\
                 -e^{c_{0}F^{2}(w_{b})L_{\epsilon}}u_{b}'(-L_{\epsilon}) \\
                 \int_{\infty}^{-L_{\epsilon}} u_{b}'(z)e^{-c_{0}F^{2}(w_{b})z}\Delta_{b}(z)dz
               \end{array}
          \right).$$
The corresponding eigenspace associated to $\lambda_{1}$ is spanned by a solution $\varphi_{1}(\xi)$ to system \eqref{eq.shift} satisfying
\begin{equation}\label{eq.varphi1 bound}
\begin{split}
\|\varphi_{1}(\xi+Z_{a,\epsilon})-\omega_{b}(\xi)\|&\leq C\epsilon|\log\epsilon|,\ \ \ ~\xi\in[-L_{\epsilon},L_{\epsilon}],\\
\|\varphi_{1}(\xi+Z_{a,\epsilon})\|&\leq C\epsilon|\log\epsilon|,\ \ \ ~\xi\in\mathbb{R}\setminus[-L_{\epsilon},L_{\epsilon}],
\end{split}
\end{equation}
where $C>0$ is {  a constant} independent of $a$ and $\epsilon$. Moreover, $M_{b,1}$ and $M_{b,2}$ satisfy the bounds
$$1/C\leq M_{b,1}\leq C,~|M_{b,2}|\leq C\epsilon|\log\epsilon|.$$
\end{itemize}
\end{theorem}

\begin{proof}
By Theorem \ref{the.expression}$(i)$, we derive
\begin{equation}\label{eq.63}
\begin{array}{ll}
\|B_{f}(\xi,\lambda;a,\epsilon)\|\leq C(\epsilon|\log\epsilon|+|\lambda|),\ \ ~\xi\in(-\infty,L_{\epsilon}],\\
\|B_{b}(\xi,\lambda;a,\epsilon)\|\leq C(\epsilon|\log\epsilon|+|\lambda|),\ \ ~\xi\in[-L_{\epsilon},L_{\epsilon}].
\end{array}
\end{equation}
Then it follows that
\begin{equation*}\label{eq.64}
\begin{split}
\left\| \left(\begin{array}{c}
    \phi'_{f}(\xi) \\
    0
  \end{array}
  \right)
-\phi'_{a,\epsilon}(\xi) \right\|
&\leq C\epsilon|\log\epsilon|,\ \ \ ~\xi\in(-\infty,L_{\epsilon}],\\
\left\| \left(\begin{array}{c}
    \phi'_{b}(\xi) \\
    0
  \end{array}
  \right)
-\phi'_{a,\epsilon}(Z_{a,\epsilon}+\xi) \right\|
&\leq C\epsilon|\log\epsilon|,\ \ \ ~\xi\in[-L_{\epsilon},L_{\epsilon}].
\end{split}
\end{equation*}
{According to Proposition \ref{pro.5}, any exponential decay solution $\varphi_{f,-}(\xi,\lambda)$ of system \eqref{eq.shift} in backward time satisfies \eqref{eq.varphi f-} at $\xi=0$ with some constant $\beta_{f,-}\in { \mathbb{C}}$.
Then, by Proposition \ref{pro.6}, it follows that there exist some $\beta_{f},~\zeta_{f}\in {\mathbb{C}}$ and $\beta_{b}\in { \mathbb{C}}$ such that any solution $\varphi^{sl}(\xi,\lambda)$ to \eqref{eq.shift} satisfies  \eqref{eq.varphi sl f} at $\xi=0$ and satisfies \eqref{eq.varphi sl b} at $\xi=Z_{a,\epsilon}$, respectively.
Finally, from Proposition \ref{pro.7} any  exponential decay solution $\varphi_{b,+}(\xi,\lambda)$ of system \eqref{eq.shift} in forward time satisfies \eqref{eq.varphi b+} at $\xi=Z_{a,\epsilon}$ with some $\beta_{b,+},~\zeta_{b,+}\in { \mathbb{C}}$.}
Next, we will match the solutions { $\varphi_{f,-}$ and $\varphi^{sl}$ at $\xi=0$, and match $\varphi^{sl}$ and $\varphi_{b,+}$} at $\xi=Z_{a,\epsilon}$. To do so is sufficient to require that
\begin{equation}\label{exmatch}
\begin{split}
Q_{f,-}^{u,s}(0)&\left(\varphi_{f,-}(0,\lambda)-\varphi^{sl}(0,\lambda)\right)=0,\\
Q_{b,-}^{u,s}(0)&\left(\varphi^{sl}(Z_{a,\epsilon},\lambda)-\varphi_{b,+}(Z_{a,\epsilon},\lambda)\right)=0.
\end{split}
\end{equation}
To solve these two equations, we need their concrete expressions.

{By computing the expressions of
$Q_{f,-}^{u}(0)\left(\varphi_{f,-}(0,\lambda)-\varphi^{sl}(0,\lambda)\right)=0$ and of
$Q_{b,-}^{u}(0)$$ \left(\varphi^{sl}(Z_{a,\epsilon},\lambda)-\varphi_{b,+}(Z_{a,\epsilon},\lambda)\right)=0$,
one can immediately obtain $\beta_{f}=\beta_{f,-}$ and $\beta_{b}=\beta_{b,+}$ by using \eqref{eq.varphi f-}, \eqref{eq.varphi sl f}, \eqref{eq.varphi sl b} and \eqref{eq.varphi b+}.  }

Next we consider the matching conditions \eqref{exmatch} with $s$. We define the vector
$$\varphi_{j,\perp}:=  \varphi_{1,j}-\int_{\infty}^{0}e^{-\eta\xi}\langle \psi_{j,ad}(\xi),F(\xi) \rangle d\xi ~\varphi_{2},\ \ \ ~j=f,~b.$$
Since ${\rm R}(Q_{j,-}^{s}(0))=\mbox{\rm Span}(\varphi_{1,j},\varphi_{2})$, it holds that $\varphi_{j,\perp}$ and $\varphi_{2}$ also span ${\rm R}(Q_{j,-}^{s}(0))$.
Moreover, some calculations show that
$$\varphi_{j,\perp}\in \mbox{\rm Ker}(Q_{j,+}^{s}(0)^{*})={\rm R}(Q_{j,+}^{u}(0)^{*})\subset {\rm R}(Q_{j,-}^{s}(0)^{*}),\ \ \ ~j=f,~b.$$
Then equations \eqref{exmatch} with $s$
%\begin{equation*}
%\begin{array}{ll}
%Q_{f,-}^{s}(0)[\varphi_{f,-}(0,\lambda)-\varphi^{sl}(0,\lambda)]=0,\\
%Q_{b,-}^{s}(0)[\varphi^{sl}(Z_{a,\epsilon},\lambda)-\varphi_{b,+}(Z_{a,\epsilon},\lambda)]=0.
%\end{array}
%\end{equation*}
can be rewritten as
\begin{equation}\label{exmatch1}
\begin{split}
&\left\langle \varphi_{2},\varphi_{f,-}(0,\lambda)-\varphi^{sl}(0,\lambda)  \right\rangle=0,\\
&\left\langle \varphi_{2},\varphi^{sl}(Z_{a,\epsilon},\lambda)-\varphi_{b,+}(Z_{a,\epsilon},\lambda) \right\rangle=0,\\
&\left\langle \varphi_{f,\perp},\varphi_{f,-}(0,\lambda)-\varphi^{sl}(0,\lambda)  \right\rangle=0, \\
&\left\langle \varphi_{b,\perp},\varphi^{sl}(Z_{a,\epsilon},\lambda)-\varphi_{b,+}(Z_{a,\epsilon},\lambda) \right\rangle=0.
\end{split}
\end{equation}
By the identities \eqref{eq.varphi f-}, \eqref{eq.varphi sl f}, \eqref{eq.varphi sl b} and \eqref{eq.varphi b+}, we can further write the first two equations as
\begin{equation}\label{eq.65}
\begin{split}
0&=\langle \varphi_{2},\varphi_{f,-}(0,\lambda)-\varphi^{sl}(0,\lambda)  \rangle=-\zeta_{f}+\mathcal{H}_{1}(\beta_{b},\beta_{f},\zeta_{f}),\\
0&=\langle \varphi_{2},\varphi^{sl}(Z_{a,\epsilon},\lambda)-\varphi_{b,+}(Z_{a,\epsilon},\lambda) \rangle=-\zeta_{b,+}+\mathcal{H}_{2}(\beta_{b},\beta_{f},\zeta_{f},\zeta_{b,+}),
\end{split}
\end{equation}
where
\begin{equation*}
\begin{split}
|\mathcal{H}_{1}(\beta_{b},\beta_{f},\zeta_{f})| &\leq C((\epsilon|\log\epsilon|+|\lambda|)(|\beta_{f}|+|\zeta_{f}|)+e^{-q/\epsilon}|\beta_{b}|),    \\
|\mathcal{H}_{2}(\beta_{b},\beta_{f},\zeta_{f},\zeta_{b,+})|  &\leq C((\epsilon|\log\epsilon|+|\lambda|)(|\beta_{b}|+|\zeta_{b,+}|)+e^{-q/\epsilon}(|\beta_{f}|+|\zeta_{f}|)),
\end{split}
\end{equation*}
with $q>0$ {  a constant} independent of $\lambda,~a,~\epsilon$. Hence  system \eqref{eq.65} is solvable for $\zeta_{f}$ and $\zeta_{b,+}$, provided that $|\lambda|,~\epsilon>0$ are sufficiently small, and the solutions satisfy
\begin{equation}\label{eq.66}
\begin{split}
\zeta_{f}=\zeta_{f}(\beta_{b},&\beta_{f}),~\zeta_{b,+}=\zeta_{b,+}(\beta_{b},\beta_{f}),\\
|\zeta_{f}(\beta_{b},\beta_{f})|  &\leq C\left((\epsilon|\log\epsilon|+|\lambda|)|\beta_{f}|+e^{-q/\epsilon}|\beta_{b}|\right),\\
|\zeta_{b,+}(\beta_{b},\beta_{f})| &\leq C\left((\epsilon|\log\epsilon|+|\lambda|)|\beta_{b}|+e^{-q/\epsilon}|\beta_{f}|\right).
\end{split}
\end{equation}
Combining \eqref{eq.varphi f-}, \eqref{eq.varphi sl f} and \eqref{eq.nu bound} with $\varphi_{f,\perp}\in {\rm R}(Q_{f,+}^{u}(0)^{*})$, the third equation of \eqref{exmatch1} can be written in
\begin{equation}\label{eq.67}
\begin{split}
0&=\langle \varphi_{f,\perp},\varphi_{f,-}(0,\lambda)-\varphi^{sl}(0,\lambda)  \rangle \\
 &=\beta_{f}\int_{-L_{\epsilon}}^{L_{\epsilon}} \langle T_{f}(0,\hat{\xi})^{*}\varphi_{f,\perp},B_{f}(\hat{\xi},\lambda)\omega_{f}(\hat{\xi}) \rangle d\hat{\xi}+\mathcal{H}_{3}(\beta_{f},\beta_{b}),
\end{split}
\end{equation}
where
\begin{equation*}
\begin{array}{ll}
|\mathcal{H}_{3}(\beta_{f},\beta_{b})|\leq C\left((\epsilon|\log\epsilon|+|\lambda|)^{2}|\beta_{f}|+e^{-q/\epsilon}|\beta_{b}|\right).
\end{array}
\end{equation*}
Similarity, by using \eqref{eq.varphi sl b}, \eqref{eq.varphi b+}, \eqref{eq.nu bound} and $\varphi_{b,\perp}\in {\rm R}(Q_{b,+}^{u}(0)^{*})$, the fourth equation of \eqref{exmatch1} can be written in
\begin{equation}\label{eq.68}
\begin{split}
0&=\left\langle \varphi_{b,\perp},\varphi^{sl}(Z_{a,\epsilon},\lambda)-\varphi_{b,+}(Z_{a,\epsilon},\lambda)  \right\rangle \\
 &=\beta_{b}\int_{-L_{\epsilon}}^{L_{\epsilon}} \left\langle T_{b}(0,\hat{\xi})^{*}\varphi_{b,\perp},B_{b}(\hat{\xi},\lambda)\omega_{b}(\hat{\xi}) \right\rangle d\hat{\xi}+\mathcal{H}_{4}(\beta_{f},\beta_{b}),
\end{split}
\end{equation}
where
\begin{equation*}
|\mathcal{H}_{4}(\beta_{f},\beta_{b})|\leq C \left((\epsilon|\log\epsilon|+|\lambda|)^{2}|\beta_{b}|+e^{-q/\epsilon}|\beta_{f}|\right).
\end{equation*}
To present clear approximate expressions of \eqref{eq.67} and \eqref{eq.68}, we first give the following approximations
\begin{equation}\label{eq.69}
\begin{split}
0&=\left\langle \varphi_{f,\perp},\phi_{a,\epsilon}'(0)- \phi_{a,\epsilon}'(0) \right\rangle\\
 &=\left\langle \varphi_{f,\perp},Q_{f,-}^{s}(0)\phi_{a,\epsilon}'(0)- Q_{f,+}^{u}(0)\phi_{a,\epsilon}'(0) \right\rangle \\
 &=\int_{-L_{\epsilon}}^{L_{\epsilon}} \left\langle e^{-\eta\hat{\xi}}T_{f}(0,\hat{\xi})^{*}\varphi_{f,\perp},B_{f}(\hat{\xi},0)\phi_{a,\epsilon}'(\hat{\xi}) \right\rangle d\hat{\xi}+O(\epsilon^{2}),
\end{split}
\end{equation}
\begin{equation}\label{eq.70}
\begin{split}
0&=\left\langle \varphi_{b,\perp},\phi_{a,\epsilon}'(Z_{a,\epsilon})- \phi_{a,\epsilon}'(Z_{a,\epsilon}) \right\rangle \\
 &=\left\langle \varphi_{b,\perp},Q_{b,-}^{s}(0)\phi_{a,\epsilon}'(Z_{a,\epsilon})- Q_{b,+}^{u}(0)\phi_{a,\epsilon}'(Z_{a,\epsilon}) \right\rangle \\
 &=\int_{-L_{\epsilon}}^{L_{\epsilon}} \left\langle e^{-\eta\hat{\xi}}T_{b}(0,\hat{\xi})^{*}\varphi_{b,\perp},B_{b}(\hat{\xi},0)\phi_{a,\epsilon}'(Z_{a,\epsilon}+\hat{\xi}) \right\rangle d\hat{\xi}\\
 &\ \ \ \ \ ~~~~ + \langle e^{\eta L_{\epsilon}}T_{b}(0,-L_{\epsilon})^{*}\varphi_{b,\perp},\phi_{a,\epsilon}'(Z_{a,\epsilon}-L_{a,\epsilon}) \rangle+O(\epsilon^{2}).
\end{split}
\end{equation}
Next, we simplify the expressions in \eqref{eq.67} and \eqref{eq.68} by \eqref{eq.69} and \eqref{eq.70}. Direct calculations yield
\begin{equation}\label{eq.71}
\begin{split}
e^{-\eta\xi}T_{j}(0,\xi)^{*}\varphi_{j,\perp}&=\left(\begin{array}{c}
 \displaystyle                                                     e^{-\eta\xi}\psi_{j,ad}(\xi) \\
 \displaystyle                                                      \int_{\xi}^{\infty}e^{-\eta z}\langle \psi_{j,ad}(z),F_{j}(z) \rangle dz
                                                    \end{array}
                                                    \right)\\
&=\left(\begin{array}{c}
               \displaystyle                                        e^{-c_{0}F^{2}(w_{j})\xi}v_{j}'(\xi) \\
                             \displaystyle                          -e^{-c_{0}F^{2}(w_{j})\xi}u_{j}'(\xi) \\
                                           \displaystyle            \int^{\xi}_{\infty} e^{-c_{0}F^{2}(w_{j}){z}}u_{j}'(z)\Delta_{j}(z) dz
                                                    \end{array}
                                                    \right),~\xi\in \mathbb{R},~j=f,~b.
\end{split}
\end{equation}
Recall that $\phi_{f}'(\xi)$  converges to $0$ at the exponential rate $F(0)\sqrt{\frac{k}{2}}$ as $\xi\rightarrow\pm\infty$, and
$\phi_{b}'(\xi)$  converges to $0$ at the exponential rate $F(w_{b})U_{2}(w_{b})\sqrt{\frac{k}{2}}$ as $\xi\rightarrow\pm\infty$.
Note that $c_{0}=\frac{\sqrt{2k}}{F(0)}(\frac{1}{2}-a)$, and $w_{b}$ satisfies \eqref{eq.wb and w0 relation}.
{Thus, for all $a\geq 0$, there exists an $a$--independent constant $C>0$ such that the upper two entries of \eqref{eq.71} are bounded by $C$ on $\mathbb{R}$, and the last entry is bounded by $C|\log \epsilon|$  on $[-L_{\epsilon},L_{\epsilon}]$.}
Combining these bounds together with \eqref{eq.nu bound}, \eqref{eq.69} and \eqref{eq.71}, we get the next $a$--uniform approximation
\begin{equation}\label{eq.72}
\begin{split}
 &\int_{-L_{\epsilon}}^{L_{\epsilon}} \left\langle  T_{f}(0,\xi)^{*}\varphi_{f,\perp},B_{f}(\xi,\lambda)\omega_{f}(\xi)   \right\rangle d\xi\\
=&\int_{-L_{\epsilon}}^{L_{\epsilon}} \left\langle  e^{-\eta\xi}T_{f}(0,\xi)^{*}\varphi_{f,\perp},B_{f}(\xi,\lambda)\phi_{a,\epsilon}'(\xi)   \right\rangle d\xi+O(|\epsilon\log\epsilon|^{2})\\
=&\int_{-L_{\epsilon}}^{L_{\epsilon}} \left\langle  e^{-\eta\xi}T_{f}(0,\xi)^{*}\varphi_{f,\perp},B_{f}(\xi,0)\phi_{a,\epsilon}'(\xi)   \right\rangle d\xi \\
&\ \ \ \ \ -\lambda \int_{-L_{\epsilon}}^{L_{\epsilon}} F(0) e^{-c_{0}F^{2}(0)\xi}(u_{f}'(\xi))^{2}d\xi +O(|\epsilon\log\epsilon|^{2})\\
=&-\lambda \int_{-\infty}^{\infty} F(0) e^{-c_{0}F^{2}(0)\xi}(u_{f}'(\xi))^{2}d\xi+O(|\epsilon\log\epsilon|^{2}).
\end{split}
\end{equation}
By similar calculations together with \eqref{eq.nu bound}, \eqref{eq.70} and \eqref{eq.71}, one has the next $a$--uniform approximation
\begin{equation}\label{eq.73}
\begin{split}
 &\int_{-L_{\epsilon}}^{L_{\epsilon}} \left\langle  T_{b}(0,\xi)^{*}\varphi_{b,\perp},B_{b}(\xi,\lambda)\omega_{b}(\xi)   \right\rangle d\xi\\
=&\int_{-L_{\epsilon}}^{L_{\epsilon}} \left\langle  e^{-\eta\xi}T_{b}(0,\xi)^{*}\varphi_{b,\perp},B_{b}(\xi,\lambda)\phi_{a,\epsilon}'(Z_{a,\epsilon}+\xi)   \right\rangle d\xi+O(|\epsilon\log\epsilon|^{2})\\
=&\int_{-L_{\epsilon}}^{L_{\epsilon}} \left\langle  e^{-\eta\xi}T_{b}(0,\xi)^{*}\varphi_{b,\perp},B_{b}(\xi,0)\phi_{a,\epsilon}'(Z_{a,\epsilon}+\xi)   \right\rangle d\xi\\
&\ \ \ -\lambda \int_{-L_{\epsilon}}^{L_{\epsilon}} F(w_{b}) e^{-c_{0}F^{2}(w_{b})\xi}(u_{b}'(\xi))^{2}d\xi +O(|\epsilon\log\epsilon|^{2})\\
=&-\left\langle e^{\eta L_{\epsilon}}T_{b}(0,-L_{\epsilon})^{*}\varphi_{b,\perp},\phi'_{a,\epsilon}(Z_{a,\epsilon}-L_{\epsilon}) \right\rangle \\
 &\ \ \ -\lambda \int_{-\infty}^{\infty} F(w_{b}) e^{-c_{0}F^{2}(w_{b})\xi}(u_{b}'(\xi))^{2}d\xi+O(|\epsilon\log\epsilon|^{2}).
\end{split}
\end{equation}
Using \eqref{eq.72} and \eqref{eq.73}, the matching conditions \eqref{eq.67} and \eqref{eq.68} can be written in the next form
\begin{equation}\label{eq.74}
\begin{split}
&\left(
        \begin{array}{cc}
          \lambda M_{f}+O((\epsilon|\log\epsilon|+|\lambda|)^{2}) & O(e^{-q/\epsilon}) \\
         O(e^{-q/\epsilon})  & -\lambda M_{b,1}-M_{b,2}+O((\epsilon|\log\epsilon|+|\lambda|)^{2})\\
        \end{array}
\right)
\left(\begin{array}{c}
        \beta_{f} \\
        \beta_{b}
      \end{array}
\right)\\
&=\vec{0},
\end{split}
\end{equation}
where the approximations are $a$--uniformly,
\begin{equation}\label{eq.75}
  M_{f}=\int_{-\infty}^{\infty}F(0)e^{-c_{0}F^{2}(0)\xi}(u'_{f}(\xi))^{2}d\xi,
\end{equation}
and  $M_{b,1}$ and $M_{b,2}$ are those defined in \eqref{eq.Mb1 and Mb2}. Hence, any nontrivial solution $(\beta_{f},\beta_{b})$ to system \eqref{eq.74} corresponds to an eigenfunction of the shifted eigenvalue system \eqref{eq.shift}.

{Since all quantities, the perturbed matrices $B_{j}(\xi,\lambda),~j=f,~b,$ the evolution $\mathcal{T}(\xi,\hat{\xi},\lambda)$ of system \eqref{eq.shift} and the projections $\mathcal{Q}^{u,s}_{r,l}(\xi,\lambda)$ associated with the exponential dichotomy of \eqref{eq.shift}, occurring in this section are analytic in $\lambda$, it induces that the matrix in \eqref{eq.74} and its determinant $D(\lambda)=D(\lambda;a,\epsilon)$ are analytic in $\lambda$.} Since $u_{j}(\xi)$ converges to $0$ as $\xi\rightarrow \pm \infty$ at an exponential rate, then the $\epsilon$--independent quantities $M_{f}$ and $M_{b,1}$ {  are at} leading order bounded away from $0$. {It follows that $1/C \leq M_{f},~M_{b,1}\leq C$.} Combining \eqref{eq.63} and \eqref{eq.70} will arrive $a$--uniform estimate $M_{b,2}=O(\epsilon|\log\epsilon|)$.
{Hence we have
$$|D(\lambda)-\lambda M_{f}(\lambda M_{b,1}+M_{b,2})|<|\lambda M_{f}(\lambda M_{b,1}+M_{b,2})|$$
for $\lambda\in \partial R_{1}(\delta):=\{\lambda\in \mathbb{C}:~|\lambda|=\delta\}$ with $\delta,~\epsilon>0$ sufficiently small.
Since the roots of the quadratic equation $\lambda M_{f}(\lambda M_{b,1}+M_{b,2})=0$ in $\lambda$ are $0$ and $-M_{b,2}M_{b,1}^{-1}$, $D(\lambda)$ has precisely two roots $\lambda_{0},~\lambda_{1}$ in $R_{1}(\delta)$, by Rouch\'e Theorem, which are $a$--uniformly $O(|\epsilon\log\epsilon|^{2})$--close to $0$ and $-M_{b,2}M_{b,1}^{-1}$.
Thus system \eqref{eq.shift} has two eigenvalues $\lambda_{0},~\lambda_{1}$ in the region $R_{1}(\delta)$.}

{Let $\lambda_{1}$ be the eigenvalue, which is $a$--uniformly $O(|\epsilon\log\epsilon|^{2})$--close to $-M_{b,2} M_{b,1}^{-1}$, and  $\varphi_{1}(\xi)$ be the associated eigenfunction of system $\eqref{eq.shift}$. The eigenvector $(\beta_{f},\beta_{b})^{T}=(O(e^{-q/\epsilon}),1)^{T}$ is the associated solution to system \eqref{eq.74}.}
Propositions \ref{pro.5}, \ref{pro.6} and \ref{pro.7} provide a piecewise continuous eigenfunction to system \eqref{eq.shift} for any prospective eigenvalue $\lambda\in R_{1}(\delta)$. Thus, the eigenfunction $\varphi_{1}(\xi)$ to $\eqref{eq.shift}$ satisfies $\eqref{eq.varphi f-}$ on $I_{f,-}$, $\eqref{eq.38}$ on $I_{f,+}$, $\eqref{eq.40}$ on $I_{r}$, $\eqref{eq.43}$ on $I_{b,-}$, $\eqref{eq.54}$ on $I_{b,+}$, and $\eqref{eq.58}$ on $I_{l}$.
{Moreover, $\beta_{f}=O(e^{-q/\epsilon})$ and $\beta_{b}=1$ can represent all variables occurring in these six expressions and we obtain the approximation \eqref{eq.varphi1 bound} of $\varphi_{1}(\xi)$.}

By translational invariance, it holds that $e^{-\eta\xi}\phi'_{a,\epsilon}(\xi)$ is an eigenfunction
of the shifted eigenvalue system \eqref{eq.shift} at $\lambda=0$. Therefore, $\lambda=0$ is one of the two eigenvalues  $\lambda_{0}$ and $\lambda_{1}$.
According to \eqref{eq.varphi1 bound}, it holds that the eigenfunction $\varphi_{1}(\xi)$ is not
a multiple of $e^{-\eta\xi}\phi'_{a,\epsilon}(\xi)$.
{By Lemma \ref{lem.1} the asymptotic matrix $\hat{A}(0,0,\lambda,a,\epsilon)$ of the shifted eigenvalue system \eqref{eq.shift} has precisely one eigenvalue with positive real part, it induces that the space of the exponentially decaying
solutions in backward time to \eqref{eq.shift} is one dimensional.}
Thereby, $\varphi_{1}(\xi)$ and $e^{-\eta\xi}\phi'_{a,\epsilon}(\xi)$ must correspond to
different eigenvalues. Consequently, $\lambda_{0}=0$ and $\lambda_{1}\neq\lambda_{0}$.

It completes the proof of the proposition.
\end{proof}

\subsubsection{The translational eigenvalue $\lambda_{0}=0$ is simple}
\hspace{0.13in}

In this section, we prove that the eigenvalue $\lambda_{0}=0$ of $\mathcal{L}_{a,\epsilon}$ is simple. Recall that $\lambda_{0}=0$ has geometric
multiplicity one by the proof of Theorem \ref{the.lambda approximate}.

\begin{proposition}\label{pro.lambda0 simple}
The translational eigenvalue $\lambda_{0}=0$ of $\mathcal{L}_{a,\epsilon}$ is simple.
\end{proposition}
\begin{proof}
{According to Theorem \ref{the.lambda approximate}, we obtain that the eigenspace of the shifted eigenvalue system \eqref{eq.shift} at
$\lambda_{0}=0$ is spanned by $e^{-\eta\xi}\phi'_{a,\epsilon}(\xi)$. Thus, translating back to
system \eqref{full matrix}, it holds that the kernel $\ker(\mathcal{L}_{a,\epsilon})$ is of one dimensional and is spanned by $\widetilde{\phi}'_{a,\epsilon}(\xi)=(u_{a,\epsilon}(\xi),w_{a,\epsilon}(\xi))^{T}$. Thereby, the geometric multiplicity of $\lambda_{0}=0$ for { $\mathcal{L}_{a,\epsilon}$} is equal to one.}
{ Next, we will prove that the algebraic multiplicity of $\lambda_{0}=0$ is also equal to one}, i.e., there is no exponentially localized solutions
$\widetilde{\psi}(\xi)$ to the generalized eigenvalue problem $\mathcal{L}_{a,\epsilon}\widetilde{\psi}=\widetilde{\phi}'_{a,\epsilon}(\xi)$.
This problem can be rewritten as
\begin{equation}\label{eq.77}
\begin{array}{ll}
\check{\psi}_{\xi}=A_{0}(\xi,0)\check{\psi}+ \partial_{\lambda}A_{0} (\xi,0)\phi'_{a,\epsilon}(\xi),
\end{array}
\end{equation}
with $A_{0}(\xi,0)$ as that in system \eqref{full matrix}. Recall from Proposition \ref{pro.ess} and Lemma \ref{lem.1} that the asymptotic matrices of $A_{0}(\xi,\lambda)$ and
{ its} shifted version $A(\xi,\lambda)$ have precisely one eigenvalue with positive real part at $\lambda=0$.
Since $e^{-\eta\xi}\phi'_{a,\epsilon}(\xi)$ is exponentially localized, it follows that
$\check{\psi}_{\xi}$ is an exponentially localized solution to system \eqref{eq.77} if and only if $\psi_{\xi}=e^{-\eta\xi}\check{\psi}_{\xi}$ is an exponentially localized solution to the system
\begin{equation}\label{eq.78}
%\begin{array}{ll}
\psi_{\xi}=A(\xi,0)\psi+e^{-\eta\xi} \partial_{\lambda}A (\xi,0)\phi'_{a,\epsilon}(\xi),
%\end{array}
\end{equation}
where $A(\xi,0)$ is the coefficient matrix of the shifted eigenvalue problem \eqref{eq.shift} at $\lambda=0$.

Again using the fact that $e^{-\eta\xi}\phi'_{a,\epsilon}(\xi)$ is an exponentially localized solution to \eqref{eq.shift} at $\lambda=0$, together with Propositions \ref{pro.5}, \ref{pro.6} and \ref{pro.7}, one can get { the solutions of system \eqref{eq.shift} on different intervals}, $\varphi_{f,-}(\xi,\lambda)$, $\varphi^{sl}(\xi,\lambda)$ and $\varphi_{b,+}(\xi,\lambda)$, which are analytic in $\lambda$ and satisfy
\[
\begin{split}
e^{-\eta\xi}\phi'_{a,\epsilon}(\xi)&=\varphi_{f,-}(\xi,0),\ \ \ ~\xi\in (-\infty,0],\\
 e^{-\eta\xi}\phi'_{a,\epsilon}(\xi)&=\varphi^{sl}(\xi,0),\ \ \ \ \ ~\xi\in [0,Z_{a,\epsilon}],\\
 e^{-\eta\xi}\phi'_{a,\epsilon}(\xi)&=\varphi_{b,+}(\xi,0), \ \ \  ~\xi\in [Z_{a,\epsilon},+\infty),
\end{split}
\]
for some $\beta_{f,-},~\beta_{f},~\zeta_{f},~\beta_{b},~\beta_{b,+},~\zeta_{b,+}\in \mathbb{C}$, {  which are given in Propositions \ref{pro.5}, \ref{pro.6} and \ref{pro.7}.}
As in the proof of Theorem \ref{the.lambda approximate}, applying the projections $Q_{j,-}^{u}(0),~j=f,~b$ to the differences $\varphi_{f,-}(0,0)-\varphi^{sl}(0,0)$
and $\varphi^{sl}(Z_{a,\epsilon},0)-\varphi_{b,+}(Z_{a,\epsilon},0)$ yield $\beta_{f,-}=\beta_{f}$ and $\beta_{b,+}=\beta_{b}$.
According to \eqref{eq.66}, one knows that $\zeta_f$ and $\zeta_{b,+}$ are also treated as functions of $\beta_b$ and $\beta_f$. Moreover, we derive
\begin{equation*}\label{eq.79}
\begin{split}
\zeta_{f}=\zeta_{f}(\beta_{b}, \ & \beta_{f}),~\zeta_{b,+}=\zeta_{b,+}(\beta_{b},\beta_{f}),\\
|\zeta_{f}(\beta_{b},\beta_{f})|  &\leq C\left(\epsilon|\log\epsilon||\beta_{f}|+e^{-q/\epsilon}|\beta_{b}|\right),\\
|\zeta_{b,+}(\beta_{b},\beta_{f})|  &\leq C\left(\epsilon|\log\epsilon||\beta_{b}|+e^{-q/\epsilon}|\beta_{f}|\right).
\end{split}
\end{equation*}
where $C>0$ is a constant independent of $a$ and $\epsilon$.

{Note that $ \partial_{\lambda}\varphi_{f,-} (\xi,0),~ \partial_{\lambda}\varphi^{sl} (\xi,0)$ and $ \partial_{\lambda}\varphi_{b,+} (\xi,0)$ are
particular solutions to equation \eqref{eq.78} on $(-\infty,0],~[0,Z_{a,\epsilon}]$ and $[Z_{a,\epsilon},+\infty)$ respectively, and that the space of exponentially localized solutions to the homogeneous problem \eqref{eq.shift} associated to \eqref{eq.78} is spanned by $e^{-\eta\xi}\phi'_{a,\epsilon}(\xi)$.
Suppose $\psi(\xi)$ is an exponentially localized solution to \eqref{eq.78}.}
Then, it holds
\begin{equation}\label{eq.80}
\begin{split}
\psi(\xi)&= \partial_{\lambda}\varphi_{f,-} (\xi,0)+\alpha_{1}e^{-\eta\xi}\phi'_{a,\epsilon}(\xi),\ \ \ ~\xi\in (-\infty,0], \\
\psi(\xi)&= \partial_{\lambda}\varphi^{sl} (\xi,0)+\alpha_{2}e^{-\eta\xi}\phi'_{a,\epsilon}(\xi),\ \ \ \ \ ~\xi\in [0,Z_{a,\epsilon}], \\
\psi(\xi)&= \partial_{\lambda}\varphi_{b,+} (\xi,0)+\alpha_{3}e^{-\eta\xi}\phi'_{a,\epsilon}(\xi),\ \ \ ~\xi\in [Z_{a,\epsilon},+\infty),
\end{split}
\end{equation}
for some $\alpha_{1},~\alpha_{2},~\alpha_{3}\in \mathbb{C}$. Differentiating the analytic expressions \eqref{eq.varphi f-} and \eqref{eq.varphi sl f} with respect to $\lambda$ derives
\begin{equation}\label{eq.81}
\begin{split}
 \partial_{\lambda}\varphi_{f,-} (0,0) &=\beta_{f}\int_{-\infty}^{0}T_{f,-}^{s}(0,\hat{\xi})\partial_{\lambda}B_{f}(\hat{\xi},0)\omega_{f}(\hat{\xi})d\hat{\xi}
+\mathcal{H}_{1}(\beta_{f}), \\
\partial_{\lambda} \varphi^{sl}(0,0) &=
\beta_{f}\int_{L_{\epsilon}}^{0}T_{f,+}^{u}(0,\hat{\xi})\partial_{\lambda}B_{f}(\hat{\xi},0)\omega_{f}(\hat{\xi})d\hat{\xi}+\mathcal{H}_{2}(\beta_{f},\beta_{b}),\\
\| \mathcal{H}_{1}(\beta_{f})  \|  &\leq C\epsilon|\log\epsilon||\beta_{f}|,\\
\| \mathcal{H}_{2}(\beta_{f},\beta_{b})  \|  &\leq C(\epsilon|\log\epsilon||\beta_{f}|+e^{-q/\epsilon}|\beta_{b}|),
\end{split}
\end{equation}
where
$$\partial_{\lambda}B_{f}(\hat{\xi},0)=\left(
    \begin{array}{ccc}
     0 & 0 & 0 \\
      F(w_{a,\epsilon}(\xi)) & 0 & 0 \\
      0 & 0 & -\frac{1}{c} \\
    \end{array}
  \right) =:  \widetilde{B}.
$$
By Theorem \ref{the.expression}$(i)$, for $\xi\in J_{f}=(-\infty,L_{\epsilon}]$, we obtain
\begin{equation}\label{eq.82}
\begin{array}{ll}
\| \varphi_{f,a,\epsilon}-\varphi_{1,f}  \|\leq C\epsilon|\log \epsilon|,~~\mathrm{where}~ \varphi_{f,a,\epsilon}=\left(\begin{array}{c}
                                                                                                                    v_{a,\epsilon}'(0) \\
                                                                                                                    -u_{a,\epsilon}'(0) \\
                                                                                                                    0
                                                                                                                  \end{array}\right).
\end{array}
\end{equation}
Some calculations show that $\varphi_{f,a,\epsilon} \perp \phi'_{a,\epsilon}(0)$. Moreover, by expressions \eqref{eq.Q projections}, it holds
$$
\varphi_{1,f}\in \left(R(Q_{f,-}^{s}(0))\right)\cap \left( R(Q_{f,+}^{u}(0))\right).
$$
Combining these results with \eqref{eq.80}, \eqref{eq.81}, \eqref{eq.82} and \eqref{eq.nu bound} yield
\begin{equation}\label{eq.83}
\begin{split}
0&=\left\langle \varphi_{f,a,\epsilon}, \ \partial_{\lambda}\varphi_{f,-} (0,0)-  \partial_{\lambda} \varphi^{sl} (0,0)+(\alpha_{1}-\alpha_{2})\phi'_{a,\epsilon}(0)   \right \rangle \\
&=\left\langle \varphi_{f,a,\epsilon}, \ \partial_{\lambda}\varphi_{f,-} (0,0)- \partial_{\lambda} \varphi^{sl} (0,0)   \right\rangle\\
&=\beta_{f}\left(\int_{-\infty}^{L_{\epsilon}} \left\langle T_{f}(0,\xi)^{*}\varphi_{1,f},\widetilde{B}\omega_{f}(\xi)  \right\rangle d\xi+O(\epsilon|\log\epsilon|)\right)+\beta_{b}O(e^{-q/\epsilon})\\
&=\beta_{f}(-M_{f}+O(\epsilon|\log\epsilon|))+\beta_{b}O(e^{-q/\epsilon}),
\end{split}
\end{equation}
with the asymptotic expression $a$--uniformly, where $M_{f}$ is defined in \eqref{eq.75}. Let $\varphi_{b,a,\epsilon}=(v'_{a,\epsilon}(Z_{a,\epsilon}),-u'_{a,\epsilon}(Z_{a,\epsilon}),0)^{T}$. Similar calculation as above shows
\begin{equation}\label{eq.84}
\begin{split}
0&=\left\langle \varphi_{b,a,\epsilon}, \  \partial_{\lambda}\varphi^{sl} (Z_{a,\epsilon},0)- \partial_{\lambda} \varphi_{b,+} (Z_{a,\epsilon},0)+(\alpha_{2}-\alpha_{3})e^{-\eta Z_{a,\epsilon}}\phi'_{a,\epsilon}(Z_{a,\epsilon}) \right\rangle \\
&=\beta_{f}(-M_{b,1}+O(\epsilon|\log\epsilon|))+\beta_{f}O(e^{-q/\epsilon}),
\end{split}
\end{equation}
with the asymptotic expression $a$--uniformly, where $M_{b,1}$ is defined in \eqref{eq.Mb1 and Mb2}. The conditions \eqref{eq.83} and \eqref{eq.84} form a system
\begin{equation}\label{eq.83 and 84}
\begin{array}{ll}
\left(
  \begin{array}{cc}
    -M_{f}+O(\epsilon|\log\epsilon|) & O(e^{-q/\epsilon}) \\
    O(e^{-q/\epsilon}) & -M_{b,1}+O(\epsilon|\log\epsilon|) \\
  \end{array}
\right)\left(\begin{array}{c}
               \beta_{f} \\
               \beta_{b}
             \end{array}
\right)=\vec{0}.
\end{array}
\end{equation}
Since $M_{f},~M_{b,1}>0$ are independent of $\epsilon$ and bounded below away from $0$ uniformly in $a$, system \eqref{eq.83 and 84} has only the trivial soluton $\beta_{f}=\beta_{b}=0$. We are in contradiction with the fact that $e^{-\eta\xi}\phi'_{a,\epsilon}(\xi)$ is not the zero solution to the shifted eigenvalue system \eqref{eq.shift}.
So far, we arrive the conclusions that system \eqref{eq.78} has no exponentially localized solution and that the algebraic multiplicity of the eigenvalue $\lambda=0$ of $\mathcal{L}_{a,\epsilon}$ is also equal to one.
\end{proof}

\subsubsection{Approximate calculation of $\lambda_{1}$}
\hspace{0.13in}

{By Theorem \ref{the.lambda approximate}, the second eigenvalue $\lambda_{1}\in {R_{1}(\delta)}$ of the shifted eigenvalue system \eqref{eq.shift} is $a$--uniformly $O(|\epsilon\log\epsilon|^2)$--close to $-M_{b,2}M_{b,1}^{-1}$. Thus, we need to show $-M_{b,2}M_{b,1}^{-1}\leq -b_{0}\epsilon$ for proving our main stability results with a constant independent of $a$ and $\epsilon$.}

\begin{proposition}\label{pro.lambda1 approximate}
For $k_1>0$ small given in Lemma \ref{lem.1}, there exists an $\epsilon_{0}>0$ such that for each $(a,\epsilon)\in [0,\frac{1}{2}-k_{1}]\times (0,\epsilon_{0})$, $M_{b,2}$ in Theorem \ref{the.lambda approximate} has the $a$--uniformly approximated expression
\begin{equation*}\label{eq.lambda1 approximate}
\begin{split}
M_{b,2}=&-\frac{\epsilon}{c_{0}}(u_{b}^{1}-\gamma w_{b}) \left( \int_{-\infty}^{\infty}u'_{b}(z)e^{-c_{0}F^{2}(w_{b})z}F(w_{b})u_{b}(z)dz\right. \\
&\ \ \ \ \ \ \ \qquad \left. +c_{0}F_{w}(w_{b}) \int_{-\infty}^{\infty}(u'_{b}(z))^2e^{-c_{0}F^{2}(w_{b})z}dz \right)+O\left(\epsilon^{2}|\log\epsilon|\right).
\end{split}
\end{equation*}
Especially, we have $M_{b,2}\geq {\epsilon}/{k_{2}}$ for some $k_{2}>1$, independent of $a$ and $\epsilon$.
\end{proposition}

\begin{proof}
The back solution $\phi_{b}(\xi)$ to system (3.5) converges to the $(u_{b}^{1},0)^{T}$ as $\xi\rightarrow -\infty$ with  the exponential rate
$\sqrt{\frac{k}{2}}F(w_{b})U_{2}(w_{b})$, where $$u_{b}^{1}\triangleq U_{2}(w_{b})=\frac{1+a}{2}+\frac{1}{2}\sqrt{(1-a)^{2}-\frac{4w_{b}}{k}}.$$
Combining this with Theorem \ref{the.expression}$(i)$, the condition \eqref{eq.nu bound} and $c=c_{0}+O(\epsilon)$, we obtain the estimation
\begin{equation*}
\begin{split}
w'_{a,\epsilon}(Z_{a,\epsilon-L_{\epsilon}})&=\frac{\epsilon}{c}(u_{a,\epsilon}(Z_{a,\epsilon-L_{\epsilon}})-\gamma w_{a,\epsilon}(Z_{a,\epsilon-L_{\epsilon}})) \\
&=\frac{\epsilon}{c_{0}}(u_{b}(-L_{\epsilon})-\gamma w_{b})+O(\epsilon^{2}|\log\epsilon|)  \\
&=\frac{\epsilon}{c_{0}}(u_{b}^{1}-\gamma w_{b})+O(\epsilon^{2}|\log\epsilon|).
\end{split}
\end{equation*}
This follows that
\begin{equation*}
\begin{split}
M_{b,2}&=\left\langle \left(\begin{array}{c}
 \displaystyle                  e^{c_{0}F^{2}(w_{b})L_{\epsilon}}v_{b}'(-L_{\epsilon}) \\
 \displaystyle                    -e^{c_{0}F^{2}(w_{b})L_{\epsilon}}u_{b}'(-L_{\epsilon}) \\
 \displaystyle                    \int_{\infty}^{-L_{\epsilon}} u'_{b}(z)e^{-c_{0}F^{2}(w_{b})z}\Delta_{b}(z)dz
                 \end{array}\right),
                 \left(\begin{array}{c}
  \displaystyle                   u_{a,\epsilon}'(Z_{a,\epsilon}-L_{\epsilon}) \\
  \displaystyle                   v_{a,\epsilon}'(Z_{a,\epsilon}-L_{\epsilon}) \\
  \displaystyle                   w_{a,\epsilon}'(Z_{a,\epsilon}-L_{\epsilon})
                 \end{array}\right)
  \right\rangle \\
&=\frac{\epsilon}{c_{0}}\left(u_{b}^{1}-\gamma w_{b}\right)+\int_{\infty}^{-\infty}u'_{b}(z) e^{-c_{0}F^{2}(w_{b})z}\left(F(w_{b})u_{b}(z)+\frac{2F_{w}(w_{b})}{F^{2}(w_{b})}u_{b}''(z)\right)dz  \\
&\qquad \ \ \ \ \ \ \ \ ~~+O(\epsilon^{2}|\log\epsilon|)\\
&=-\frac{\epsilon}{c_{0}}(u_{b}^{1}-\gamma w_{b})\left(\int_{-\infty}^{+\infty}u'_{b}(z) e^{-c_{0}F^{2}(w_{b})z}F(w_{b})u_{b}(z)dz \right. \\
&\qquad \ \ \ \ \ \ \ \  \left. ~+c_{0}F_{w}(w_{b})\int_{-\infty}^{+\infty}(u'_{b}(z))^{2}e^{-c_{0}F^{2}(w_{b})z}dz\right)+O(\epsilon^{2}|\log\epsilon|).
\end{split}
\end{equation*}
Recall that $u_{b}^{1}-\gamma w_{b}>0$, $u_{b}(z)>0$, $u'_{b}(z)=v_{b}(z)<0$ and
$$
F_{w}(w_{b})=-\frac{1}{2c_{1}}\frac{1}{\sqrt{(1+\frac{M}{2c_{1}})^{2}-\frac{2}{c_{1}}w_{b}}}<0.
$$
It holds clearly that $M_{b,2}\geq {\epsilon}/{k_{2}}$ for some $k_{2}>1$, independent of $a$ and $\epsilon$.
\end{proof}

\subsection{ The Region $R_{2}(\delta,\widetilde{M})$}
\hspace{0.13in}

The purpose of this section is to prove that the region $R_{2}(\delta,\widetilde{M})$ does not contain any eigenvalue of system \eqref{eq.shift} for any $\widetilde{M}>0$  and each $\delta>0$ sufficiently small.

As mentioned in the previous sections, our method is to prove that system \eqref{eq.shift} allows exponential dichotomies on each
of the intervals $I_{f},~I_{r},~I_{b}$ and $I_{l}$, which together form a partition of the entire real line $\mathbb{R}$.
{Recall that system \eqref{eq.shift} allows the exponential dichotomies on $I_{r}$ and $I_{l}$  by Proposition \ref{pro.2}.}
Using the roughness results, the exponential dichotomies of the reduced eigenvalue problem generate the exponential dichotomies of system \eqref{eq.shift} on $I_{f}$ and $I_{b}$. Our plan is to compare the projections of the above exponential dichotomies at the endpoints of the intervals $I_{f},~I_{r},~I_{b}$ and $I_{l}$. The resulting estimates conclude that for $\lambda\in R_{2}$, any exponential localized solution of system \eqref{eq.shift} must be trivial.

%As described in previous sections, our approach is to show that system \eqref{eq.shift} admits exponential dichotomies on each
%of the intervals $I_{f},~I_{r},~I_{b}$ and $I_{l}$, which together form a partition of the whole real line $\mathbb{R}$. The exponential dichotomies on $I_{r}$ and $I_{l}$ are yet established in Proposition \ref{pro.2}.
%The exponential dichotomies on $I_{f}$ and $I_{b}$  are generated from exponential dichotomies of a reduced eigenvalue problem via roughness results. Our plan is to compare the projections of the aforementioned exponential dichotomies at the endpoints of the intervals. The obtained estimates yield that any exponentially localized solution to \eqref{eq.shift}  must be trivial for $\lambda\in R_{2}$.

\subsubsection{ A reduced eigenvalue problem}\label{subsec.R-2}
\hspace{0.13in}

Similar to Section \ref{section.Reduced Eigenvalue Problem 1}, we obtain a reduced eigenvalue problem by setting $\epsilon$ to $0$ in
system \eqref{eq.shift} for $\xi$ in $I_{f}$ or $I_{b}$ and $\lambda \in R_{2}.$  Thus, the reduced eigenvalue problem is of the form
\begin{equation}\label{eq.refuced eigenvalue in R2}
\varphi'=A_{j}(\xi,\lambda)\varphi, ~j=f,~b,
\end{equation}
where
\begin{eqnarray*}\label{eq.j=f,b matrx in R2}
\aligned
A_{j}(\xi,\lambda)= A_{j}(\xi,\lambda;a)
 =\left(
                                            \begin{array}{ccc}
                                              -\eta & F(w_{j}) & 0 \\
                                              F(w_{j})\left(f_{u}(u_{j},w_{j})+\lambda\right) & c_{0}F^{2}(w_{j})-\eta & \Delta_{2,j} \\
                                              0 & 0 & -\frac{\lambda}{c_{0}}-\eta \\
                                            \end{array}
                                          \right),
\endaligned
\end{eqnarray*}
and
$$\Delta_{2,j}=F(w_{j})f_{w}(u_{j},w_{j})+2\frac{F_{w}(w_{j})}{F^{2}(w_{j})}u_{j}''-\frac{\lambda}{c_{0}F_{w}(w_{j})}u_{j}',~j=f,~b.$$
Here $u_{j}(\xi)$ denotes the $u$--component of $\phi_{j}$, $a\in[0,\frac{1}{2}-k_{1}]$ and $\lambda \in R_{2}(\delta,\widetilde{M})$.
{By the particular structure of the coefficient matrix $A_{j}(\xi,\lambda)$, the linear differential system \eqref{eq.refuced eigenvalue in R2} admits an invariant subspace $\mathbb{C}^2\times \{0\}\subset \mathbb{C}^3$ on which the dynamics are given by}
\begin{eqnarray}\label{eq.fast subsystem's eigenvalue in R2}
\aligned
\psi'=C_{j}(\xi,\lambda)\psi,~\quad j=f,~b,
\endaligned
\end{eqnarray}
with
\[
C_{j}(\xi,\lambda)=\left(
                                              \begin{array}{ccc}
                                                -\eta & F(w_{j}) \\
                                                F(w_{j})\left(f_{u}(u_{j},w_{j})+\lambda\right) & c_{0}F^{2}(w_{j})-\eta \\
                                              \end{array}
                                              \right).
\]

Next, we will show that systems \eqref{eq.refuced eigenvalue in R2} and \eqref{eq.fast subsystem's eigenvalue in R2} admit exponential dichotomies on both the half--lines.
The equation for $u_{f}$ is
\begin{equation*}\label{eq.uf}
\begin{array}{ll}
u_{f}'=F(0)v_{f},\\
v_{f}'=c_{0}F^{2}(0)v_{f}+f(u_{f},0)F(0).
\end{array}
\end{equation*}
The reduced linear eigenvalue problem along the front $u_{f}$ is given by
\begin{equation*}
\begin{array}{ll}
\left(\begin{array}{c}
        p \\
        q
      \end{array}
\right)'=\left(
              \begin{array}{ccc}
                  0 & F(0) \\
                  F(0)\left(f_{u}(u_{f},0)+\lambda \right) & c_{0}F^{2}(0)  \\
              \end{array}
           \right)
\left(\begin{array}{c}
        p \\
        q
      \end{array}
\right),
\end{array}
\end{equation*}
and the associated shifted eigenvalue problem is
\begin{equation*}\label{eq.uf shift eigenvalue}
\begin{array}{ll}
\left(\begin{array}{c}
        p \\
        q
      \end{array}
\right)'=\left(
              \begin{array}{ccc}
                  -\eta & F(0) \\
                  F(0)(f_{u}(u_{f},0)+\lambda) & c_{0}F^{2}(0)-\eta  \\
              \end{array}
           \right)
\left(\begin{array}{c}
        p \\
        q
      \end{array}
\right).
\end{array}
\end{equation*}
Analogously, the reduced equation along the back is given by
\begin{equation*}\label{eq.ub}
\begin{split}%{ll}
u_{b}'&=F(0)v_{b},\\
v_{b}'&=c_{0}F^{2}(w_{b})v_{b}+f(u_{b},w_{b})F(w_{b}).
\end{split}
\end{equation*}
and the associated shifted eigenvalue problem is
\begin{equation*}\label{eq.ub shift eigenvalue}
\begin{array}{ll}
\left(\begin{array}{c}
        p \\
        q
      \end{array}
\right)'=\left(
              \begin{array}{ccc}
                  -\eta & F(w_{b}) \\
                  F(w_{b})\left(f_{u}(u_{b},w_{b})+\lambda \right) & c_{0}F^{2}(w_{b})-\eta  \\
              \end{array}
           \right)
\left(\begin{array}{c}
        p \\
        q
      \end{array}
\right).
\end{array}
\end{equation*}
So $e^{-\eta\xi}\phi_{j}'(\xi)$ is the exponential localized solution of system \eqref{eq.fast subsystem's eigenvalue in R2} at $\lambda=0$ and it has no zeros. According to Sturm--Liouville Theorem \ref{th.SL}, the eigenvalues of system \eqref{eq.fast subsystem's eigenvalue in R2} are finite and simple.
Moreover, $\lambda=0$ is the maximum eigenvalue. Thus, system \eqref{eq.fast subsystem's eigenvalue in R2} admits no exponentially localized solutions for $\lambda \in R_{2}(\delta,\widetilde{M})$ with $\delta>0$ sufficiently small. Then system \eqref{eq.fast subsystem's eigenvalue in R2}  admits exponential dichotomy on $\mathbb{R}$. This is the content of the following proposition.

\begin{proposition}\label{pro.exponential dichotomies on R}
Let $k_{1},~\widetilde{M}>0$. For each $\delta>0$ sufficiently small, $a\in[0,\frac{1}{2}-k_{1}]$
and $\lambda\in R_{2}(\delta,\widetilde{M})$ system \eqref{eq.refuced eigenvalue in R2} admits exponential dichotomies on $\mathbb{R}$ with $\lambda$-- and $a$--independent constants $C,~ {\mu}/{2}>0$.
\end{proposition}
\begin{proof}
{According to Lemma \ref{lem.1}, it follows that  the asymptotic matrices $$C_{j,\pm\infty}(\lambda)=C_{j,\pm\infty}(\lambda;a):=\lim_{\xi\rightarrow\pm\infty}C_{j}(\xi,\lambda)$$ of system \eqref{eq.fast subsystem's eigenvalue in R2} admit a uniform spectral gap larger than $\mu>0$ for $a\in[0,\frac{1}{2}-k_{1}]$ and $\lambda\in R_{2}(\delta,\widetilde{M})$. So system \eqref{eq.fast subsystem's eigenvalue in R2} admits  exponential dichotomies on both the half--lines with constants $C,~\mu>0$ and projections $\Pi^{u,s}_{j,\pm}(\xi,\lambda)=\Pi^{u,s}_{j,\pm}(\xi,\lambda;a),~j=f,~b$ by Theorems \ref{th.dichotomy} and \ref{th.dichotomy intergral}.
Since $R_{2}(\delta,\widetilde{M})\times [0,\frac{1}{2}-k_{1}]$ is compact, the constant $C>0$ can be chosen independent of $\lambda$ and $a$.}

{Since $\lambda=0$ is the maximum eigenvalue of the coefficient matrix associated to system \eqref{eq.fast subsystem's eigenvalue in R2}, this linear differential system admits no bounded solutions for $\lambda\in R_{2}(\delta,\widetilde{M})$.
According to \cite[p.16--19]{1978book}, we can paste the exponential dichotomies  by defining $\Pi_{j}^{s}(0,\lambda)$ to be the projection onto ${\rm R}(\Pi_{j,+}^{s}(0,\lambda))$ along ${\rm R}(\Pi_{j,-}^{u}(0,\lambda))$.
Thus, there exists an exponential dichotomy for $(\lambda,a)\in R_{2}(\delta,\widetilde{M})\times [0,\frac{1}{2}-k_{1}]$ to system \eqref{eq.fast subsystem's eigenvalue in R2} on $\mathbb{R}$ with $\lambda$-- and $a$--independent constants $C,~\mu>0$ and projections $\Pi_{j}^{u,s}(\xi,\lambda)=\Pi_{j}^{u,s}(\xi,\lambda;a),~j=f,~b$.}

{Similar to the proof of Proposition \ref{pro.4},  by the variation of constants formula the exponential dichotomy of the subsystem \eqref{eq.fast subsystem's eigenvalue in R2} on $\mathbb{R}$ can be transferred to the full system \eqref{eq.refuced eigenvalue in R2}.
By taking $\delta>0$ sufficiently small, the exponential dichotomy on $\mathbb{R}$ of system \eqref{eq.refuced eigenvalue in R2} has the constant $C$ independent of $a$ and $\lambda$ and the constant $\min\{\mu,\eta-\frac{\delta}{c_{0}}\}\geq\frac{\mu}{2}$.}
\end{proof}

Next, we will show that the shifted eigenvalue system \eqref{eq.shift} admits no nontrivial exponentially localized solution in the region $R_{2}(\delta,\widetilde{M})$, namely, the operator $\mathcal{L}_{a,\epsilon}$ admits no spectrum in the region $R_{2}(\delta,\widetilde{M})$.

\subsection{Absence of point spectrum in $R_{2}(\delta,\widetilde{M})$}
\hspace{0.13in}

For proving our results we need the next result, which provides an estimation on projections of the evolution operator of a linear differential system in case of exponential dichotomy.

\begin{lemma}$($\cite[{\bf Lemma}~6.19]{CRS}$)$ \label{lem.CRS 2016}
Let $n\in \mathbb{N},~a,~b\in\mathbb{R}$ with $a<b$ and $A\in C([a,b],~\mathbf{Mat}_{n\times n}(\mathbb{C}))$. Suppose that the linear differential system
\begin{equation}\label{eq.CRS 2016 lemma}
\varphi_{x}=A(x)\varphi,
\end{equation}
has an exponential dichotomy on $[a,b]$ with constants $C,m>0$ and projections
$P_{1}^{u,s}(x)$. Denote by $T(x,y)$ the evolution of system \eqref{eq.CRS 2016 lemma}. Let $P_{2}$ be a projection such that
$\|P_{1}^{s}(b)-P_{2}\|\leq \delta_{0}$ for some $\delta_{0}>0$, and let $\nu \in \mathbb{C}^{n}$ be a vector such that $\|P_{1}^{s}(a)\nu\|\leq k\|P_{1}^{u}(a)\nu\|$ for some $k\geq 0$. If $\delta_{0}(1+kC^{2}e^{-2m(b-a)})<1$, then it holds
$$\|P_{2}T(b,a)\nu\|\leq\frac{\delta_{0}+kC^{2}e^{-2m(b-a)}(1+\delta_{0})}{1-\delta_{0}(1+kC^{2}e^{-2m(b-a)})}\|(1-P_{2})T(b,a)\nu\|.$$
\end{lemma}
{ Recall again that $1$ denotes the unit matrix or the identity operator in case of no confusion.}

\begin{proposition}\label{pro.no exponential localized in R2}
Let $k_{1},\widetilde{M}>0$. For each $\delta>0$ sufficiently small, $a\in[0,\frac{1}{2}-k_{1}]$
and $\lambda\in R_{2}(\delta,\widetilde{M})$, system \eqref{eq.shift} admits no nontrivial exponentially localized solution.
\end{proposition}

\begin{proof}
By Theorem \ref{the.expression}$(i)$, there exists an $\epsilon_{0}>0$ such that for $\epsilon\in(0,\epsilon_{0})$ one has the next estimation
\begin{equation}\label{eq.estimate difference}
\begin{split}
\|A(\xi,\lambda)-A_{f}(\xi,\lambda)  \|   & \leq C\epsilon|\log\epsilon|,\ \ \ ~\xi\in(-\infty,L_{\epsilon}], \\
\|A(Z_{a,\epsilon}+\xi,\lambda)-A_{b}(\xi,\lambda)\|  &\leq C\epsilon|\log\epsilon|,\ \ \ ~\xi\in[-L_{\epsilon},L_{\epsilon}],
\end{split}
\end{equation}
with $C>0$ a constant independent of $\lambda,~a,~\epsilon$.
{From Proposition \ref{pro.exponential dichotomies on R}, there exists an exponential dichotomy to system \eqref{eq.refuced eigenvalue in R2} on $\mathbb{R}$ with $\lambda$-- and $a$-- independent constants $C,~{\mu}/{2}>0$ and projections $Q_{j}^{u,s}(\xi,\lambda)=Q_{j}^{u,s}(\xi,\lambda;a),~j=f,~b.$
The spectral projection onto the (un)stable eigenspace of the asymptotic matrices $A_{j,\pm}(\lambda)=A_{j,\pm}(\lambda;a)$ of system \eqref{eq.refuced eigenvalue in R2} is denoted by $P_{j,\pm}^{u,s}(\lambda)=P_{j,\pm}^{u,s}(\lambda;a)$.}
As in the proof of Proposition \ref{pro.6}, it follows that
\begin{equation}\label{eq.6.5}
\begin{split}
\left\|Q_{f}^{u,s}(\pm\xi,\lambda)-P_{f,\pm}^{u,s}(\lambda)\right\| &\leq C\left(e^{-\sqrt{\frac{k}{2}}F(0)\xi}+e^{-\frac{\mu}{2}\xi}\right),\\
\left\|Q_{b}^{u,s}(\pm\xi,\lambda)-P_{b,\pm}^{u,s}(\lambda)\right\| &\leq C\left(e^{-\sqrt{\frac{k}{2}}F(w_{b})U_{2}(w_{b})\xi}+e^{-\frac{\mu}{2}\xi}\right)
\end{split}
\end{equation}
for $\xi\geq 0$. By the estimations in \eqref{eq.estimate difference} and Theorem \ref{th.dichotomy perturbation}, the shifted eigenvalue problem \eqref{eq.shift} admits exponential dichotomies on $I_{f}=(-\infty,L_{\epsilon}]$ and on $I_{b}=[Z_{a,\epsilon}-L_{\epsilon},Z_{a,\epsilon}+L_{\epsilon}]$  with $\lambda$-- and $a$--independent constants $C,{\mu}/{2}>0$ and projections $\mathcal{Q}_{j}^{u,s}(\xi,\lambda)=\mathcal{Q}_{j}^{u,s}(\xi,\lambda;a,\epsilon),~j=f,~b$, which satisfy
\begin{equation}\label{eq.6.6}
\begin{split}
\left\|\mathcal{Q}_{f}^{u,s}(\xi,\lambda)-Q_{f}^{u,s}(\xi,\lambda)\right\|  &\leq C\epsilon|\log\epsilon|,\\
\left\|\mathcal{Q}_{b}^{u,s}(Z_{a,\epsilon}+\xi,\lambda)-Q_{b}^{u,s}(\xi,\lambda)\right\| &\leq C\epsilon|\log\epsilon|,
\end{split} \ \ \ \ \ \  ~|\xi|\leq L_{\epsilon}.
\end{equation}
Simultaneously, by Proposition \ref{pro.2}, system \eqref{eq.shift} admits exponential dichotomies on $I_{r}=[L_{\epsilon},Z_{a,\epsilon}-L_{\epsilon}]$ and on $I_{l}=[Z_{a,\epsilon}+L_{\epsilon},\infty)$  with $\lambda$-- and $a$--independent constants $C,\mu>0$ and projections $\mathcal{Q}_{r,l}^{u,s}(\xi,\lambda)=\mathcal{Q}_{r,l}^{u,s}(\xi,\lambda;a,\epsilon)$, which satisfy
\begin{equation}\label{eq.6.7}
%\begin{split}
\|(\mathcal{Q}_{r}^{s}-\mathcal{P})(L_{\epsilon},\lambda)\|,~
\|(\mathcal{Q}_{r}^{s}-\mathcal{P})(Z_{a,\epsilon}-L_{\epsilon},\lambda)\|,~
\|(\mathcal{Q}_{l}^{s}-\mathcal{P})(Z_{a,\epsilon}+L_{\epsilon},\lambda)\|
\leq C\epsilon |\log\epsilon|,
%\end{split}
\end{equation}
where $\mathcal{P}(\xi,\lambda)=\mathcal{P}(\xi,\lambda;a,\epsilon)$ is the spectral projection onto the stable
eigenspace of $A(\xi,\lambda)$.

Since $A_{f}(\xi,\lambda)$ (resp. $A_{b}(\xi,\lambda)$) converges at the exponential rate $\sqrt{\frac{k}{2}}F(0)$ (resp. $\sqrt{\frac{k}{2}}F(w_{b})U_{2}(w_{b})$) to the asymptotic
matrix $A_{f,\pm}(\lambda)$ (resp. $A_{b,\pm}(\lambda)$) as $\xi\rightarrow\pm\infty$. Combining this with \eqref{eq.estimate difference} and \eqref{eq.nu bound}, we obtain
\[
\|A(L_{\epsilon},\lambda)-A_{f,+}(\lambda)\|,~\|A(Z_{a,\epsilon}\pm L_{\epsilon},\lambda)-A_{b,\pm}(\lambda)\|\leq C\epsilon|\log\epsilon|.
\]
{By continuity, the spectral projections associated with these matrices admit the same bound, namely,}
\begin{equation}\label{eq.6.7.1}
\begin{split}
\|(\mathcal{P}-P_{f,+}^{s})(L_{\epsilon},\lambda)\|,& ~\|[(1-\mathcal{P})-P_{f,+}^{u}](L_{\epsilon},\lambda)\|\leq C\epsilon|\log\epsilon|,\\
\|(\mathcal{P}(Z_{a,\epsilon}\pm L_{\epsilon},\lambda)-P_{f,\pm}^{s}(\lambda)\|, & ~\|(1-\mathcal{P})(Z_{a,\epsilon}\pm L_{\epsilon},\lambda)-P_{b,\pm}^{u}(\lambda)\|\leq C\epsilon|\log\epsilon|.
\end{split}
\end{equation}
By \eqref{eq.6.5}, \eqref{eq.6.6}, \eqref{eq.6.7}, \eqref{eq.6.7.1} and {\eqref{eq.nu bound}}, it holds
\begin{equation}\label{eq.6.8}
\begin{split}
\|(\mathcal{Q}_{r}^{u,s}-\mathcal{Q}_{f}^{u,s})(L_{\epsilon},\lambda)\|&\leq C\epsilon|\log\epsilon|,\\
\|(\mathcal{Q}_{l}^{u,s}-\mathcal{Q}_{b}^{u,s})(Z_{a,\epsilon}+L_{\epsilon},\lambda)\|, \ &
\|(\mathcal{Q}_{r}^{u,s}-\mathcal{Q}_{b}^{u,s})(Z_{a,\epsilon}-L_{\epsilon},\lambda)\|\leq C\epsilon|\log\epsilon|.
\end{split}
\end{equation}
{Denote by $\varphi(\xi)$ an exponentially localized solution of system \eqref{eq.shift} at some $\lambda\in  R_{2}(\delta,\widetilde{M})$, and it holds $\mathcal{Q}_{f}^{s}(0,\lambda)\varphi(0)=0$.}
Combining \eqref{eq.6.8} with $\nu\geq 2/\mu$, together with Lemma \ref{lem.CRS 2016}, we obtain
\begin{equation}\label{eq.6.9}
%\begin{array}{ll}
\|\mathcal{Q}_{r}^{s}(L_{\epsilon},\lambda)\varphi(L_{\epsilon})\|\leq C\epsilon|\log\epsilon| \|\mathcal{Q}_{r}^{u}(L_{\epsilon},\lambda)\varphi(L_{\epsilon})\|.
%\end{array}
\end{equation}
At the endpoint $Z_{a,\epsilon}-L_{\epsilon}$, applying Lemma \ref{lem.CRS 2016} to the inequality \eqref{eq.6.9} and using \eqref{eq.6.8} we get a similar inequality as that in \eqref{eq.6.9}
\begin{equation*}\label{eq.6.9.1}
%\begin{array}{ll}
\|\mathcal{Q}_{b}^{s}(Z_{a,\epsilon}-L_{\epsilon},\lambda)\varphi(Z_{a,\epsilon}-L_{\epsilon})\|\leq C\epsilon|\log\epsilon| \|\mathcal{Q}_{b}^{u}(Z_{a,\epsilon}-L_{\epsilon},\lambda)\varphi(Z_{a,\epsilon}-L_{\epsilon})\|.
%\end{array}
\end{equation*}
Applying Lemma \ref{lem.CRS 2016} again yields
\begin{equation*}\label{eq.6.10}
\begin{array}{ll}
\|\mathcal{Q}_{l}^{s}(Z_{a,\epsilon}+L_{\epsilon},\lambda)\varphi(Z_{a,\epsilon}+L_{\epsilon})\|&\leq C\epsilon|\log\epsilon| \|\mathcal{Q}_{l}^{u}(Z_{a,\epsilon}+L_{\epsilon},\lambda)\varphi(Z_{a,\epsilon}+L_{\epsilon})\| \\
&=0.
\end{array}
\end{equation*}
This proves that $\varphi(\xi)$ is the trivial solution of system \eqref{eq.shift}.
\end{proof}

Now we have enough preparations to prove our main results.

\section{Proof of the main results}\label{sx6}

\subsection{Proof of Theorem \ref{the.stability}}%\label{sx61}
%\hspace{0.13in}

%\begin{proof}
In the regime $a\in(0,1/2)$, the essential spectrum of $\mathcal{L}_{a,\epsilon}$ is contained in the half plane $\left\{\lambda\in\mathbb C|\ \mbox{\rm Re}(\lambda)\leq \max\{-\epsilon \gamma,-ka\}\right\}$ by Proposition \ref{pro.ess}.
{From subsection 5.1 and Proposition \ref{pro.no exponential localized in R2} the regions $R_{2}$ and $R_{3}$  do not intersect the point spectrum of $\mathcal{L}_{a,\epsilon}$.
By Theorem \ref{the.lambda approximate} and Proposition \ref{pro.lambda0 simple}, the point spectrum in $R_{1}(\delta)$ to the right--hand side of the essential spectrum admits at most two eigenvalues. One of the eigenvalues is the simple translational eigenvalue $\lambda_{0}=0$ and the other real eigenvalue is $\lambda_{1}$, which is approximated by $-M_{b,2}M_{b,1}^{-1}$, where $M_{b,1}>0$ is $\epsilon$--independent and bounded by an $a$--independent constant.}
Finally, applying Proposition \ref{pro.lambda1 approximate} to estimate
$M_{b,2}$ arrives the conclusion that there exists a constant $b_{0}>0$ such that $\lambda_{1}\leq -\epsilon b_{0}$. This finishes the proof of the theorem.
\qed
%\end{proof}

\subsection{Proof of Theorem \ref{the.stability for a=0}}\label{sx62}

The proof follows directly from Propositions \ref{pro.ess} and \ref{pro.1.a=0}, subsection 5.1 and the fact that the shifted eigenvalue system \eqref{eq.shift} does not have an eigenvalue in the region $\Omega_{+}$. We are done.
\qed

\section{Appendix }\label{sxa}
\hspace{0.13in}

\subsection{Appendix A: The proof of Theorem \ref{the.expression}}
\hspace{0.13in}

For $a\in (0,\ 1/2)$, by Fenichel theory \cite{Fen71,Fen74,Fen79}, the segments $S_{0}^{r}$ and $L_{0}$ persist for sufficiently small $\epsilon>0$ as locally invariant manifolds $S_{\epsilon}^{r}$ and $L_{\epsilon}$. { When $a=0$, this argument holds too,  which was obtained by Shen and Zhang \cite{SZ} combining Fenichel's three theorems and the center manifold theorem. In both cases $L_\epsilon$ approaches the origin as $\xi \rightarrow +\infty$.}  In addition, the center--stable manifold $W^{s}(S_{0}^{r})$ and the  center--unstable manifold $W^{u}(S_{0}^{r})$
persist as locally invariant center--stable and center--unstable manifolds $W^{s}(S_{\epsilon}^{r})$ and $W^{u}(S_{\epsilon}^{r})$, respectively. Similarity, the center--stable manifold $W^{s}(L_{0})$ and the center--unstable manifold $W^{u}(L_{0})$
persist as locally invariant center--stable and center--unstable manifolds $W^{s}(L_{\epsilon})$ and $W^{u}(L_{\epsilon})$, respectively.

For any $r\in \mathbb{N}$, there exists a $C^{r}$--change of coordinates $\Psi_{\epsilon}: \mathcal{U}\rightarrow \mathbb{R}^3$, with $\mathcal{U}$ an $\epsilon$--independent open neighborhood of $S_{\epsilon}^{r}$, such that the flow under this new coordinates is given by the Fenichel normal form system \cite{Fen79,Jones95} (with notation abuse)
\begin{equation}\label{eq.Fenichel normal form}
\begin{split}
U'&=-\Lambda(U,V,W;c,a,\epsilon)U,\\
V'&=\Gamma(U,V,W;c,a,\epsilon)V,\\
W'&=\epsilon(1+H(U,V,W;c,a,\epsilon)UV).
\end{split}
\end{equation}
Here the functions $\Lambda,~\Gamma$ and $H$ are $C^{r}$, and $\Lambda$ and $\Gamma$ are bounded below away from zero.
The slow manifold $S_{\epsilon}^{r}$ is represented by $U=V=0$, and the manifolds $W^{u}(S_{\epsilon}^{r})$ and $W^{s}(S_{\epsilon}^{r})$
are given by $U=0$ and $V=0$, respectively.  Consider the Fenichel neighborhood $\Psi_{\epsilon}(\mathcal{U})$, which contains the  box
$$
\left\{(U,V,W){\in \mathbb R^3}|~U,V\in [-\Theta,\Theta],~W\in [-\Theta,W^{*}+\Theta] \right \}
$$
for $W^{*}>0,~0<\Theta\ll W^{*}$, both independent of $\epsilon$. And then, we define two {manifolds}
\begin{equation*}\label{entry,exit manifolds}
\begin{split}
N_{1}&:=\left\{(U,V,W)\in \mathbb R^3|\ ~U=\Theta,~V\in [-\Theta,\Theta],~W\in [-\Theta,\Theta]  \right\},\\
{N_{2}}&:=\left\{(U,V,W)\in \mathbb R^3|\ ~U,V\in [-\Theta,\Theta],~W=W_{0}  \right\},
\end{split}
\end{equation*}
with the flow of the Fenichel normal form system entering in $N_1$ and exiting from $N_2$, called \textit{entry manifold} and \textit{exit manifold}, where $0<W_{0} <W^{*}$. The next result will also be used in the proof of Theorem \ref{the.expression}.

\begin{theorem}\label{th.Eszter}$($\cite{CRS,Eszter,HDK}$)$
Assume that
\begin{itemize}
\item $\Xi(\epsilon)$ is a continuous function of $\epsilon$ satisfying
\begin{equation}\label{eq.Xi epsilon}
\begin{array}{ll}
\lim\limits_{\epsilon\rightarrow 0}\Xi(\epsilon)=\infty,~\lim\limits_{\epsilon\rightarrow 0}\epsilon\Xi(\epsilon)=0.
\end{array}
\end{equation}
\item there is a one--parameter family of solutions $(U,V,W)(\xi,\cdot)$
to the Fenichel normal form system \eqref{eq.Fenichel normal form} with $$(U,V,W)(\xi_{1},\epsilon)\in N_{1},~(U,V,W)(\xi_{2}(\epsilon),\epsilon)\in N_{2}$$
and $\lim\limits_{\epsilon\rightarrow 0}W(\xi_{1},\epsilon)=0$ for some ${ \xi_{1},~\xi_{2}(\epsilon)}\in \mathbb{R}$.
\end{itemize}
Let $U_{0}(\xi)$ be the solution to the system
\begin{equation}\label{eq.U}
\begin{array}{ll}
U'=-\Lambda(U,0,0;c,a,0)U,
\end{array}
\end{equation}
satisfying $U_{0}(\xi_{1})=\Theta+\widetilde{U}_{0}$ with $|\widetilde{U}_{0}|\ll \Theta$. Then, for $0<\epsilon\ll 1$, there holds the next estimation
\begin{equation*}\label{bound.epsilon Xi epsilon}
\begin{array}{ll}
\| (U,V,W)(\xi,\epsilon)-(U_{0}(\xi),0,0) \|\leq C\left(\epsilon\Xi(\epsilon)+|\widetilde{U}_{0}|+|W(\xi_{1},\epsilon)|\right),
\end{array}
\end{equation*}
for $\xi\in[\xi_{1},\Xi(\epsilon)]$, where $C>0$ is a constant independent of $a$ and $\epsilon$.
\end{theorem}

Having the above preparation, we can prove Theorem \ref{the.expression}.

\begin{proof}
Recall that $\Xi_{\tau}(\epsilon):=-\tau \log\epsilon$, for every $\tau>0$, satisfies the condition \eqref{eq.Xi epsilon}.

\noindent $(i)$.  By the geometric singular perturbation theory, we know that the solution $\phi_{a,\epsilon}(\xi)$ is $a$--uniformly $O(\epsilon)$--close to $(\phi_{f}(\xi),0)$ upon entry in $N_{1}$ at $\xi_{f}=O(1)$. Since $(\phi_{f}(\xi),0)$ exponentially converges to $(1,0,0)$ for $\epsilon=0$, it must be located in $W^{s}(S_{0}^{r})$. Thus, it holds that $\Psi_{0}(\phi_{f}(\xi),0)=(U_{0}(\xi),0,0)$, where $U_{0}(\xi)$ is a solution to system \eqref{eq.U}. Let $\Psi_{\epsilon}(\phi_{a,\epsilon}(\xi))=(U_{a,\epsilon}(\xi),V_{a,\epsilon}(\xi),W_{a,\epsilon}(\xi))$. Then we derive
$$\| (U_{a,\epsilon}(\xi),V_{a,\epsilon}(\xi),W_{a,\epsilon}(\xi))-(U_{0}(\xi),0,0) \|\leq C\epsilon \Xi_{\tau}(\epsilon)$$
for $\xi\in[\xi_{f},\Xi_{\tau}(\epsilon)]$ by Theorem \ref{th.Eszter}. Since the transformation $\Psi_{\epsilon}$ for getting the Fenichel normal form is $C^{r}$--smooth in $\epsilon$, when transforming back to the $(u,v,w)$--coordinates there incur at most $O(\epsilon)$ error.
Therefore, $\phi_{a,\epsilon}(\xi)$ is $a$--uniform $O(\epsilon\Xi_{\tau}(\epsilon))$--close to $(\phi_{f}(\xi),0)$ for $\xi\in[\xi_{f},\Xi_{\tau}(\epsilon)]$. This proves the first estimation of statement $(i)$.

Again by the geometric singular perturbation theory, the pulse solution $\phi_{a,\epsilon}(\xi)$ arrives a neighborhood of the slow manifold $S_{\epsilon}^{r}$, where the flow spends the time being of order $\epsilon^{-1}$.
{Let $Z_{a,\epsilon}=O(\epsilon^{-1})$ be the leading order of the time at which the pulse solution along the back exits the Fenichel neighborhood
$\mathcal{U}$ of  $S_{\epsilon}^{r}$. Using the similar manner as treating the flow in a neighborhood of the left slow
manifold $L_{\epsilon}$,} we can obtain the second estimation in $(i)$ by using similar arguments
as those in the proof of the first estimation.

{\noindent $(ii)$. Taking the $a$-- and $\epsilon$--independent neighborhood $\mathcal{U}$ smaller if necessary
and choosing the $a$-- and $\epsilon$--independent $\xi_{0}$ sufficiently large, there holds that $\phi_{a,\epsilon}(\xi)$ lies in the region $\mathcal{U}$ for
$\xi\in[\xi_{0},Z_{a,\epsilon}-\xi_{0}]$.} This verifies the estimation in $(ii)$ along the right branch $S_{0}^{r}$.

\noindent $(iii)$. The estimation could be proved using similar arguments as those in the proof of $(ii)$ along the left branch $L_0$.

It completes the proof of Theorem \ref{the.expression}.
\end{proof}

\subsection{Appendix B}
\hspace{0.13in}

For readers' convenience we recall here some results on exponential dichotomy and trichotomy. For details, we refer readers to \cite{CRS}. We also recall a theorem on Sturm--Liouville eigenvalue problem on an infinite interval.
%Below, we provide the definitions of exponential dichotomy and trichotomy, where

The linear differential system
\begin{equation}\label{Appendix linear system}
  \varphi'=A(x)\varphi
\end{equation}
is said to have an \textit{exponential dichotomy} on $J$ if there are constants $~C,\mu>0$ and linear projections
$P^{s}(x),~P^{u}(x):~\mathbb{C}^{n}\rightarrow \mathbb{C}^{n},~x\in J$ satisfying, for all $x,~y\in J$,
%$P^2=P$ such that the following conditions hold:
\begin{itemize}
\item $P^{s}(x)+P^{u}(x)=1$,
 \item $P^{s,u}(x)T(x,y)=T(x,y)P^{s,u}(y)$,
\item $\|T(x,y)P^{s}(y)\|,~\|T(y,x)P^{u}(x)\|\leq Ce^{-\mu(x-y)}$ for $x\geq y$.
\end{itemize}
Here, $T(x,y)$ is the evolution operator of the linear system \eqref{Appendix linear system}, and the interval $J$ is typically
taken to be either $\mathbb{R}$, or $\mathbb{R_{+}}$ or $\mathbb{R_{-}}.$

Equation \eqref{Appendix linear system} has an \textit{exponential trichotomy} on $J$ with constants $C,~\mu,~\nu>0$ and
projections $P^{s}(x),~P^{u}(x),~P^{c}(x):~\mathbb{C}^{n}\rightarrow \mathbb{C}^{n},~x\in J$ if the next conditions hold for all $x,~y\in J$,
\begin{itemize}
\item $P^{s}(x)+P^{u}(x)+P^{c}(x)=1$,
 \item $P^{s,u,c}(x)T(x,y)=T(x,y)P^{s,u,c}(y)$,
\item $\|T(x,y)P^{s}(y)\|,~\|T(y,x)P^{u}(x)\|\leq Ce^{-\mu(x-y)}$ for $x\geq y$,
 \item  $\|T(x,y)P^{c}(y)\|\leq Ce^{-\nu|x-y|}$.
\end{itemize}
We often use the abbreviations $T^{s,u,c}(x,y)=T(x,y)P^{s,u,c}(y)$ to implicitly represent the corresponding projections of the dichotomy or of the trichotomy.

There are many known results on existence of exponential dichotomies. Here we only recall those we have used in the proof of our main results.

\begin{theorem}\label{th.dichotomy perturbation}$($\cite{1978book,HDK}$)$
Let $\varphi\in \mathbb{R}^{n}$ and suppose that the
linear system \eqref{Appendix linear system} %$  \varphi'=A(x)\varphi$
has an exponential dichotomy on $J$. Then the linear system
$\varphi'=(A(x)+B(x))\varphi$ also has an exponential dichotomy on $J$ provided $\|B(x)\| <\delta$ for all $x\in J$, where $\delta>0$ is chosen to be sufficiently small.
\end{theorem}

\begin{theorem}\label{th.dichotomy intergral}$($\cite{1965book,HDK}$)$
If the linear system \eqref{Appendix linear system} % $\varphi'=A(x)\varphi$
has an exponential dichotomy on $J$, and
$$\int_{J}\|B(x)\|dx <\infty,$$
then the system $\varphi'=(A+B(x))\varphi$ also has an exponential dichotomy on $J$  with the
same decay rates as those for system \eqref{Appendix linear system}.
\end{theorem}

\begin{theorem}\label{th.dichotomy}$($\cite{1978book,HDK}$)$
Suppose that
\begin{itemize}
\item the $n \times n$ matrix $A(x)$
is bounded and hyperbolic for all $x\in J$ with $k$ eigenvalues having their real parts less than
$-\alpha< 0$ and $n-k$ eigenvalues having their real parts greater than $\beta> 0$;
\item $A(x)$ is continuously differentiable and there exists a $\delta > 0$ such that $\|A(x)\|<\delta$
for all $x \in J$.
\end{itemize}
Then the linear system \eqref{Appendix linear system} %$  \varphi'=A(x)\varphi$
has an exponential dichotomy on $J$.
\end{theorem}

Finally we recall Sturm--Liouville theorem on infinite interval.
A Sturm--Liouville operator $\mathcal{L}$ is a second order differential operator of the form
$$\mathcal{L}u:=\partial^2_{x}u+a_{1}(x)\partial_{x}u+a_{0}(x)u.$$
Consider the Sturm--Liouville operator $\mathcal{L}$  acting on $H^{2}(\mathbb{R})$ with smooth coefficients
$a_{0}(x)$ and $a_{1}(x)$, which decay exponentially to constants $a_{0}^{\pm},~a_{1}^{\pm}\in \mathbb{R}$ as $x\rightarrow \pm\infty.$ Recall that $H^{2}(\mathbb{R})$ is the Hilbert space formed by second order differentiable functions defined on $\mathbb R$ with its element $u$ having the norm $\|u\|_{H^2}:=\left(\sum\limits_{i=0}\limits^{2}\|\partial_{x}^{i}u\|_{L^{2}}^{2}\right)^{\frac{1}{2}}<\infty$.

The associated eigenvalue problem
$$\mathcal{L}u=\lambda u$$
has the following well--known properties, see e.g. \cite{2013book}.

\begin{theorem}\label{th.SL}$($\cite{2013book}$)$
Consider the eigenvalue problem $\mathcal{L}u=\lambda u$ on the space $H^{2}(\mathbb{R})$,
with the coefficients of $\mathcal L$ decaying exponentially to the constants $a_{0}^{\pm},~a_{1}^{\pm}\in \mathbb{R}$ as $x\rightarrow \pm\infty.$ The following statements hold.
\begin{itemize}
\item The point spectrum of $\mathcal L$, $\sigma_{p}(\mathcal{L})$, consists of a finite number, possibly zero, of simple eigenvalues, which can be enumerated in a strictly descending order
$$\lambda_{0}>\lambda_{1}>\cdots >\lambda_{N}>b:=\max\{a_{0}^{-},a_{0}^{+}\}.$$
\item For $j=0,1,\cdots,N$, the eigenfunction $u_{j}(x)$ associated with the eigenvalue $\lambda_{j}$ can be
normalized and it has {exactly} $j$ simple zeros.
\end{itemize}
\end{theorem}

\section*{Acknowledgments}

The authors sincerely appreciate the referees for their nice comments and suggestions which greatly improve this paper both in mathematics and presentations.

This work is  partially supported by National Key R$\&$D Program of China grant
number 2022YFA1005900.

The authors are partially supported by National Natural Science Foundation (NNSF) of China grant numbers 12071284 and 12161131001. The second author is also partially supported by NNSF of China grant number 11871334, by Innovation Program of Shanghai Municipal Education Commission grant number 2021-01-07-00-02-E00087.

\medskip

%Received September 15, 2004; revised February 2005.

\medskip


\vspace{0.1in}

\begin{thebibliography}{99}

\bibitem{AGJ} J. Alexander, R. Gardner, C.K.R.T. Jones, A topological invariant arising in the stability of travelling waves, \emph{J. Reine Angew. Math.} \textbf{410}  (1990)   167--212.

\bibitem{BCD} R. Bastiaansen, P. Carter, A. Doelman, Stable planar vegetation stripe patterns on sloped terrain in dryland ecosystems, \emph{Nonlinearity} \textbf{32}  (2019)   2759--2814.

\bibitem{Carpenter} G. Carpenter, A geometric approach to singular perturbation problems with applications to nerve impulse equations.  \emph{J. Differential Equations} \textbf{23} (1977) 335--367.

\bibitem{CRS} P. Carter, B. de Rijk,  B. Sandstede, Stability of traveling pulses with oscillatory tails in the FitzHugh-Nagumo system, \emph{J. Nonlinear Sci.} \textbf{26}  (2016)  1369--1444.

\bibitem{CS} P. Carter, B. Sandstede, Fast pulses with oscillatory tails in the FitzHugh-Nagumo system, \emph{SIAM J. Math. Anal.} \textbf{47} (2015) 3393--3441.

\bibitem{CG} C. Conley, R. Gardner, An application of the generalized Morse index to travelling wave solutions of a competitive reaction-diffusion model, \emph{Indiana Univ. Math. J.} \textbf{33} (1984)  319--343.

\bibitem{1965book} W.A Coppel, Stability and Asymptotic Behavior of Differential Equations. Heath, Boston (1965).

\bibitem{1978book} W.A Coppel, Dichotomies in stability theory, Lecture Notes in Mathematics, vol. 629, Springer-Verlag, Berlin Heidelberg (1978).

\bibitem{DGK01} A. Doelman, R.A. Gardner, T.J. Kaper, Large stable pulse solutions in reaction-diffusion equations, \emph{Indiana Univ. Math. J.} \textbf{50} (2001) 443--507.

\bibitem{DGK02} A. Doelman, R.A. Gardner, T.J. Kaper,  A stability index analysis of 1-D patterns of the Gray-Scott model, \emph{Mem. Amer. Math. Soc.} \textbf{155} (2002) xii+64 pp.

\bibitem{DHV} A. Doelman, G. Hek, N. Valkhoff, Stabilization by slow diffusion in a real Ginzburg-Landau system, \emph{J. Nonlinear Sci.} \textbf{14} (2004) 237--278.


\bibitem{Eszter} E.G. Eszter, An Evans function analysis of the stability of periodic travelling wave solutions of the FitzHugh-Nagumo system, Ph.D. thesis, University of Massachusetts (1999).

%\bibitem{Evans 1972} J.W. Evans, Nerve axon equations: II. Stability at rest, \emph{Indiana Univ. Math. J.} \textbf{22} (1972) 75-90.

\bibitem{Evans1972III} J.W. Evans, Nerve axon equations: III. Stability of the nerve impulse, \emph{Indiana Univ. Math. J.} \textbf{22} (1972) 577--593.

\bibitem{Evans1975} J.W. Evans, Nerve axon equations: IV. The stable and the unstable impulse, \emph{Indiana Univ. Math. J.} \textbf{24} (1975) 1169--1190.
%\bibitem{DLL} Z. Du, J. Li, X. Li, The existence of solitary wave solutions of delayed Camassa-Holm equation via a geometric approach, \emph{J. Funct. Anal.}  \textbf{275} (2018) 988-1007.

\bibitem{Fen71} N. Fenichel, Persistence and smoothness of invariant manifolds for ows, \emph{Indiana Univ. Math. J.} \textbf{21}  (1971) 193--226.

\bibitem{Fen74} N. Fenichel, Asymptotic stability with rate conditions, \emph{Indiana Univ. Math. J.} \textbf{23} (1974) 1109--1137.

\bibitem{Fen79} N. Fenichel, Geometric singular perturbation theory for ordinary differential equations, \emph{J. Differential Equations} \textbf{31} (1979) 53--98.


\bibitem{FitzHugh} R. FitzHugh, Impulses and physiological states in theoretical models of nerve membrane, \emph{Biophys. J.}  \textbf{1} (1961) 445--466.

\bibitem{Flores} G. Flores, Stability analysis for the slow travelling pulse of the FitzHugh-Nagumo system, \emph{SIAM J. Math. Anal.}  \textbf{22} (1991) 392--399.


\bibitem{GJ90} R.A. Gardner, C.K.R.T. Jones,  Traveling waves of a perturbed diffusion equation arising in a phase field model, \emph{Indiana Univ. Math. J.} \textbf{39} (1990) 1197--1222.

\bibitem{GJ91} R.A. Gardner, C.K.R.T. Jones, Stability of travelling wave solutions of diffusive predator-prey systems, \emph{Trans. Amer. Math. Soc.} \textbf{327} (1991) 465--524.

\bibitem{Hastings1976} S.P. Hastings, On the existence of homoclinic and periodic orbits for the FitzHugh-Nagumo equations, \emph{Quart. J. Math. Oxford Ser. (2)} \textbf{27} (1976)  123--134.

\bibitem{Hastings1982} S.P. Hastings, Single and multiple pulse waves for the Fitzhugh-Nagumo, \emph{ SIAM J. Appl. Math.} \textbf{42} (1982)  247--260.


\bibitem{HDK} M. Holzer, A. Doelman, T.J. Kaper, Existence and stability of traveling pulses in a Reaction-Diffusion-Mechanics system, \emph{J. Nonlinear Sci.} \textbf{23}  (2013)  129--177.

\bibitem{Jones1984}C.K.R.T. Jones, Stability of the travelling wave solution of the FitzHugh-Nagumo system, \emph{Trans. Amer. Math. Soc.} \textbf{286} (1984) 431--469.

\bibitem{Jones95} C.K.R.T. Jones, Geometrical singular perturbation theory. In: Johnson, R. (ed.) Dynamical Systems, Lecture Notes in Mathematics, vol. 1609, Springer, New York (1995).

\bibitem{2013book} T. Kapitula, K. Promislow, Spectral and Dynamical Stability of Nonlinear Waves, Applied Mathematical Sciences, vol. 185, Springer, New York (2013).

\bibitem{Langer} R. Langer, Existence of homoclinic travelling wave solutions to the FitzHugh-Nagumo equations, Ph.D. Thesis, Northeastern University (1980).

\bibitem{Levinson} N. Levinson, The asymptotic nature of solutions of linear systems of differential equations, \emph{Duke Math. J.} \textbf{15} (1948) 111--126.

\bibitem{Lin} X.B. Lin, Using Melnikov's method to solve Silnikov's problems, \emph{Proc. Roy. Soc. Edinburgh Sect. A} \textbf{116} (1990) 295--325.

\bibitem{Nagumoetal} J. Nagumo, S. Afimoto, S. Yoshizawa, An active pulse transmission line simulating nerve axon, \emph{Proc. IRE} \textbf{50} (1962) 2061--2070.

\bibitem{NP} M.P. Nash, A.V. Panfilov, Electromechanical model of excitable tissue to study reentrant cardiac arrhythmias, \emph{Prog. Biophys. Mol. Biol.} \textbf{85} (2004) 501--522.

\bibitem{NMIF} Y. Nishiura, M. Mimura, H. Ikeda, H. Fujii, Singular limit analysis of stability of traveling wave solutions in bistable reaction-diffusion systems, \emph{SIAM J. Math. Anal.} \textbf{ 21} (1990) 85--122.

\bibitem{NS} Y. Nishiura, H. Suzuki, Higher dimensional SLEP equation and applications to morphological stability in polymer problems, \emph{SIAM J. Math. Anal.} \textbf{36} (2004/05) 916--966.

\bibitem{P1984} K. Palmer, Exponential dichotomies and transversal homoclinic points, \emph{J. Differential Equations} \textbf{55} (1984) 225--256.
%
\bibitem{P1988} K. Palmer,  Exponential dichotomies and Fredholm operators, \emph{Proc. Amer. Math. Soc.} \textbf{104} (1988) 149--156.

\bibitem{PKN} A.V. Panfilov, R.H. Keldermann, M.P. Nash, Self-organized pacemakers in a coupled reaction-diffusion-mechanics system, \emph{Phys. Rev. Lett.} \textbf{95} (2005) 258104.

\bibitem{Sandstede1998} B. Sandstede, Stability of multiple-pulse solutions, \emph{Trans. Amer. Math. Soc.} \textbf{350} (1998) 429--472.

\bibitem{Sandstede} B. Sandstede, Stability of Travelling Waves In: Handbook of Dynamical Systems, vol. 2, pp. 983--1055. North-Holland, Amsterdam (2002)

\bibitem{SZ} J.H. Shen, X. Zhang, Traveling pulses in a coupled FitzHugh-Nagumo equation, \emph{Phys. D} \textbf{418}  (2021)  132848.

%\bibitem{Ye86} Y. Ye et al, Theory of Limit Cycles, Translations of Mathematical Monographs, vol. {\bf 86}, American Math. Soc., Rhode Island (1986).

%\bibitem{Zhang} X. Zhang, Homoclinic, heteroclinic and periodic orbits of singularly perturbed systems, \emph{Science China Math.} \textbf{61}  (2018)  1-18

\bibitem{WZ} Y. Wu, X. Zhao, The existence and stability of travelling waves with transition layers for some singular cross-diffusion systems, \emph{Phys. D} \textbf{200} (2005)  325--358.


\bibitem{Yanagida} E. Yanagida, Stability of fast travelling pulse solutions of the FitzHugh-Nagumo equations. \emph{J. Math. Biol.} \textbf{22} (1985) 81--104.


\end{thebibliography}
\end{document}